\documentclass[graybox]{svmult}

\usepackage{mathptmx}       % selects Times Roman as basic font
\usepackage{eufrak}  
\usepackage{helvet}         % selects Helvetica as sans-serif font
\usepackage{courier}        % selects Courier as typewriter font
\usepackage{type1cm}        % activate if the above 3 fonts are
\usepackage{makeidx}         % allows index generation
\usepackage{graphicx}        % standard LaTeX graphics tool
\usepackage{multicol}        % used for the two-column index
\usepackage[bottom]{footmisc}% places footnotes at page bottom

\usepackage{newtxtext}      
\usepackage{newtxmath} 

\usepackage{tikz}
\usepackage{caption}
\usepackage{mathrsfs}
\usetikzlibrary{arrows}
\usepackage{pgfplots}
\usetikzlibrary{calc} % to do arithmetic with coordinates
\usetikzlibrary{angles,quotes}
\def\R{{\mathbb R}}

\def\Z{{\mathbb Z}}

\def\C{{\mathbb C}}
\def\DD{{\mathbb D}}

\def\H{{\mathbb H}}

\def\C{{\mathbb C}}

\def\H{{\mathbb H}}

\def\cH{{\mathcal H}}
\def\cD{{\mathcal D}}
\def\cB{{\mathcal B}}
\def\cU{{\mathcal U}}

\def\cW{{\mathcal W}}
\def\cA{{\mathcal A}}

\def\cR{{\mathcal R}}

\def\cF{{\mathcal F}}

\def\gg{{\mathfrak g}}
\def\gge{{\mathfrak e}}
\def\gn{\mathfrak n}

\def\gp{{\mathfrak p}}

\def\gk{{\mathfrak k}}
\def\ga{{\mathfrak a}}
\def\gm{{\mathfrak m}}

\def\gz{{\mathfrak z}}

\def\gu{{\mathfrak u}}
\def\Pdr{{\rm P(n,{\mathbb R})}}
\def\Ptwor{{\rm P(2,{\mathbb R})}}

\def\sldr{{\mathfrak {sl}}(d,\R)}
\def\gldr{{\mathfrak {gl}}(d,\R)}
\def\sodr{{\mathfrak {so}}(d,\R)}

\def\Symdr{{\rm Sym}(d)}

\def\Gldr{{\rm GL}(d,\R)}

\def\Sltwor{{\rm SL}(2,\R)}

\def\Sldr{{\rm SL}(d,\R)}

\def\Suone{{\rm SU}(1,1)}

\def\Sodr{{\rm SO}(d,\R)}
\def\Son{{\rm SO}(n)}
\def\Sod{{\rm SO}(d)}
\def\Sotwo{{\rm SO}(2)}

\def\Sodone{{\rm SO}(d,1)}

\def\Pdr{{\rm P}(d,\R)}
\def\Ptwor{{\rm P}(2,\R)}

\def\SPdr{{\rm SP}(d,\R)}

\def\Ad{\mathop{\rm Ad}}
\def\ad{\mathop{\rm ad}}

\def\Exp{\mathop{\rm Exp}}

\def\tr{\mathop{\rm tr}}
\def\id{\mathop{\rm id}}
\def\diag{\mathop{\rm diag}}

\def\E{\mathop{\rm e}}

\newcommand{\Le}{L}
\newcommand{\vertiii}[1]{{\left\vert\kern-0.25ex\left\vert\kern-0.25ex\left\vert #1 
    \right\vert\kern-0.25ex\right\vert\kern-0.25ex\right\vert}}

\def\t{\!\;^t\!}   % transpose %

\makeindex             % used for the subject index
\newcommand{\de}{\mathrm{d}}

\begin{document}

\title*{Unitarization of the Horocyclic Radon Transform on Symmetric Spaces}
% Use \titlerunning{Short Title} for an abbreviated version of
% your contribution title if the original one is too long
\author{Francesca Bartolucci, Filippo De Mari and Matteo Monti}
% Use \authorrunning{Short Title} for an abbreviated version of
% your contribution title if the original one is too long
\institute{Francesca Bartolucci \at Seminar for Applied Mathematics, ETH Zurich, Raemistrasse 101, 8092 Zurich, Switzerland, \email{francesca.bartolucci@sam.math.ethz.ch}
\and Filippo De Mari \at Department of Mathematics \& MaLGa Center,
   University of Genoa, Via Dodecaneso 35, 16146 Genova, Italy, \email{demari@dima.unige.it}
   \and Matteo Monti \at Department of Mathematics \& MaLGa Center,
   University of Genoa, Via Dodecaneso 35, 16146 Genova, Italy, \email{m.monti@dima.unige.it}}
%
% Use the package "url.sty" to avoid
% problems with special characters
% used in your e-mail or web address
%
\maketitle

\abstract*{Each chapter should be preceded by an abstract (10--15 lines long) that summarizes the content. The abstract will appear \textit{online} at \url{www.SpringerLink.com} and be available with unrestricted access. This allows unregistered users to read the abstract as a teaser for the complete chapter. As a general rule the abstracts will not appear in the printed version of your book unless it is the style of your particular book or that of the series to which your book belongs.
Please use the 'starred' version of the new Springer \texttt{abstract} command for typesetting the text of the online abstracts (cf. source file of this chapter template \texttt{abstract}) and include them with the source files of your manuscript. Use the plain \texttt{abstract} command if the abstract is also to appear in the printed version of the book.}

\tableofcontents

\section{Introduction}

The Radon transform has its origin in the problem of recovering a function defined on $\R^d$ from its integrals over hyperplanes. In 1917 Radon proved the reconstruction formula for two and three-dimensional signals. In $\R^3$ it reads 
\begin{equation}\label{jradon1917}
f(x)=-\frac{1}{8\pi^2}\Delta\int_{S^2}\cR f(\theta,x\cdot\theta)\de\theta,
\end{equation}
where $\Delta$ is the Laplacian acting on the variable $x$, $S^2$ is the sphere in $\R^3$ and for every $\theta\in S^2$ and $t\in\R$ we denote with $\mathcal{R}f(\theta,t)$ the integral of $f$ over the hyperplane $x\cdot\theta=t$. Formula \eqref{jradon1917} suggests to define two dual transforms $f\mapsto\cR f$, $g\mapsto\cR^{\#} g$, known as Radon transform and dual Radon transform, or back-projection, respectively. The Radon transform $\cR$ maps a function on $\R^d$ into the set of integrals over all hyperplanes, while the dual Radon transform $\cR^{\#}$ maps a function defined on the set of hyperplanes of $\R^d$ into its integrals over the sheaves of hyperplanes through a point. Formula \eqref{jradon1917} can be rewritten
\begin{equation*}
f=-\frac{1}{2}\Delta\cR^{\#}\cR f
\end{equation*}
and solves the inverse problem of recovering $f$ from the measured datum $\cR f$.
%We refer to $\cite{helgason99,natterer}$ for full details on the subject.

This classical inverse problem is a particular case of the more general issue of recovering an unknown function on a manifold by means of its integrals over a family of submanifolds, already investigated by Gelfand in the 1950's~\cite{gel1}. A natural framework for such general inverse problems was considered by Helgason \cite{gass} and is motivated by the group structure hidden in the polar Radon transform setting \cite{helgason99}, whereby the signals to be analyzed are in $\R^d$.

In the planar case, $\R^2$ and $[0,2\pi)\times\R$,  which parametrizes the set of lines in the plane by polar coordinates, are both transitive spaces of the rigid motions' group. This is $G=\R^2\rtimes K$, with $K=\{R_\phi:\phi\in [0,2\pi)\}$ where \[
R_{\phi}=\left[\begin{matrix}\cos{\phi} & -\sin{\phi} \\ \sin{\phi} & \cos{\phi}\end{matrix}\right].
\]
We write $(b,\phi)\in\R^2\times [0,2\pi)$ for the elements in $G$ and the group law is
\[
(b,\phi)(b',\phi')=(b+R_{\phi}b', \phi+\phi'\ \text{mod}\ 2\pi).
\]
%and the inverse of $(b,\phi,a)$ is given by
% \begin{equation}\label{eq:inverseSIM}
%(b,\phi,a)^{-1}=(-A_a^{-1}R_{\phi}^{-1}b,-\phi\ \text{mod}\ 2\pi, a^{-1}).
%\end{equation}
%A left Haar measure of $SIM(2)$ is 
%\begin{equation}\label{eq:haarSIM}
%\D\mu(b,\phi,a)=a^{-3}\D b\D \phi\D a,
%\end{equation}
%where $\D b$, $\D \phi$ and $\D a$ are the Lebesgue measures on
%  $\R^2$, $[0,2\pi)$ and $\R_{+}$, respectively.
  The group $G$ acts transitively on $\R^2$ by the action 
\[
(b,\phi)[x]=R_{\phi}x+b
\]
and the isotropy at the origin $x_0=(0,0)$ is the Abelian subgroup
\[
K\simeq\{(0,\phi):\phi\in[0,2\pi)\}.
\]
Therefore $\R^2\simeq G/K$ under the canonical isomorphism $gK\mapsto g[x_0]$. The group $G$ is a group of affine transformations of the plane and maps lines into lines. A line in the plane is parametrized by the direction $\t n(\theta)=(\cos(\theta),\sin(\theta))$, where $\theta\in [0,2\pi)$, of its normal and by the coordinate $t$ on the oriented normal line\footnote{ The orientation is such that the coordinate $t$ is $1$ exactly at $(\cos\theta,\sin\theta)$.} which describes its intersection with the given line.  The action of $G$  is then given by 
\[
(b,\phi).(\theta,t)=(\theta+\phi\ \text{mod}\ 2\pi,t+n(\theta)\cdot R_{\phi}^{-1}b)
\]
%or equivalently 
%\[
%(b,\phi,a)^{-1}.(\theta,t)=\left(\theta-\phi\ \text{mod}\ \pi,\frac{t-n(\theta)\cdot b}{a}\right)
%\]
and is easily seen to be transitive. The isotropy at the $y$-axis $\xi_0=(0,0)\in [0,2\pi)\times\R$ is 
\[
H=\{((0,b_2),\phi):b_2\in\R,\phi\in\{0,\pi\}\}.
\]
Thus, $[0,2\pi)\times\R\simeq G/H$ under the canonical isomorphism $gH\mapsto g.\xi_0$. 
%The action $g.\xi_0$ is obtained by acting with $g$ on the points of the $y$-axis by \eqref{eq:actionx}. 
From this group-theoretic point of view, the fact that a point $x\in\R^2$ belongs to the line $(\theta,t)\in[0,2\pi)\times\R$  is equivalent to requiring that the left cosets $x=g_1K$ and $(\theta,t)=g_2H$ intersect. Indeed, $g_1[x_0]$ belongs to the line $g_2.\xi_0$ if and only if there exists $h\in H$ such that $g_1[x_0]=g_2h[x_0]$, so that $g_1(g_2h)^{-1}\in K$ and $g_1K\cap g_2H\ne\emptyset$. This structure illustrates the following general framework introduced by Helgason. 

Consider two $G$-spaces $X$ and $\Xi$, where the actions on $x\in X$ and $\xi\in\Xi$ are
  \[(g,x) \mapsto g[x],\quad  (g,\xi) \mapsto g.\,\xi. \]
Both $X$ and $\Xi$ are assumed to be  transitive spaces,  so that
there exist quasi-invariant measures ${\rm d}x$ and ${\rm d}\xi$.
In Helgason's approach, it is assumed that ${\rm d}x$ and ${\rm d}\xi$ are invariant measures.
Fix $x_0\in X$ and $\xi_0\in\Xi$ and denote by $K$ and $H$ the corresponding stability subgroups, so that $X\simeq G/K$ and $\Xi\simeq G/H$ under the isomorphisms $gK\mapsto g[x_0]$ and $gH\mapsto g.\xi_0$, respectively. The space $X$ is meant to describe the ambient in which the functions to be analysed live, for example the Euclidean plane, or the sphere $S^2$ or the hyperbolic plane $H^2$. The second space $\Xi$   parametrises the set of submanifolds of $X$ over which one wants to integrate functions, for instance lines in the Euclidean plane, great circles in $S^2$, geodesics or horocycles in $H^2$. Motivated by the group structure behind the polar Radon transform, the elements in $\Xi$ can be realized as submanifolds of $X$ introducing the concept of incidence. Two elements $x=g_1K$ and $\xi=g_2H$ are said to be incident if they intersect as cosets in $G$. The concept of incidence translates the fact that a point $x\in X$ belongs to the submanifold parametrized by $\xi\in\Xi$. Any point $\xi\in\Xi$ is realized as a submanifold $\widehat{\xi}\subset X$ by taking all the points $x\in X$ that are incident to $\xi$. Precisely,
\begin{equation}\label{xihat}
\widehat{\xi}=\{x\in X: x\ \text{and}\ \xi\ \text{are incident}\}	\subset X.
\end{equation}
Conversely, one builds the ``sheaf'' of manifolds $\check{x}$ through the point $x\in X$ by taking all the points $\xi\in \Xi$ that are incident to $x$
\begin{equation}\label{xcheck}
\check{x}=\{\xi\in \Xi: \xi\ \text{and}\ x\ \text{are incident}\}\subset \Xi.
\end{equation}
By \eqref{xihat} and \eqref{xcheck} we have that
$$
\widehat{\xi}_0=H[x_0]\subset X,\qquad \check{x}_0=K.\xi_0\subset\Xi.
$$
Both $\check{x}_0$ and $\widehat{\xi}_0$ are transitive spaces and  hence carry quasi-invariant measures. By definition, for any $x=gK$ and $\xi=\gamma H$
\[
\check{x}=g.\check{x}_0\subset\Xi,\qquad \widehat{\xi}=\gamma[\widehat{\xi}_0]\subset X,
\]
which are closed subsets by Lemma 1.1 in \cite{helgason99}. If the maps $\xi\mapsto\widehat{\xi}$ and $x\mapsto\check{x}$ are both injective, then the pair of homogeneous spaces $(X,\Xi)$ is called a dual pair. This assumption is called transversality, see Lemma 1.3 in \cite{helgason99} for an equivalent characterization. The transversality condition avoids
a redundant parametrisation of the submanifolds of $X$. The reader may consult \cite{helgason99} for numerous examples of dual pairs. It is worth observing that the leading example of the polar Radon transform does not satisfy the transversality condition. Indeed, the points $(\theta,t)$ and $(\theta+\pi\ \text{mod}\ 2\pi,-t)$ in $[0,2\pi)\times\R$ both parametrise the line given by the set of points 
\[
\widehat{(\theta,t)}=\widehat{(\theta+\pi\ \text{mod}\ 2\pi,-t)}=\{x\in\R^2: x\cdot n(\theta)=t\}.
\] 
For a deeper study on the injectivity issue, the reader may consider \cite{abddchapter}.

In Helgason's approach the transitive spaces $\check{x}_0$ and $\widehat{\xi}_0$ are supposed to carry $K$-invariant and $H$-invariant measures, respectively, that is 
\[
\int_{\check{x}_0}g(k^{-1}.\xi)\de\mu_0(\xi)=\int_{\check{x}_0}g(\xi)\de\mu_0(\xi),\qquad g\in L^1(\check{x}_0,\de\mu_0),\, k\in K,
\]
\[
\int_{\widehat{\xi}_0}f(h^{-1}[x])\de m_0(x)=\int_{\widehat{\xi}_0}f(x)\de m_0(x),\qquad g\in L^1(\widehat{\xi}_0,\de m_0),\, h\in H.
\]
In order to define the Radon transform and its dual, one needs to introduce measures on $\widehat{\xi}$ and $\check{x}$. This  may be done taking the pushforward of the measure $\de \mu_0$ to $\widehat{\xi}=(gH)^{\hat{}}$ by the map $\widehat{\xi}_0\ni x\mapsto g[x]\in\widehat{\xi}$ and of the measure $\de m_0$ to $\check{x}=(gK)^{\check{}}$ by the map $\check{x}_0\ni\xi\mapsto g.\xi\in\check{x}$, respectively. We denote by $\de\mu_x$ the measure on $\check{x}$ and by $\de m_\xi$ the measure on $\widehat{\xi}$. Since the measures on $\widehat{\xi}_0$ and $\check{x}_0$ are invariant , the measures $\de m_\xi$ and $\de\mu_x$ do not depend on the choice of the representatives of $\xi$ and $x$ and the transversality condition guarantees that they are unique. \begin{definition}

The Radon transform of $f$ is the map $\cR f:\Xi\to\C$ given by
\[
\cR f(\xi)=\int_{\widehat{\xi}}f(x)\de m_\xi(x),
\]
and the dual Radon transform of $g$ is the map  $\cR^{\#} g:X\to\C$ given by 
\[
\cR^{\#} g(x)=\int_{\check{x}}g(\xi)\de\mu_x(\xi),
\]
for any $f$ and $g$ for which the integrals converge. 
\end{definition}
%We may think of $\cR f(\xi)$ as the integral of $f$ over the submanifold parametrized by $\xi\in \Xi$ and of $\cR^{\#} g(x)$ as the integral of $g$ over the sheaf of submanifolds passing through the point $x\in X$.
Observe that, even if the transversality condition is not satisfied for the polar Radon transform, both $\widehat{(\theta,t)}$ and $\widehat{(\theta+\pi\ \text{mod}\ 2\pi,-t)}$ are endowed with the same measure since the arc-length measure is invariant under translations and rotations. For this reason the polar Radon transform satisfies
\[
\cR^{\rm pol} f(\theta,t)=\cR^{\rm pol} f(\theta+\pi\ \text{mod}\ 2\pi,-t).
\]
In this context, the most relevant issue is  to recover $f$ from the values of $\mathcal{R}f$. Another central issue is to prove that the Radon transform, up
to a composition with a suitable pseudo-differential operator, can be
extended to a unitary map $\mathcal{Q}$ from $L^2(X,{\rm d}x)$ to
$L^2(\Xi,{\rm d}\xi)$ intertwining the quasi-regular representations $\pi$ and $\hat\pi$ of $G$  acting on $L^2(X,{\rm d}x)$ and $L^2(\Xi,{\rm d}\xi)$, respectively.

In \cite{acha}, the authors obtain  both an intertwining and a unitarization result for the affine Radon transform. The techniques used in \cite{acha} mimic the approach followed by Helgason to unitarize the polar Radon transform \cite{helgason99}. 

Later, inspired by the results in \cite{acha} a new approach based on representation theory has been taken in order to treat in a general and unified way the problem of unitarizing and inverting the Radon transform \cite{abdd} under the assumption that $\pi$ and $\hat\pi$ are irreducible. The approach taken in \cite{abdd}, \cite{acha} differs from Helgason's  since the assumptions on the measures carried by  $X$ and $\Xi$ and by the submanifolds $\hat\xi\subset X$ are weaker, namely their relative
invariance instead of (proper) invariance. This allows to consider a wider variety of cases of interest in applications, such as the similitude group studied by Murenzi \cite{anmu96}, and the generalized shearlet dilation groups introduced by F\"uhr in \cite{fu98}, \cite{futo16} for the purpose of generalizing the standard shearlet group introduced in
\cite{lawakuwe05}, \cite{dastte10}. It is assumed that there exists a non-trivial $\pi$-invariant subspace $\cA$ of $L^2(X,\de x)$ such that $\mathcal{R}$ is well defined for all $f\in\cA$ and the adjoint of the operator $\mathcal{R}\colon\cA\to L^2(\Xi,\de\xi)$ has non-trivial domain. Then, it is proved that the Radon transform $\mathcal{R}$ is a closable operator from $\cA$ into $L^2(\Xi,{\rm d}\xi)$ and that its closure $\overline{\mathcal{R}}$ is independent of the choice of $\cA$ and is the unique closed extension of $\mathcal{R}$. The main result states that if the quasi regular representations $\pi$ of $G$ on $L^2(X,\de x)$ and $\hat\pi$ of $G$ on $L^2(\Xi,\de\xi)$ are irreducible, then the Radon transform $\mathcal{R}$, up
to a composition with a suitable pseudo-differential operator, can be
extended to a unitary operator $\mathcal{Q}:L^2(X,{\rm d}x)\to L^2(\Xi,{\rm d}\xi)$ which intertwines them, namely
\[
\hat\pi(g)\mathcal{Q}\pi(g)^{-1}=\mathcal{Q},\qquad g\in G.
\]
The proof is based on the extension of Schur's lemma due to Duflo and Moore \cite{dumo76}. 

%%%%%Reconstruction
{A direct consequence of the result above is studied in \cite{abdd}. Adding the hypothesis of square-integrability of $\pi$, the authors derive a new general inversion formula for the Radon transform of the form
\[f=\int_G\chi(g)\langle \mathcal{R}f,\hat{\pi}(g)\Psi \rangle \pi(g)\psi\de g, 
\]
where $\chi$ is a character of $G$ and $\psi\in L^2(X, \de x)$ and $\Psi \in L^2(\Xi, \de\xi)$ are suitable mother wavelets and where  the Haar integral is weakly convergent. Such formula is obtained by the usual reconstruction formula for square-integrable representations and then by applying the unitary operator $\mathcal{Q}$ to both entries of the scalar product $\langle f,\pi(g)\psi\rangle$. We stress that the above formula allows to reconstruct an unknown signal by computing the family of coefficients $\{\langle \mathcal{R}f,\hat{\pi}(g)\Psi \rangle \}_{g\in G}$.

The results achieved in \cite{abdd} and \cite{acha} have posed many interesting mathematical challenges. 
A natural question is to investigate how to generalize these findings to other groups and related representations without the hypothesis of irreducibility, because the techniques used in~\cite{abdd} cannot be transferred directly.

In this direction, we have considered in \cite{bdm} the case of homogeneous trees. Precisely, we construct the unitarization of the horocyclic Radon transform on a homogeneous tree $X$ and we prove that it intertwines the quasi regular representations of the group of isometries of $X$ acting on the space of square-integrable functions on the tree itself and on the space of horocycles, respectively.
Since the quasi regular representation is not irreducible, we adopt a combination of the approach followed by Helgason in the context of symmetric spaces \cite{gass} and the techniques that have been developed in \cite{acha}. The main observation motivating  \cite{bdm}  is that homogeneous trees are the natural discrete counterpart of rank-one symmetric spaces. 

This article is devoted to investigate the unitarization problem in the case when $X$ is a symmetric space and $\Xi$ is the set of horocycles of $X$, which has at large been addressed by Helgason. A remarkable difference from the cases treated in \cite{abdd} is that   the quasi regular representations $\pi$ of the group of isometries of the symmetric space $X$ acting on $L^2(X)$ is not irreducible, nor is it the representation $\hat{\pi}$ on $L^2(\Xi)$. We are well aware that the unitarization problem was already addressed and essentially solved by Helgason in \cite{gass}.  Precisely, he constructs a pseudo-differential operator $\Lambda$ and he proves that the pre-composition with the horocyclic Radon transform yields an isometric operator, see Theorem 3.9 in Chap.~II in \cite{gass}. Here, we prove that the composition $\Lambda\mathcal{R}$ can actually be extended to a unitary operator $\mathcal{Q}:L^2(X, {\rm d}x)\to\Le_\flat^2(\Xi, {\rm d}\xi)$, where ${\rm d}x$ and ${\rm d}\xi$ are the $G$-invariant measures and where $\Le_\flat^2(\Xi, {\rm d}\xi)$ is a closed subspace of $\Le^2(\Xi, {\rm d}\xi)$ which accounts for the Weyl symmetries. Furthermore, we are able to show that $\mathcal{Q}$ intertwines the  quasi regular representations $\pi$ and $\hat{\pi}$.

%%%%%%%Geodesics
This work is focused on the horocyclic Radon transform, but another interesting setting could be obtained by considering geodesics. Such Radon transform is commonly called X-ray transform and has been introduced and inverted  by Helgason on the hyperbolic space $\H^n$, see Theorem 3.12 in Chap.I in \cite{gass}, and on symmetric spaces of the noncompact type by Rouvi\`ere \cite{rouviere}. Although it is not in general true that a horocycle has codimension one in the symmetric space, the horocyclic Radon transform can be seen as the analogue of the Euclidean Radon transform on hyperplanes in $\R^n$, whereas the X-ray transform is the analogue of the Radon on lines in $\R^n$.

The primary reason of the present contribution was to settle the unitarization issue in the setup of noncompact symmetric spaces in all details, in a self-contained and accessible way to the readers that have little experience with the heavy machinery of semisimple groups. We do make use of the basic Lie theoretic notions but avoid as much as possible to make extensive use of the full body of the theory. Rather, we collect all the most relevant results of the theory that may serve as a map.
 
We are not aware of a general statement such as our Theorem~\ref{thm:unitarizationtheorem} in the literature, though it is quite clear to us
that the result comes as no surprise if not for the flexibility of our proof (see once again~\cite{acha}, \cite{bdm}). We also believe that the material presented here is a readable introduction to a subject that may attract the attention of a wide community of young researchers. 

The chapter is organized as it follows. In Sect.~ \ref{sec:prelim} we recall the basic facts of the analysis on semisimple Lie groups and we introduce the notation used throughout in the geometric analysis on noncompact Riemannian symmetric spaces. 
In Sect.~ \ref{sec:symsp} we present a brief overview of the general theory of symmetric spaces enriched with the examples of the Euclidean space, the sphere, the upper half plane, the unit disk and the positive definite symmetric matrices. 
%Such spaces are homogeneous spaces of semisimple Lie groups $G$ on a maximal compact subgroup $K$. In this context, we can consider a boundary parametrized by a homogenous space of $K$. It is showed how the decompositions of $G$ lead to decompositions of its Haar measure and the associated integral formulae. We endow a symmetric space with a $G$-invariant measure. Furthermore, it is highlighted the dependence of the whole structure by the reference point and it is exhibited how such dependence changes by changing the reference point. Such fact will play a crucial role in our treatment.
Of particular interest for our purposes are $\S$~\ref{sec:boundary}, \ref{CRP} and \ref{sec:horo}. In $\S$~\ref{sec:boundary} we present the notion of boundary of a symmetric space and in $\S$~\ref{CRP} we show the infinitely many ways to represent it changing the reference point in the symmetric space. Finally, in $\S$~\ref{sec:horo} we define the family of horocycles and we prove some technical results needed in Sects.~\ref{sec:analysis} and \ref{sec:unit}. 
In Sect. \ref{sec:analysis} we collect the analytic ingredients that come into play. We endow the symmetric space, its boundary and the family of horocycles with invariant measures. We introduce the Helgason-Fourier transform and its main features. Then, we study the horocyclic Radon transform and we discuss its relation with the Helgason-Fourier transform.  
Finally, in Sect.~\ref{sec:unit} we prove the unitarization result for the horocyclic Radon transform.

\section{Preliminaries}\label{sec:prelim}
The purpose of the introductory section is to recall the basic facts of the analysis on semisimple Lie groups and to establish the notation used throughout in the geometric analysis on noncompact Riemannian symmetric spaces.
For a concise and effective exposition, see~\cite{gga}. Classical references with a wider scope are
\cite{dls}, \cite{gass} and \cite{knapp}. For a detailed introduction to differential geometry and Lie groups, we refer to~\cite{war}.

%In the theory of symmetric spaces, three decompositions play a crucial role: the Cartan, Iwasawa and Bruhat decompositions. The latter, however, will not be relevant to usand we therefore shall not recall it.  At the end of the section we include some explicit calculations concerning $\Sltwor$  that serve as good examples.
%We shall often use the word ``smooth'' in place of ``$C^\infty$''.

%Let $G$ be a noncompact connected semisimple Lie group with finite center and $K$ a maximal compact subgroup. Its Lie algebra $\gg$ has a Cartan involution, that is an involution $\theta$ such that the symmetric bilinear form
%$
%B_\theta(X,Y)=-B(X,\theta Y)
%$
%is positive definite. The involution gives rise to the vector space direct sum
%$\gg=\gk+\gp$
%where $\gk$ and $\gp$ are the 
%$+1$ and  $-1$ eigenspaces of $\gg$ relative to $\theta$, respectively. 
%%%%%%%%%%%%%%%%%%%%%%%%%%%%%%%%%%%%%%%%%%%%%%%%%
%\begin{definition}
%	A decomposition $\gg$ in a vector space direct sum $\gg=\gk+\gp$
%	which satisfies
%	\eqref{plusminus} and \eqref{definitezze}  is called a  {\it
%		Cartan decomposition} of $\gg$. 
%\end{definition}

\vskip0.5truecm
A \emph{Lie algebra}\index{Lie algebra} $\gg$ is {\it simple}\index{simple!Lie algebra} if it is not Abelian and contains no proper Abelian ideals. A {\it semisimple}\index{semisimple!Lie algebra} Lie algebra is then the Lie algebra direct sum of (all) its simple ideals. Cartan proved that on every semisimple Lie algebra $\gg$ there exists a \emph{Cartan involution}~$\theta$\index{Cartan!involution}, namely an involution such that the symmetric bilinear form $B_\theta(X,Y)=-B(X,\theta Y)$ is positive definite, where $B$ is the usual Killing form defined by $B(X,Y)=\tr(\ad X\circ\ad Y)$. Such an involution gives rise to a \emph{Cartan decomposition}\index{Cartan!decomposition} of the Lie algebra, namely a vector space direct sum  $\gg=\gk+\gp$, where $\gk$ and $\gp$ are the 
$+1$ and  $-1$ eigenspaces of $\gg$ relative to $\theta$, respectively.

Fix a maximal Abelian subspace $\ga$ of $\gp$.  
%Thus, $\ga$ is a vector subspace of $\gp$ such that $[\ga,\ga]=\{0\}$ and is maximal with this property. It is an easy exercise to check that
%\[
%(\ad X)^*=-\ad(\theta X)
%\]
%where the adjoint $(\cdot)^*$  is taken w.r.t the inner product $B_\theta$, so that the elements of $\ad(\gp)$ are self-adjoint. This entails that 
The set $\{\ad H:H\in \ga\}$ is a commuting family of self-adjoint linear maps. Therefore,  $\gg$  is the $B_\theta$-orthogonal direct sum  of their joint eigenspaces, all of the eigenvalues of which are real and depend linearly on $H$. For any fixed $\alpha\in\ga^*$, the linear dual of $\ga$, we write
\[
\gg_\alpha=\{X\in\gg: (\ad H)X=\alpha(H)X
\hbox{ for all }H\in\ga\}
\]
and we say that  $\alpha\neq0$ is a  {\it restricted root}, or simply a {\it  root}\index{root} of the pair
$(\gg,\ga)$, whenever  $\gg_\alpha\not=\{0\}$. 
The set of restricted roots is $\Sigma$ and the spaces
$\gg_\alpha$ with $\alpha\in\Sigma$  are called {\it (restricted) root spaces}\index{space!root}. 

An element $H\in\ga$ is called {\it regular}\index{regular} if $\alpha(H)\neq0$ for all $\alpha\in\Sigma$, otherwise it is {\it singular}\index{singular}. The set $\ga'$ of regular elements is the complement in $\ga$ of finitely many hyperplanes and its connected components are called the {\it Weyl chambers}\index{chamber!Weyl}.

We fix a Weyl chamber $\ga^+\subset\ga$ and we declare a root $\alpha$ to be {\it positive}\index{positive!root} if it has positive values on $\ga^+$. A root is {\it simple}\index{simple!root} if it cannot be written as the sum of positive roots. The set $\Delta$ of simple roots turns out to be a basis of $\ga^*$. Thus, there are exactly $\ell=\dim\ga$ simple roots. This number is an important invariant and is called the {\it real rank}\index{real rank} of~$\gg$.
We order the elements in  $\ga^*$, hence the roots in $\Sigma$,  {\it lexicographically}
with respect to an ordering  $\delta_1,\dots,\delta_\ell$ of the simple roots. This means that $\lambda=\sum a_j\delta_j$ is positive (written $\lambda>0$) if the first non-zero coefficient $a_k$ is positive.
Together with $\gg$, $\theta$ and $\ga$ we assume that an ordering ``$>$'' has been fixed on $\ga^*$ by choosing a labeling of the simple roots relative to a fixed Weyl chamber $\ga^+$. We consequently denote by $\Sigma^+$ and $\Sigma^-$ the positive and negative roots, respectively. Clearly, $\Sigma=\Sigma^+\cup\Sigma^-$, a disjoint union. 

If $G$ is a \emph{Lie group}\index{Lie!group}, then it is said to be \emph{semisimple}\index{semisimple!group} if such is its Lie algebra. Furthermore, for any Cartan involution $\theta$ on its Lie algebra $\gg$ there exists an automorphism $\Theta$ of $G$
such that  $d\Theta=\theta$ and $\Theta^2=\id$. 
\begin{theorem}[The \emph{Iwasawa decomposition}\index{Iwasawa!decomposition}]\label{iwa}
	\label{ID}  Let $G$ be a connected semisimple Lie group, $\gg=\gk+\gp$ be a Cartan decomposition of its Lie algebra and fix a maximal Abelian subspace $\ga$ of $\gp$ and an ordering on $\ga^*$. The vector space direct sum 
	\begin{equation}
	\gn=\sum_{\alpha\in\Sigma^+}\gg_\alpha
	\label{iwaN}
	\end{equation}
	is a nilpotent Lie algebra and $\gg$ decomposes as the vector space direct sum
	\begin{equation*}
	\gg=\gk+\ga+\gn.
	%\label{iwasawALG}
	\end{equation*}
	%	Furthermore, $\ga+\gn$ is solvable and $[\ga+\gn,\ga+\gn]=\gn$.
	%	
	%	 Let $G$ be a connceted semisimple Lie group and let 
	%	$\gg=\gk+\ga+\gn$ be an Iwasawa decomposition of its Lie algebra. 
	Furthermore, let $K$, $A$ and $N$ be the connected subgroups of $G$ whose Lie algebras are $\gk$, $\ga$ and $\gn$, respectively. The multiplication map  $K\times A\times N\to G$ given by $(k,a,n)\mapsto
	kan$ is a diffeomorphism. The groups $A$ and $N$ are simply connected and $AN$ is solvable.
\end{theorem}
%The interpretation of the above theorem is that any element in $G$ can be written uniquely as a product $g=kan$ with $k\in K$, $a\in A$ and $n\in N$. These groups are cllaed {\it Iwasawa subgroups} of $G$, as any of their conjugates. The decomposition  is normally expressed in the short form $G=KAN$. Actually, the result entails three similar decompositions, that is  $G=KNA$,  $G=ANK$ and $G=NAK$,
%where it is to be understood that each element may be written in a unique way as a product of factors in the three Iwasawa subgroups in any of the indicated orders. The groups $K$, $A$ and $N$ are always the same but the factors of each element are not.

Observe that $AN$ is in fact a semidirect product. Indeed, $A$ acts on $N$  by conjugation, as is most rapidly seen by observing that $\Ad a(X)\in \gg_\alpha$ if $X\in\gg_\alpha$ for any root $\alpha\in\Sigma$ and for all $a\in A$.
	Indeed, for any $H\in\ga$, since $\ga$ is Abelian, one has 
\[
[H,\Ad a(X)]=\Ad a\left([\Ad a^{-1}(H),X]\right)=\Ad a\left([H,X]\right)=\alpha(H)\Ad a(X).
\]
	Therefore  $\Ad a$ preserves root spaces and in particular it preserves $\gn$. Thus $A$ acts on $\gn$ via the adjoint action and, passing to exponentials, it acts on $N$ by conjugation. This is tantamount to saying that $A$ normalizes $N$ inside $G$. Hence $NA=AN$ is the semidirect product $N\rtimes A$.

Let $M$ and $M'$ denote the \emph{centralizer}\index{centralizer} and \emph{normalizer}\index{normalizer} of $\ga$ in $K$, respectively. This means  that
\begin{align*}
	M&=\Bigl\{m\in K:\Ad m(H)=H\text{ for all }H\in\ga\Bigr\}\\
	M'&=\Bigl\{w\in K:\Ad w(H)\in\ga\text{ for all }H\in\ga\Bigr\}.
\end{align*}
Passing to exponentials, it follows that 
if $m\in M$, then $mam^{-1}=a$ for all $a\in A$  and 
if $w\in M'$, then $waw^{-1}\in A$ for all $a\in A$. 
The quotient group $W=M'/M$ is called the {\it Weyl group}\index{group!Weyl} of $(G,K)$. The compact Lie groups $M$ and $M'$ have the same Lie algebra, namely $\gm$, so that $W$ is in fact a finite group. The Weyl group $W$ acts on $\Sigma$ by 
\begin{equation}\label{WonSigma}
(w\cdot\alpha)(H)=\alpha(\Ad w^{-1}H),
\qquad
H\in\ga.
\end{equation}
The very same formula defines an action on the whole dual space $\ga^*$. It is worth observing that the action of $W$ on $\ga^*$ maps Weyl chambers in Weyl chambers in a free and transitive way (see \textcolor{red}{$\ldots$}). So that, the cardinality of the Weyl chambers coincides with $|W|$.
For any $\alpha\in\Sigma$, the vector space dimension of $\gg_\alpha$ is called the {\it multiplicity}\index{multiplicity!root} of $\alpha$ and is usually denoted $m_\alpha$.   The following element of $\ga^*$ plays a crucial role in the theory:
\begin{equation}\label{rho}
\rho=\frac{1}{2}\sum_{\alpha\in\Sigma^+} m_\alpha\alpha.
\end{equation}
This linear functional on $\ga$  naturally appears in relation with the semidirect product structure of the Iwasawa group $AN$, see~\eqref{modAN}. 

\vskip0.2truecm
{\bf Example: the decomposition of $\Sldr$.} We consider the Lie algebra $\gg=\sldr$ of $G=\Sldr$, namely
\[\sldr=\{X\in\gldr:\tr X=0 \}. \]
The Cartan decomposition associated to  the standard involution $
\theta(X)=-\t X$ reads
\[
\sldr=\sodr+{\rm Sym}_0(d),
\]
where $\gp={\rm Sym}_0(d)$ is the space of $d\times d$ symmetric and traceless real matrices. The
Cartan involution $\Theta$ for $\Sldr$ is then
\begin{equation*}
%\label{realTheta}
\Theta g=\t g^{-1}
\end{equation*}
as for all matrix groups with real entries. Hence $K=\Sod$, a maximal compact subgroup of $\Sldr$.
The diffeomorphism $(k,X)\mapsto k\exp X$ of $\Sod\times{\rm Sym}_0(d)\to G$ is just the  classical polar decomposition. The center of $\Sldr$ is the identity matrix if $d$ is odd and $\{\pm\id\}$ if $d$ is even.
The natural maximal Abelian subspace of 
${\rm Sym}_0(d)$ is the $(d-1)$-dimensional vector space consisting of the diagonal matrices
${\rm diag}(a_1,\dots,a_d)$ with $a_1+\dots+a_d=0$.  Thus, the real rank of $\sldr$ is $d-1$. Let $E_{ij}$ denote  the matrix whose only non-zero entry is  $1$  at position $(i,j)$. Then,  for
$H={\rm diag}(a_1,\dots,a_d)$  and  $i\neq j$ 
\[
[H,E_{ij}]=(a_i-a_j)E_{ij}
\]
and in fact $E_{ij}$ spans a root space provided  that $i\neq j$. It is customary to introduce the linear functionals  $e_k(\cdot)$ on $\ga$, with $1\leq k\leq d$,  via $e_k({\rm diag}(a_1,\dots,a_d))=a_k$. Thus,  for $i\neq j$ the (restricted) root $\alpha_{ij}=e_i-e_j$ acts on $H={\rm diag}(a_1,\dots,a_d)$ by
\[
\alpha_{ij}(H)=a_i-a_j.
\]
and we write in simplified form $\gg_{ij}$ in place of $\gg_{\alpha_{ij}}$ for the root space
\[
\gg_{ij}={\rm sp}\{E_{ij}\},\qquad i\neq j.
\]
For $i<j$ the matrix $E_{ij}$ is upper triangular, and for $i>j$ it is lower triangular.  A natural choice of 
Weyl chamber is 
\[
\ga^+=\Bigl\{{\rm diag}(a_1,\dots,a_d):a_1>a_2>\dots>a_d\Bigr\}.
\]
It is immediate to  check that for $j=1,\dots,d-1$ the roots $\delta_j=e_j-e_{j+1}$ are the simple ones and that  the set of positive roots is 
\[
\Sigma^+=\{\alpha_{ij}:i<j\}.
\]
It follows  that the nilpotent Iwasawa Lie algebra $\gn$ defined in~\eqref{iwaN} is just the Lie algebra of strictly upper triangular matrices.
 Notice that  $\gg_{0}=\ga$, that is, $\gm=\{0\}$
and  that $\dim\gg_\alpha=1$ for every restricted root $\alpha\in\Sigma$. Hence the functional $\rho$ has the form
\[\rho(H)=\frac{1}{2}\sum_{i<j}\alpha_{ij}(H)=\frac{1}{2}\sum_{i<j}(a_i-a_j)=\sum_{j=1}^d(\frac{d+1}{2}-j)a_j.\]

Let $A$ be the group of diagonal matrices with positive entries and determinant $1$, namely
\[
{\rm diag}({\E}^{a_1},\dots,{\E}^{a_d}),
\qquad
a_1+\dots+a_d=0,
\]
and let $N$ be the group of unipotent upper triangular matrices, namely those of the form
\begin{equation*}
%\label{UUT}
\bmatrix
1      &a_{12}&\dots     &&&\dots&a_{1,d} \\
0      &1             &\ddots &     &&&\vdots          \\
\vdots      &\ddots        &\ddots  &&& \ddots    &\vdots          \\
\vdots &       &\ddots&    &&        1&a_{d-1,d}\\
0      &\dots        &\dots&&&0           &1
\endbmatrix.
\end{equation*}
Then $\ga$ and $\gn$ are the Lie algebra of $A$ and $N$, respectively. Hence $\Sldr\simeq KAN$ by the Iwasawa decomposition.
\section{Symmetric Spaces}\label{sec:symsp}
Symmetric spaces are very special kinds of homogeneous spaces.  The reader is assumed to be familiar with basic Differential Geometry and in particular with the main results on group actions and homogeneous spaces. The natural reference for the material in this section is the celebrated monography~\cite{dls} by Helgason, of which this is a synthesis with examples. Other sources are for example \cite{Iozzi}, \cite{Wolf}.%{\bf Iozzi, J. Wolf,}

We very briefly recall the basic facts that we shall use throughout. A homogeneous space $X$ is a transitive $G$-space. Saying that $X$ is a $G$\emph{-space}\index{$G$-space} means that we are given a continuous map
$G\times X\to X$, written $(g,x)\mapsto gx$ and called an action of $G$ on $X$, which satisfies
\begin{enumerate}
\item[(i)] $x\mapsto gx$ is a homeomorphism of $X$ for each $g\in G$,
\item[(ii)]  $g(hx)=(gh)x$ for all $g,h\in G$ and $x\in X$.
\end{enumerate}
The $G$-space $X$  is called \emph{transitive}\index{transitive!$G$-space}
if for every $x,y\in X$ there exists $g\in G$  such that $gx=y$.
%From now on, we shall reserve the letter $x$ for the elements of the homogeneous space and $g$ for the elements of the group. 
In this case $X$ is indentified with $G/H$ through the action of $G$, where $H$ is the isotropy subgroup at some point $x_0\in X$, namely
\[
H=\bigl\{g\in G:gx_0=x_0\bigr\}.
\]
This identification depends on the choice of the reference point $x_0\in X$ and is given by the bijection
\[
G/H\to X,
\qquad
gH\mapsto gx_0.
\]
If we choose a different reference point $x_0'=g_0x_0$ for some $g_0\in G$, it is sufficient to replace $H$ with $H'=g_0 Hg_0^{-1}$. The map $g\mapsto g_0 gg_0^{-1}$ induces a $G$-equivariant homeomorphism between $G/H$ and $G/H'$. If the topology on $G/H$ is the quotient topology then the identification map is actually a homeomorphism. 

In the present contribution, we  often consider different $G$-spaces of the same group. For clarity, we shall thus adopt notational variations to distinguish among different actions, such as $g[x]$ or $g.x$ or $g\cdot x$ or $g\langle x\rangle$ and so forth.

\subsection{Riemannian Globally Symmetric Spaces}\label{RGSS}
Let $\mathcal{M}$ be a Riemannian manifold and let $I(\mathcal{M})$ denote the group of isometries of~$\mathcal{M}$. We shall endow $I(\mathcal{M})$ with the compact-open topology, the smallest topology in which all the sets
\[
W(C,U)=\left\{g\in I(\mathcal{M}): g(C)\subset U\right\}
\]
are open, where $C$ varies in the compacta of $\mathcal{M}$ and $U$ in the open sets.
%%%%%%%%%%%%%%%%%%%%%%%%%%%%%%%%%%%%%%%%%%%%%%%%
\begin{theorem} [Theorem~2.5, Chap.~IV,\cite{dls}] Let $\mathcal{M}$ be a Riemannian manifold.
\begin{enumerate}
\item[(i)] The group of isometries $I(\mathcal{M})$  with the compact-open topology is a locally compact topological group acting on $\mathcal{M}$.
\item[(ii)] The isotropy subgroup of $I(\mathcal{M})$ at any point of $\mathcal{M}$ is compact.
\end{enumerate}
\end{theorem}
%%%%%%%%%%%%%%%%%%%%%%%%%%%%%%%%%%%%%%%%%%%%%%%%
\begin{definition} The Riemannian manifold $\mathcal{M}$ is a Riemannian globally symmetric space if each $p\in \mathcal{M}$ is an isolated fixed point of an isometry $\sigma_p$ of $\mathcal{M}$ that is  involutive ($\sigma_p^2=\id$).
\end{definition}
It may be shown that each $\sigma_p$ is in this case unique and that  there exists a neighborhood $N_p$ of $p$  in which $\sigma_p$ is the geodesic symmetry. This means that if $q\in N_p$ and $\gamma(t)$ is the  geodesic such that $\gamma(0)=p$ and $\gamma(1)=q$, then $\sigma_p(q)=\gamma(-1)$.
\vskip0.2truecm
%%%%%%%%%%%%%%%%%%%%%%%%%%%%%%%%%%%%%%%%%%%%%%%%
{\bf Euclidean space. } Let $\mathcal{M}=\R^n$ and fix $p\in\R^n$. The globally defined map $\sigma_p(x)=2p-x$ is clearly involutive and isometric with respect to the the Euclidean distance  because $\|\sigma_p(x)-\sigma_p(y)\|=\|y-x\|$. Further, $\sigma_p(x)=x$ if and only if $x=p$, so $p$ is an isolated fixed point.

{\bf The sphere. } Let $\mathcal{M}=S^{n-1}$ and consider the map defined on $\R^n$ by  $x\mapsto\Omega x$ where
\[
\Omega=\begin{bmatrix}1&\\&-{\rm I}_n\end{bmatrix}.
\]
Evidently, it  leaves the unit sphere invariant
and  is an isometry with respect to the natural Riemannian structure on it. It fixes the north pole $e_0=(1,0,\dots,0)$. Next choose $p\in \mathcal{M}$ and take $R\in\Son$ such that $p=Re_0$. Then $\sigma_p=R\Omega R^{-1}$ is the required involutive isometry, as the reader is invited to check.
\vskip0.2truecm
%%%%%%%%%%%%%%%%%%%%%%%%%%%%%%%%%%%%%%%%%%%%%%%%

{\bf The upper half plane.} Let $\mathcal{M}$ denote the upper half plane, which we  think of as one of the  natural models of the 2-dimensional hyperbolic space. We realize it as the complex numbers with positive imaginary part.  The Riemannian structure on $\mathcal{M}$ is given by the inner product 
\begin{equation*}
\langle u, v\rangle_z=\frac{(u,v)}{4y^2}
\end{equation*}
where  $u,v\in T_z( \mathcal{M})$ are tangent vectors at $z=x+iy\in \mathcal{M}$. It is important to observe that $G=\Sltwor$ acts transitively on $ \mathcal{M}$ by means of the {\it M\"obius action}\index{M\"obius action}, namely
\begin{equation}
\label{moebius}
g[z]=\begin{bmatrix}a&b\\c&d\end{bmatrix}[z]=\frac{az+b}{cz+d}.
\end{equation}
The  imaginary part of $g[z]$ is positive if such is that of $z$, so that~\eqref{moebius} is indeed an action.
To show transitivity, we fix $p=b+ia\in \mathcal{M}$ with $a>0$ and consider 
\[
g_p=\begin{bmatrix}1&b\\0&1\end{bmatrix}\begin{bmatrix}\sqrt{a}&0\\0&1/\sqrt{a}\end{bmatrix}
=\begin{bmatrix}\sqrt{a}&b/\sqrt{a}\\0&1/\sqrt{a}\end{bmatrix},
\]
an element of the  Iwasawa  subgroup $NA$ of $\Sltwor$. 
It is  immediate to check that $g_p[i]=p$ and that the isotropy group at $i$ is $K=\Sotwo$, so that $ \mathcal{M}\simeq G/K$.

As for the isometric involutions, consider first the M\"obius action induced by
\[
J=\begin{bmatrix}0&1\\-1&0\end{bmatrix},
\]
which is the map $z\mapsto -1/z$, namely $x+iy\mapsto(-x+iy)/(x^2+y^2)$, and may also be described  in polar coordinates by
\[
\rho(\cos\theta+i\sin\theta)\mapsto\frac{1}{\rho^2}(-\cos\theta+i\sin\theta).
\]
This fixes only $i$ (for $\rho=1$ and $\theta=\pi/2$) and is thus  a global involution of which $i$ is an isolated fixed point. A global involution fixing only the point $p$ is given by the M\"obius action of the $\Sltwor$ element $g_pJg_p^{-1}$.  

Of course, it needs to be to seen that these maps are indeed isometries relative to the hyperbolic distance.  To this end, observe that any differentiable path $\gamma\colon[a,b]\to  \mathcal{M}$, with $\gamma(t)=x(t)+iy(t)$  has length
\begin{equation*}
%\label{lengthUP}
L(\gamma)
=\int_a^b\langle\dot\gamma(t),\dot\gamma(t)\rangle^{1/2}\,\de t
=\frac{1}{2}\int_a^b\frac{\sqrt{\dot x^2(t)+\dot y^2(t)}}{y(t)}\,\de t.
\end{equation*}
It is then very easy to check that $L(g[\gamma])=L(\gamma)$ if $g$ is either $J$ or any of the following
\[
\begin{bmatrix}\E^s&0\\0&\E^{-s}\end{bmatrix}\in A,
\qquad
\begin{bmatrix}1&t\\0&1\end{bmatrix}\in N.
\]
We now show  that these are enough. Indeed, any lower triangular unipotent matrix in $G$ is of the form
$JnJ^{-1}$ for some $n\in N$. Next, any rotation in $\Sotwo$ with $\cos\theta\neq0$ can be written
\[
\begin{bmatrix}\cos\theta&\sin\theta\\-\sin\theta&\cos\theta\end{bmatrix}
=
\begin{bmatrix}1 &\tan\theta \\
0&1
\end{bmatrix}
\begin{bmatrix} 1/\cos\theta&0 \\
0&\cos\theta
\end{bmatrix}
\begin{bmatrix} 1&0 \\
-\tan\theta&1
\end{bmatrix}.
\]
The rotations with $\cos\theta=0$ are of the form $\pm J$, and we conclude that
any element in $K$ is a finite product  of elements\footnote{This argument is nothing else but  the Bruhat decomposition of $\Sltwor$.} chosen in $\{\pm J\}\cup A\cup N$. By the Iwasawa decomposition we conlcude that in fact $L(g[\gamma])=L(\gamma)$ for any $g\in G$. 
This entails that  $\Sltwor$ acts by isometries on $ \mathcal{M}$. It  is worth mentioning that the isometry group of the upper half plane is generated by $\Sltwor$ and by the map $z\mapsto1/\overline{z}$.  
 \vskip0.2truecm

%%%%%%%%%%%%%%%%%%%%%%%%%%%%%%%%%%%%%%%%%%%%%%%%
{\bf The unit disk}. A second natural model of the 2-dimensional hyperbolic space is the unit disk $ \mathcal{M}=\{z\in\C:|z|<1\}$, later denoted $\DD$. This is the Riemannian manifold with inner product
\begin{equation*}
\langle u, v\rangle_z=\frac{(u,v)}{(1-|z|^2)^2}
\end{equation*}
where  $u,v\in T_z( \mathcal{M})$ are tangent vectors at $z\in  \mathcal{M}$. The group 
\[G=\Suone:=\Bigl\{\begin{bmatrix}a&b\\\bar{b}&\bar{a}\end{bmatrix}:a,b\in\C,\;|a|^2-|b|^2=1\Bigr\}\] 
acts on $ \mathcal{M}$ by the very same M\"obius action as given by~\eqref{moebius}. Following similar reasoning as above, we can prove that the action is transitive using the $NA$ action on the point $0\in  \mathcal{M}$, where the  Iwasawa components of $G$ are  obtained from that of $\Sltwor$  by conjugating within ${\rm SL}(2,\C)$ first with a $\pi/4$-rotation and then with $\Lambda^{-1}$, where
\[\Lambda=\frac{1}{\sqrt2}\begin{bmatrix}1&i\\i&1\end{bmatrix}.\] 
The Iwasawa subgroups are explicitly given by 
\begin{align*}
K&=\Bigl\{\begin{bmatrix}\E^{i\theta}&0\\0&\E^{-i\theta}\end{bmatrix}   :\theta \in [0,2\pi)\Bigr\},\\
A&=\Bigl\{\begin{bmatrix}\cosh t&\sinh t\\\sinh t&\cosh t\end{bmatrix}    : t\in\R\Bigr\},%\label{IWASuone}
\\
N&=\Bigl\{\begin{bmatrix}1+is&-is\\is&1-is\end{bmatrix}\   :s\in\R\Bigr\}.
\end{align*}
Of course the isotropy at $o\in  \mathcal{M}$ is $K$ and $ \mathcal{M}\simeq G/K$. 
The reader is invited to write the isometric involutions that prove $ \mathcal{M}$ to be a symmetric space. We content ourselves with remarking that the {\it Cayley transform}\index{Cayley!transform} $c\colon  \mathcal{M}\to\C$
\begin{equation*}
%\label{cayley}
c(z)=i\frac{z+i}{z-i}
\end{equation*}
is an isometry of the unit disc onto the upper half plane which commutes with the M\"obius actions.
\vskip0.2truecm

%%%%%%%%%%%%%%%%%%%%%%%%%%%%%%%%%%%%%%%%%%%%%%%%
{\bf The positive definite symmetric matrices.}
The example of the upper half plane can be generalized in higher dimensions. We have already seen that there exists a diffeomorphism between the upper half plane and $G/K$ where $G=\Sltwor$ and $K=\Sotwo$. We are going to investigate the case where $G=\Sldr$, $d\geq 2$. We denote by $(\,\cdot\,,\,\cdot\,)$ the usual scalar product in $\R^d$ and we put
\begin{align*}
%\Symnr&:=\{g\in\Mnr :  \t g=g \},\\
\Pdr&:=\{p\in\Symdr : (v,pv)>0 \text{ for every }v\in\R^d \},
\end{align*}
the set of $d\times d$ positive-definite symmetric matrices. Observe that $\Pdr$ is an open subset of $\Symdr$ and so it is naturally a smooth manifold. Its dimension   is 
\[
m:=\dim(\Pdr)=\frac{d(d+1)}{2}.
\]
We show that  $\Pdr\subseteq\R^m$ is the interior of a convex cone. Let $p,\,q\in\Pdr$ and $t>0$, 
then $tp\in\Pdr$, $(1-t)q\in\Pdr$  and also
\[tp+(1-t)q\in\Pdr,\] 
provided that $0\leq t\leq 1$. The boundary of $\Pdr$ is the set of all singular positive semidefinite matrices. It is easy to see that $\Pdr$ is a foliated manifold in which each leaf is the preimage of a positive number through the determinant mapping. The preimage of ${1}$ under the determinant mapping is denoted by
\[\SPdr:=\Pdr\cap\Sldr.\]

% In fact, if $X_1=\left[\begin{matrix}
% 0 & 1\\-1 & 0
% \end{matrix} \right]$, then $gK\mapsto\Ad_G(g)X_1$ is a bijection of $X=G/K$ onto $\{X\in \gg\colon B(X,X)=B(X_1,X_1)=-8\}$ 
\color{black}
\begin{figure}[h]
	\centering
	\begin{tikzpicture}[xscale=2]
		\draw[->] (0,0)--(2.5,0);
		\draw[->] (0,0)--(0,5);
		\draw[->] (0,0)--(-0.5,-1);
			\shade[right color=white,left color=black,opacity=0.1]
			(-2,4.02)  arc (180:360:{2} and 0.4) --(2,4) --(0,0) -- cycle
			;
			\shade[right color=white,left color=black,opacity=0.2]
			(-1.937,4.02)  arc (180:360:{1.937} and 0.37) --  plot[samples=100,variable=\s,domain=75.5:-75.5]({0.5*tan(\s)},{sec(\s)}) -- cycle
			;
			\shade[left color=white,right color=black,opacity=0.2]
			(0,4.02) circle ({1.937} and 0.37);
%			\draw[fill=bla,opacity=0.2]
%			\draw(-2,4) -- (0,0) -- (2,4);
%			\draw
%			(0,4.02) circle ({2} and 0.4);
			\draw%[thick]
			(-1.99,3.98) -- (0,0) -- (1.99,3.98);
			\draw%[thick]
			(0,4.02) circle ({2} and 0.4);
			\draw[thick,samples=100,variable=\s,domain=-75.5:75.5] plot({0.5*tan(\s)},{sec(\s)});
			\draw[samples=100,variable=\s,domain=-82.7:82.7,dashed] plot({0.25*tan(\s)},{0.5*sec(\s)});
			\draw[samples=100,variable=\s,domain=-66.2:66.2,dashed] plot({0.8*tan(\s)},{1.6*sec(\s)});
			\draw[dashed]
			(0,4.02) circle ({1.837} and 0.35);
			\draw[thick]
			(0,4.02) circle ({1.937} and 0.37);
			\node at (0.5,3) {$\SPdr$};
			\node at (1.1,1) {$\Pdr$};
	\end{tikzpicture}
	\caption{ The foliation of the cone $\Ptwor$ consists of the connected components of the hyperboloids of two sheets each of which is the preimage under the determinant mapping of a positive number.}
	\label{fig:cono}
\end{figure}
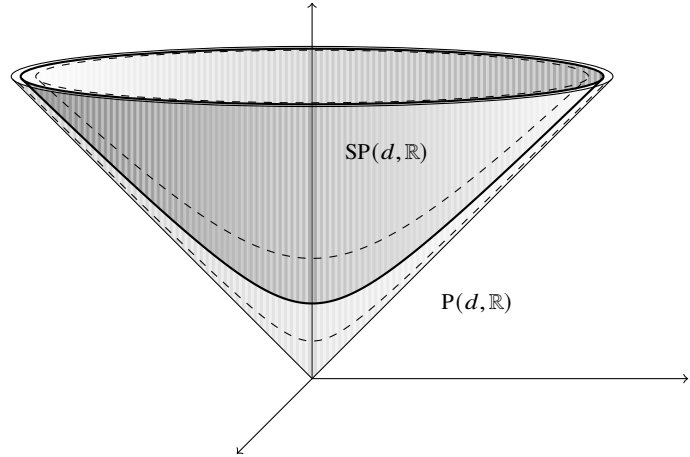
The group $\Gldr$ acts on $\Pdr$ by the action
\begin{equation}
\label{actionPnr}
(g,p)\mapsto gp\t g=:g[p].
\end{equation}
We next show that the action is transitive. By the  spectral theorem, for every  $p\in\Pdr$ there exist $O\in\Sodr$ and a diagonal matrix  $D$ with positive entries on the diagonal such that $p=O^{-1}DO$.
%all eigenvalues of a positive-definite matrix are positive. Hence let $A\in \Pdr$ then there exists $O\in\Sonr$ such that $OAO^{-1}=OA\t O$ is diagonal with positive entries on the diagonal. 
Since $D$ has positive entries on the diagonal, we can take its square root $D^\frac{1}{2}$. Let $g=O^{-1}D^\frac{1}{2}O$, then $g=\t g$ and 
\[p=g\t g=g[{\mathrm I}_d]\;,\]
which proves that the action is transitive. The stabilizer at ${\rm I}_d\in\Pdr$ is
\[{\rm O}(d,\R):=\{g\in\Gldr: g\t g={\rm I}_d\}\;.\]
Hence we have the diffeomorphism
\begin{equation*}
%\label{Pnrhomsp}
\Pdr\simeq \Gldr/{\rm O}(d,\R).
\end{equation*}
The submanifold $\SPdr$ is stable under the restriction of the previous action to $\Sldr$, whose action on $\SPdr$ is transitive. The stabilizer of ${\rm I}_d$  is $\Sodr$, so
\begin{equation*}
\SPdr\simeq\Sldr/\Sodr.
\end{equation*}

%We recall that $\dim (\Glnr)=n^2$ and $\dim ({\rm O}(n,\R))=n(n-1)/2$, then it follows by \eqref{Pnrhomsp} that 

Now we analyze the Riemannian structure on $\Pdr$, using \cite{pranab} as main reference. First of all, we observe that if $p\in \Pdr$, then $T_p\Pdr\simeq \Symdr$. We define
\begin{equation}
\label{RiemPnr}
\langle X,Y \rangle_p:=\tr (p^{-1}Xp^{-1}Y)
\end{equation}
where  $X,\,Y\in T_p(\Pdr)$. It is easy to see that $\langle\,\cdot\,,\,\cdot\,\rangle_p$ is an inner product. We check that the $\Gldr$-action preserves this form. Let $g\in \Gldr$. Then by \eqref{actionPnr} and \eqref{RiemPnr}
\begin{align*}
\langle \de g(X),\de g(Y)\rangle_{g.p} &= \langle gX\t g,gX\t g\rangle_{g.p} \\
%&=\tr ((g.p)^{-1}gX\t g(g.p)^{-1}gY\t g)\\
&=\tr (\t g^{-1}p^{-1}Xp^{-1}Y\t g)\\
&=\tr (p^{-1}Xp^{-1}Y)=\langle X,Y\rangle_p,
\end{align*}
because the trace is invariant under conjugation. Hence the Riemannian structure on $\Pdr$ defined in \eqref{RiemPnr} is $\Gldr$-invariant.
Now, take $p\in\Pdr$ and  define the mapping $\sigma_p: \Pdr\rightarrow \Pdr $ by
\[
\sigma_p(q)=pq^{-1}p=pq^{-1}\t p.
\]
Clearly,  $\sigma_p(p)=p$ and $\sigma_p^2(q)=q$ for every $q\in \Pdr$. It remains to show that $p$ is an isolated fixed point for $\sigma_p$. Let $q\in\Pdr$ be another nearby fixed point, that is $pq^{-1}p=q$. Thus, there exist   $ Y\in \gg$ and a small  $t>0$ such that  $q=p\exp (tY)$. Hence 
\[
p(p\exp (tY))^{-1}p=p\exp(tY),
\]
that is $\exp (-tY)=\exp(tY)$. If $t$ is smaller than the radius of the ball in which the exponential mapping is injective, this implies $Y=0$ and so $q=p$.
We have proved that $\Pdr$ is a symmetric space. Observe that if $p\in\SPdr$, then $\sigma_p(\SPdr)=\SPdr$ and so $\SPdr$ with the Riemannian metric restricted from $\Pdr$ is a symmetric space,  too. 

In the special case $d=2$, the symmetric space ${\rm SP}(2,\R)$ is isomorphic to the unit disk, in fact it is one of the possible realizations of the hyperbolic space $\H ^1$. It is important to observe that for a general $d>2$ there are no isometries between $\H ^d=\Sodone/\Sod$ and $\SPdr$, because the former has constant curvature while the latter has not.\\

The next results establish that there the Riemannian globally symmetric spaces are completely described by Lie algebraic data.

%%%%%%%%%%%%%%%%%%%%%%%%%%%%%%%%%%%%%%%%%%%%%%%%
\begin{proposition}[Lemma~3.2, Chap.~IV, \cite{dls}]\label{} Let $ \mathcal{M}$ be a Riemannian globally symmetric space. Then $I( \mathcal{M})$ has a smooth structure compatible with the compact-open topology which makes it a Lie group. 
\end{proposition}
%%%%%%%%%%%%%%%%%%%%%%%%%%%%%%%%%%%%%%%%%%%%%%%%

%%%%%%%%%%%%%%%%%%%%%%%%%%%%%%%%%%%%%%%%%%%%%%%%
\begin{theorem}[Theorem~3.3, Chap.~IV, \cite{dls}]\label{SS} Let $ \mathcal{M}$ be a Riemannian globally symmetric space, $p_0\in  \mathcal{M}$, $G=I_0( \mathcal{M})$, the connected component of the identity of $I( \mathcal{M})$.
\begin{enumerate}
\item[(i)] The isotropy subgroup $K$ of $G$ at $p_0$ is compact, and $ \mathcal{M}\simeq G/K$ under the map $gK\mapsto g[p_0]$.
\item[(ii)] The map $\sigma\colon g\mapsto s_{p_0}gs_{p_0}$ is an involutive automorphism of $G$ such that $K$ lies between the closed group $K_\sigma$ of the fixed points of $\sigma$ and its identity component. The subgroup $K$ contains no normal subgroups other than $\{e\}$.
\item[(iii)] Let $\gg$ be the Lie algebra of $G$ and $\gk$ be the Lie algebra of $K$. Then
\[
\gk=\Bigl\{X\in\gg:(\de\sigma_e)X=X\Bigr\}
\]
and if 
\[
\gp=\Bigl\{X\in\gg:(\de\sigma_e)X=-X\Bigr\}
\]
then $\gg=\gk+\gp$ as vector space direct sum. Let $\pi$ denote the natural projection $G\to G/K$. Then
$\de\pi_e$ maps $\gk$ into $\{0\}$ and $\gp$ isomorphically onto $T_{p_0} \mathcal{M}$. If $X\in\gp$, then the geodesic emanating from $p_0$ with tangent vector $\de\pi_e(X)$ is given by
\[
\gamma_{\de\pi_e(X)}(t)=\exp tX\cdot p_0.
\]
 Moreover, if $Y\in T_{p_0} \mathcal{M}$, then $(\de\exp tX)_{p_0}Y$ is the parallel translate of $Y$ along the geodesic.
\end{enumerate}
\end{theorem}
%%%%%%%%%%%%%%%%%%%%%%%%%%%%%%%%%%%%%%%%%%%%%%%%
\begin{definition}\label{SP} Let $G$ be a connected Lie group and $H$ a closed subgroup. The pair $(G,H)$ is called a {\it symmetric pair}\index{symmetric pair} if there exists an involutive analytic automorphism $\sigma$ of $G$, briefly called an {\it involution}\index{involution},  such that 
\[
({\rm Fix}(\sigma))_0\subset H \subset{\rm Fix}(\sigma).
\]
If in addition the group $\Ad_G(H)$ is compact, then $(G,H)$ is called a {\it Riemannian symmetric pair}\index{Riemannian symmetric pair}.
\end{definition}
%%%%%%%%%%%%%%%%%%%%%%%%%%%%%%%%%%%%%%%%%%%%%%%%
\begin{proposition}[Proposition~3.4 and Proposition~3.5, Chap.~IV, \cite{dls}]\label{SPSS} Let $(G,K)$ be a Riemannian symmetric pair, $\pi\colon G\to G/K$ the projection, $o=\pi(e)$. Let $\sigma$ be any involution of $G$ such that
$({\rm Fix}(\sigma))_0\subset K \subset{\rm Fix}(\sigma)$. In each $G$-invariant Riemannian structure $Q$ on $G/K$, and such $Q$ do exist, the manifold $G/K$ is a Riemannian globally symmetric space. 
The geodesic symmetry $\sigma_o$ satisfies 
\[
\sigma_o\circ\pi=\pi\circ\sigma,
\qquad
\tau(\sigma(g))=\sigma_o\tau(g)\sigma_o,
\]
where $\tau(g)\colon G/K\to G/K$ is the natural action of $g$, namely $\tau(g) xK=gxK$.
In particular $\sigma_o$ is independent of the choice of $Q$. Finally, if $\gz$ is the Lie algebra of the center of $G$ and $\gk\cap\gz=\{0\}$, then there exists exacly one involution $\sigma$  of $G$ such that
$({\rm Fix}(\sigma))_0\subset K \subset{\rm Fix}(\sigma)$. 
\end{proposition}
%%%%%%%%%%%%%%%%%%%%%%%%%%%%%%%%%%%%%%%%%%%%%%%%
The previous two results may be condensed in the statement that there is a bijective correspondence between Riemannian globally symmetric spaces and Riemannian symmetric pairs.

\subsection{Types of Symmetric Spaces}\label{sec:typessymmspaces}

The next step in the general theory of symmetric spaces is to look at the Lie algebra level. This is suggested by Theorem~\ref{SP}, which  shows that a Riemannian   globally symmetric space gives rise to a pair $(\gg, s)$, where $s=\de\sigma_e$, that satisfies
\begin{enumerate}
\item[(i)] $\gg$ is a real Lie algebra;
\item[(ii)] $s$ is an involutive automorphism of $\gg$;
\item[(iii)] the fixed points $\gk$ of $s$ form a Lie algebra  compacly contained in $\gg$,
\end{enumerate}
where (iii) holds because $K$ is compact (see Chap.~II in \cite{dls} for the definition of compactly embedded  Lie subalgebra).
\\
A pair $(\gg, s)$ satisying (i), (ii), and (iii) above is called an {\it orthogonal symmetric Lie algebra}\index{orthogonal symmetric Lie algebra}. If in addition
\begin{enumerate}
\item[(iv)] $\gk\cap\gz=\{0\}$,
\end{enumerate}
then $(\gg, s)$ is called {\it effective}\index{effective!orthogonal symmetric Lie algebra}. Fix such a pair and consider the decomposition $\gg=\gu+\gge$  into the $+1$ and $-1$ eigenspaces with respect to $s$. Motivated by the 
 important decomposition result stated below in Theorem~\ref{DEC}, one introduces the following terminology:
\begin{enumerate}
\item[(a)] if $\gg$ is compact and semisimple, then $(\gg, s)$ is said to be of the {\it compact type}\index{compact type!orthogonal symmetric Lie algebra};
\item[(b)]  if $\gg$ is noncompact and semisimple and if $\gg=\gu+\gge$ is a Cartan decomposition, then $(\gg, s)$ is said to be of the {\it noncompact type}\index{noncompact type!orthogonal symmetric Lie algebra};
\item[(c)] if $\gge$ is an Abelian ideal in $\gg$, then $(\gg, s)$ is said to be of the {\it Euclidean type}\index{Euclidean type!orthogonal symmetric Lie algebra}.
\end{enumerate}

%%%%%%%%%%%%%%%%%%%%%%%%%%%%%%%%%%%%%%%%%%%%%%%%
\begin{theorem}[Theorem~1.1, Chap.~V, \cite{dls}]\label{DEC} Suppose that $(\gg, s)$ is an effective  
orthogonal symmetric Lie algebra.
Then there exist ideals $\gg_0$, $\gg_-$ and $\gg_+$ such that 
\begin{enumerate}
\item[(i)] $\gg=\gg_0+\gg_-+\gg_+$, a Lie algebra direct sum;
\item[(ii)] $\gg_0$, $\gg_-$ and $\gg_+$ are invariant under $s$ and orthogonal with respect to the Killing form;
\item[(iii)] the pairs $(\gg_0,s_0)$,  $(\gg_+,s_+)$ and  $(\gg_-s_-)$ are effective  orthogonal symmetric Lie algebras of the Euclidean, compact and noncompact type, respectively.
\end{enumerate}
\end{theorem}
%%%%%%%%%%%%%%%%%%%%%%%%%%%%%%%%%%%%%%%%%%%%%%%%
The involutions $s_0$, $s_-$ and $s_+$ are those that arise by restricting $s$ to the corresponding ideals.
The above result is of course of central importance because it allows to study separately the various cases. Clearly,  the  decomposition yields a corresponding decomposition of a symmetric space and thus induces the notions of symmetric space of Euclidean, compact and noncompact types. The Euclidean space, the sphere and the unit disk, introduced in Sect.~~\ref{RGSS}, are the prototypical examples of such spaces. There is a remarkable duality between compact and noncompact types in which we are not interested. We content ourselves with mentioning that the compact types have positive sectional curvature and the noncompact ones have negative sectional curvature.

%%%%%%%%%%%%%%%%%%%%%%%%%%%%%%%%%%%%%%%%%%%%%%%%
Since we are only interested in noncompact globlally symmetric spaces, we focus on the corresponding structural assumptions. To this end, we need yet another piece of terminology and we also slightly change the current notation to tune into the noncompact case. Any pair $(G,K)$ where $G$ is a connected Lie group with Lie algebra $\gg$ and where $K$ is a Lie subgroup of $G$ with Lie algebra $\gk$ is said to be associated to the (effective) orthogonal symmetric Lie algebra $(\gg, \theta)$, and will be called of the noncompact type if such is  $(\gg, \theta)$. Thus, from now on we fix 
an effective orthogonal symmetric Lie algebra $(\gg, \theta)$ of the noncompact type, so that the eigenspace decomposition relative to $\theta$, namely $\gg=\gk+\gp$, is a Cartan decomposition. The next result is a cornerstone in the theory.
%%%%%%%%%%%%%%%%%%%%%%%%%%%%%%%%%%%%%%%%%%%%%%%%
\begin{theorem}[Theorem~1.1, Chap.~VI, \cite{dls}]\label{NCSS}  With the notation above, suppose that $(G,K)$ is any pair associated with the effective orthogonal symmetric Lie algebra  of the noncompact type $(\gg, \theta)$. Then:
\begin{enumerate}
\item[(i)] $K$ is connected, closed and contains the center $Z$ of $G$. Moreover, $K$ is compact if and only if $Z$ is finite. In this case, $K$ is a maximal compact subgroup of $G$;
\item[(ii)] there exists an involutive analytic automorphism $\Theta$ of $G$ whose fixed point set is $K$ and whose differential at the identity $e\in G$ is $\theta$; the pair $(G,K)$ is a Riemannian symmetric pair;
\item[(iii)]  the mapping $\varphi\colon (X,k)\mapsto(\exp X)k$ is a diffeomorphism of $\gp\times K$ onto $G$ and the mapping $\Exp$ is a diffeomorphism of $\gp$ onto the globally symmetric space $G/K$.
\end{enumerate}
\end{theorem}
The exponential mapping $\Exp$ in item~(iii) above, quoted for completeness, is just the Riemannian exponential mapping (see for instance~\cite{dls}) and will play no explicit role in what follows.
%%%%%%%%%%%%%%%%%%%%%%%%%%%%%%%%%%%%%%%%%%%%%%%%
\vskip0.2truecm
{\bf Assumption 1.}
From now on, let $G$ be a connected semisimple Lie group with finite center and $X=G/K$ the associated symmetric space of the noncompact type, where $K$ is a maximal compact subgroup of $G$. We also fix an Iwasawa decomposition $G=KAN$ and we denote by $M$ the centralizer of $A$ in $K$.

%%%%%%%%%%%%%%%%%%%%%%%%%%%%%%%%%%%%%%%%%%%%%%%%

\subsection{Boundary of a Symmetric Space}\label{sec:boundary} Our basic example of noncompact symmetric space will be the unit disk $\DD$, which has a rather obvious (topological) boundary, namely the unit circle $S^1=\{z\in\C:|z|=1\}$.
The notion of boundary of a symmetric space is highly non-trivial.  For a deep study on the matter, the reader is referred to the classical paper of Furstenberg \cite{furst} in which a detailed motivation of Definition~\ref{bdry} below may be found. For our purposes, some heuristics and some basic observations will suffice.

Notice first that the M\"obius action of $G=\Suone$ on $\C$ has precisely three orbits, namely $\DD$, $S^1$ 
and the complement $\{w\in\C:|w|>1\}$. We already know that $\DD$ is an orbit. 
Further, $AN$ fixes $1$ (easy to check) and $K$ moves it along the unit circle, so that the $G$-orbit of $1$ is $S^1$. 
Finally, for $\rho>1$ the formula
\begin{equation}\label{Kone}
k_{\theta/2}\cdot\rho=\begin{bmatrix}\E^{i\theta/2}&0\\0&\E^{-i\theta/2}\end{bmatrix}\cdot\rho=
\rho\cos\theta+i\rho\sin\theta,
\end{equation}
shows that $K$ maps the point $\rho$ along the circle of radius $\rho$  and any such real point may be reached, say, from $2$ by means of $A$ because for $t>0$ the real numbers
\[
a_t[2]=
\begin{bmatrix}\cosh t&\sinh t\\\sinh t&\cosh t\end{bmatrix} [2]=\frac{2+\tanh t}{2\tanh t+1}
\]
span the half-line $(1,+\infty)$. Thus the set $\{w\in\C:|w|>1\}$ is an orbit.

Let's go back to the unit circle. As already noticed, $AN$ fixes $1$ and $K$ moves it along the circle, as can also be deduced from \eqref{Kone} when $\rho=1$. The very same formula shows also that the elements
$k_{\theta/2}$ when $\theta$ is any multiple of $2\pi$ fix $1$. These are $\pm I$, namely the elements  of $M$, the centralizer of $A$ in $K$. Therefore, the stabilizer of $1$ is the group $P=MAN$ and 
$S^1\simeq G/P$. By means of the Iwasawa decomposition we may write 
\[
S^1\simeq KAN/MAN
\]
and the natural question arises whether this is the same as $K/M$ or not. In the case at hand this is quite clearly so because $K$ acts transitively with isotropy $M$. This actually holds more generally in the sense that
\[
G/P= KAN/MAN\simeq K/M.
\]
Indeed, $K$ acts on the coset space $G/P$ in the natural fashion $k\cdot gP=(kg)P$ and 
by the  Iwasawa decomposition $k\in P=MAN$ if and only if $k\in M$. Hence the isotropy at the coset $\{P\}$ is $M$. Further, again by the Iwasawa decomposition, the action is transitive, and we conclude that $G/P\simeq K/M$. The reverse point of view (that of $G$ acting on $K/M$ with isotropy $P$) will be illustrated below in~\eqref{GonB}, where the explicit action of $G$ on $K/M$ is given.

\begin{definition}\label{bdry} The  {\it boundary}\index{boundary} of  $X$ is the coset space $B:=K/M$.
\end{definition}

We remark here {\it en passant} that $M$, which will play an important role below, normalizes $N$, that is
\begin{equation}\label{MnormN}
mNm^{-1}=N, \qquad m\in M.
\end{equation}
To see this, look at the Lie algebra level. If $\alpha$ is a positive root and $X\in\gg_\alpha$, then for
every $H\in\ga$ it is
\[
[H,\Ad m X]=\Ad m[\Ad m^{-1}H,X]=\Ad m[H,X]=\alpha(H)\Ad mX,
\]
so that $\Ad m(\gg_\alpha)\subset \gg_\alpha$. An other normalization property that involves $N$  is that 
for any $\alpha\in A$ and any $\nu\in N$ it holds 
\begin{equation}\label{saN}
\alpha\nu aN=aN\alpha\nu.
\end{equation} This, in turn,  follows from choosing $\nu'\in N$ such that $\nu'\alpha=\alpha\nu$, which gives
\[
\alpha\nu aN=\alpha aa^{-1}\nu aN
=\alpha aN\alpha^{-1}\alpha
=a\alpha N\alpha^{-1}\alpha
=aN\alpha
=aN\nu'\alpha
=aN\alpha\nu.
\]

%%%%%%%%%%%%%%%%%%%%%%%%%%%%%%%%%%%%%%%%%%%%%%%%
\subsection{Changing the Reference Point}\label{CRP}
In what follows, it will be useful to change the reference point of both the symmetric space $X$ 
and its boundary. Although conceptually very well known and somehow trivial, the actual explicit determination of what happens when doing so is not to be found in the literature, to the best of our knowledge.
 In order to see how the various decompositions are affected by changing the origin of our spaces, it is convenient to introduce Borel sections and occasionaly adopt a slightly different notation for the  (various) $G$-actions.

The action of $G$ on $X=G/K$ will be written $g[x]$, namely
\[
g[x]=g[hK]=ghK.
\]
 For any fixed $x_0\in X=G/K$, a  \emph{Borel section}\index{Borel section} relative to $x_0$ is a measurable map  $s_{x_0}\colon X\to G$ satisfying  $s_{x_0}(x)[x_0]=x$
and $s_{x_0}(x_0)=e$, with $e$ the neutral element of $G$. Borel sections always
exist since $G$ is second countable, see Theorem~5.11 in \cite{varadarajan85}. 
%We shall denote by $s_o\colon X\rightarrow G$ a Borel section associated to $o=eK\in G/K$, that is a function satisfying  $s_o(x)[o]=x$ for every $x\in X$ and $s_o(o)=e\in G$.
 
  We next show how, in the present context, a Borel section associated to $o=eK\in G/K$ can be determined quite explicitly. Since $K$ is the isotropy subgroup of $G$ at $o$, the map $\beta\colon gK\mapsto g[o]$ is a diffeomorphism of $G/K$ onto $X$. Furthermore, by the Iwasawa decomposition of $G$ (Theorem~\ref{ID}), each element of $g\in G$ can be written as the product $g=nak$ for exactly one triple $(n,a,k)\in N\times A\times K$, and the correspondence $(n,a,k)\leftrightarrow nak$ is a diffeomorphism with $G$. Hence each class in $G/K$ has a representative of the form $naK$ with unique $a\in A$ and $n\in N$, so that the mapping $\psi\colon G/K\rightarrow NA$ given by $naK\mapsto na$ is a diffeomorphism. It follows that the measurable, actually smooth, map
	\begin{equation*}
	%\label{borelsec}
	\psi\circ\beta^{-1}\colon X\longrightarrow NA
	\end{equation*}
is a Borel section. Indeed, $\psi\circ\beta^{-1}(o)=\psi(K)=e$ and, by construction, for every $x\in X$, it holds $\psi\circ\beta^{-1}(x)[o]=x$ . From now on, we will denote by $s_o$ the Borel section $\psi\circ\beta^{-1}$ with image $NA\subseteq G$.
 
% {\color{red}In fact, $s_0$ may be chosen to have image in $AN$. Indeed,  if $\tilde s_o$ is any Borel section, then since $\psi\colon A\times N\times K\to G$ given by $\psi(a,n,k)=ank$ is a diffeomorphism by Theorem~\ref{IWASAWA},  the composition $s_o=\psi\circ\pi\circ\psi^{-1}\circ\tilde s_o$ where  $\pi\colon A\times N\times K\to A\times N$ is the projection $\pi(a,n,k)=(a,n)$, 
% is measurable. The uniqueness properties of the Iwasawa decomposition show that if $gK=g'K$ then their $AN$-components are necessarily equal in the decomposition $ANK$, so that if $\tilde s_o$ is a section such is also $s_o$ and the latter has image in $AN$. The way to picture $s_o$ is to pick a representative $g=ank$ in the coset $x=gK$ via a Borel section $\tilde s_0$ and then map it to $an$. }
 
 Fix now $x\in X$ and let $K_x$ be the isotropy of $G$ at $x\in X$. Evidently,
 \[
 K_x=s_o(x)Ks_o(x)^{-1}.
 \]
It is then possible to write an Iwasawa decomposition w.r.t. the subgroup $K_x$. In fact,
\[
G=s_o(x)Gs_o(x)^{-1}=s_o(x)KANs_o(x)^{-1}=s_o(x)Ks_o(x)^{-1}AN=K_xAN,
\]
because, as observed earlier, $s_o(x)\in AN$. By using the same approach, one obtains the various versions of the Iwasawa decomposition where the factors appear in a different order. It is worth  observing that the subgroups $A$ and $N$ are independent of the maximal compact subgroup $K_x$, but the individual factors appearing in the decomposition of a fixed element $g\in G$ are not. Given $g\in G$, we denote with $H_x(g)$, $A_x(g)$ the elements of $\mathfrak{a}$ uniquely determined by
\begin{equation}\label{IWAx}
g\in K_x\exp H_x(g) N,\qquad g\in N\exp A_x(g) K_x
\end{equation}
and by $\kappa_x(g)$ the unique element in $K_x$ such that $g\in\kappa_x(g)AN$. Clearly,
\begin{equation}\label{eq:fondAH}
A_x(g^{-1})=-H_x(g).
\end{equation}

%%%%%%%%%%%%%%%%%%%%%%%%%%%%%%%%%%%%%%%%%%%%%%%%
Once the point $x\in X$ has been fixed, a Borel section $s_x\colon X\rightarrow G$ can also be fixed, so  that for every $y\in X$, $s_x(y)[x]=y$ and $s_x(x)=e$. As before, it may be arranged that $s_x(y)\in NA=AN$. 
Also, we denote by $M_x$ the centralizer of $A$ in $K_x$, so that $M_x= s_{o}(x) M s_{o}(x)^{-1}$.
The following technical observation will be useful below.
%%%%%%%%%%%%%%%%%%%%%%%%%%%%%%%%%%%%%%%%%%%%%%%%
\begin{lemma}\label{kappas} For any $x\in X$ it is
\begin{enumerate}
\item[(i)] $\kappa_{o}\circ{\kappa_x}\bigl|_{K}=id_K$; in particular, if $k_x=\kappa_x(k)$ for some $k\in K$, then $k=\kappa_o(k_x)$;
\item[(ii)] $\kappa_{x}\circ{\kappa_o}\bigl|_{K_x}=id_{K_x}$.
\end{enumerate}
\end{lemma}
\begin{proof} We start by proving (i). Let $k\in K$. Then according to the Iwasawa decomposition $K_xAN$ it is $k=\kappa_x(k)an$,
	that is $\kappa_x(k)=k(an)^{-1}\in KAN$. So that $\kappa_o(\kappa_x(k))$ is precisely $k$, as desired.
	The proof of (ii) is analogous.
%	Let $k\in K$. Write $k=\kappa_x(k)an$
%and $\kappa_x(k)=\kappa_o(\kappa_x(k))a'n'$ according to the Iwasawa decompositions $K_xAN$ and $KAN$, respectively. Thus, 
%\[
%k^{-1}\kappa_o(\kappa_x(k))=\left(\kappa_x(k)an\right)^{-1}\left(\kappa_x(k)(a'n')^{-1}\right)\in AN.
%\]
%However, $k^{-1}\kappa_o(\cdot)\in K$, and hence  for any $k\in K$ it is  $k^{-1}\kappa_{o}(\kappa_x(k))\in AN\cap K=\{e\}$. Hence $\kappa_{o}\circ{\kappa_x}\bigl|_{K}=id_K$.
%
%(ii) Let $k_x\in K_x$. Then according to the Iwasawa decomposition $KAN$ it is $k_x=\kappa_o(k_x)an$,
%that is $\kappa_o(k_x)=k_x(an)^{-1}\in K_xAN$ so that $\kappa_x(\kappa_o(k_x))$ is precisely $k_x$, as desired.
\end{proof}

The action of $G$ on the boundary $B=K/M$ is induced by the decomposition $G/P=KAN/MAN$ in the sense that if $g\in G$ and $kM\in B$ then 
\begin{equation}\label{GonB}
g\langle kM\rangle:=\kappa_{o}(gk)M.
\end{equation}
Consider now the action of $K_x$. By the definition~\eqref{GonB} and by item~(i) in Lemma~\ref{kappas}, for any $k\in K$ it is
\begin{equation*}
%\label{k0kx}
\kappa_x(k)\langle M\rangle=\kappa_{o}(\kappa_x(k))M=kM. 
\end{equation*}
Thus the action of $K_x$ on the boundary is transitive. 
Next, observe that an element $k_x=s_o(x)ks_o(x)^{-1}$ stabilizes $M\in K/M$ if and only if $\kappa_o(s_o(x)ks_o(x)^{-1})\in M$, which means $s_o(x)k\in MAN$. This, together with the fact that $M$ normalizes $AN$, implies that $k\in M$, hence 
$k_x\in M_x$. Therefore the isotropy group of $K_x$ at $M$ is $M_x$. This shows that the map induced by $\kappa_o$ on $K_x/M_x$, which we denote $\kappa_{x,o}$, namely
\begin{equation}\label{identification}
\kappa_{x,o}:K_x/M_x\to K/M,
\qquad
k_xM_x\mapsto\kappa_{x,o}(k_xM_x):=\kappa_o(k_x)M,
\end{equation}
is a diffeomorphism.
Furthermore, $kM$ and $\kappa
_x(k)M_x$ determine the same boundary point, because by \eqref{identification} $\kappa_o(\kappa_x(k))M=kM$. By Lemma~\ref{kappas} the inverse of $\kappa_{x,o}$ is the map
\[\kappa_{o,x}:K/M\to K_x/M_x,
\qquad
kM\mapsto\kappa_{o,x}(kM):=\kappa_x(k)M_x.\]
%
%under the transitive $K_x$-action on the boundary $B$, the coset $\kappa_x(k)M_x$  is identified with its action on the base point $M$, namely $\kappa_x(k)\langle M\rangle$, which is $kM$ by~\eqref{k0kx}.
%%%%%%%%%%%%%%%%%%%%%%%%%%%%%%%%%%%%%%%%%%%%%%%%

%%%%%%%%%%%%%%%%%%%%%%%%%%%%%%%%%%%%%%%%%%%%%%%%
%%%%%%%%%%%%%%%%%%%%%%%%%%%%%%%%%%%%%%%%%%%%%%%%
%%%%%%%%%%%%%%%%%%%%%%%%%%%%%%%%%%%%%%%%%%%%%%%%
%%PARTE NUOVA%%%

\subsection{Horocycles}\label{sec:horo}

A hyperplane in $\R^n$ is orthogonal to a family of parallel lines. What is a reasonable analogue of this in, say, Riemannian geometry? Since geodesics  are very natural generalizations of lines, a possible answer is given by a manifold that is orthogonal to families of parallel geodesics. In the context of symmetric spaces, such manifolds will be called {\it horocycles}, sometimes also {\it horospheres}. 
%%%%%%%%%%%%%%%%%%%%%%%%%%%%%%%%%%%%%%%%%%%%%%%%

\vskip0.2truecm
Let us see what this idea leads to in the context of the unit disk, our basic example of noncompact symmetric space.
The origin in $\DD$ will be denoted $o$. If $\gamma\colon[a,b]\to\DD$ is a smooth curve with $\gamma(a)=o$ and $\gamma(b)=x\in(-1,1)$ is a point on the real axis, then the simple inequality
\[
\frac{\dot x(t)^2}{(1-x(t)^2)^2}\leq\frac{\dot x(t)^2+\dot y(t)^2}{(1-x(t)^2-y(t)^2)^2}
\]
shows that straight real  lines through the origin are geodesics. We observe {\it en passant} that since $\gamma_0(t)=(tx,0)$ with $t\in[0,1]$ is such a straight line, then
\[
d(o,x)=L(\gamma_0)=\int_0^1 \frac{|x|}{1-t^2|x|}\,\de t=\frac{1}{2}\log\frac{1+|x|}{1-|x|}.
\]
As we know,  $G=\Suone$ acts by isometries via the M\"obius action on $\DD$. Such maps are conformal and map circles and lines into circles and lines. Hence the geodesics in $\DD$ are circular arcs perpendicular to the boundary $|z|=1$. All circular arcs perpendicular to the same point at the boundary may be seen as parallel lines and thus a natural notion of horocycle in this context is that of circle tangent to the boundary (less the point on $S^1$) because such a circle is of course perpendicular to all the above parallel geodesics. 

The circle through the origin and tangent to the boundary at $1\in\C$ is therefore the prototype of horocycle. 
Observe that
\[
n_s[o]=\begin{bmatrix}1+is&-is\\is&1-is\end{bmatrix}[o]
=\frac{-is}{1-is}=\frac{s}{s+i}=\frac{s^2}{s^2+1}-i\frac{s}{s^2+1}
\]
and an easy calculation shows that these are precisely the points on the circle of radius $1/4$ centered at $1/2\in\C$ that are contained in $\DD$. Furthermore, as $s\to\pm\infty$ one gets  the boundary point $b_0=1\in\C$. We have obtained the basic horocycle, which will be denoted $\xi_o$, as the $N$-orbit $N[o]$.  
%-----------------------------------------------------------------------------
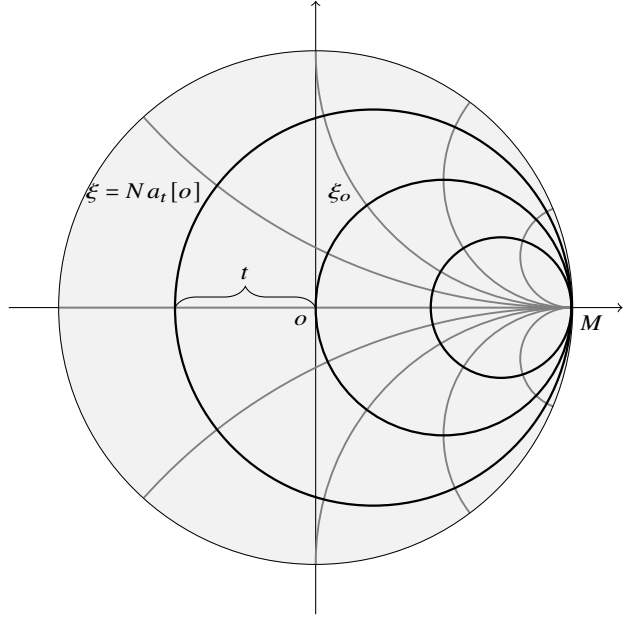
\begin{figure}[htbp]
\begin{center}
	\begin{tikzpicture}[scale=1.7]
	\draw[fill=gray,opacity=0.1] (0,0) circle (2.01cm);
	\draw[->] (0,-2.4)--(0,2.4);
	\draw[->] (-2.4,0)--(2.4,0);
	\draw [decorate,decoration={brace,amplitude=8pt}] (-1.10,0) -- (0,0) node [black,midway,yshift=13pt] {$t$};
	\draw[gray,line width=0.25mm] (2,0) arc (270:142.9:1cm);
	\draw[gray,line width=0.25mm] (2,0) arc (270:110.5:0.4cm);
	\draw[gray,line width=0.25mm] (2,0) arc (90:249.5:0.4cm);
	\draw[gray,line width=0.25mm] (-2.01,0) -- (2.01,0);
	\draw[gray,line width=0.25mm] (2,0) arc (90:217.5:1cm);
	\draw[gray,line width=0.25mm] (2,0) arc (270:179.7:2cm);
	\draw[gray,line width=0.25mm] (2,0) arc (90:180.3:2cm);
	\draw[gray,line width=0.25mm](2,0) arc (90:138:4.5cm);
	\draw[gray,line width=0.25mm](2,0) arc (270:222:4.5cm);
	\draw[line width=0.3mm] (1,0) circle (1cm);
	\draw[line width=0.3mm] (1.45,0) circle (0.55cm);
	\draw[line width=0.3mm] (0.45,0) circle (1.55cm);
	\draw[line width=0.1mm] (0,0) circle (2.01cm);
	\node[anchor=east] at (-0.82,0.9) 	{$\xi=Na_t[o]$};
	\node[anchor=north east] at (-0,0) 	{$o$};
	\node[anchor=east] at (0.35,0.9) 	{$\xi_o$};
	\node[anchor=north west] at (2,0) 	{$M$};
	\end{tikzpicture}
	\caption{The basic horocycle  $\xi_o$ in the unit disc and the horocycle $\xi$ tangent to the boundary at $1$ and with distance $-t$ from the origin $o$. In gray, the sheaf of parallel geodetics perpendicular to $\xi_o$ and $\xi$.}
	\label{fig:horo}
\end{center}
\end{figure} 
%-----------------------------------------------------------------------------

Other horocycles tangent to $b_0$ are the orbits $Na_t[o]=a_tN[o]$ where of course
\[
a_t=\begin{bmatrix}\cosh t&\sinh t\\\sinh t&\cosh t\end{bmatrix} 
\]
is any member of $A$ (recall that $A$ normalizes $N$). This is because
\[
a_t[o]=\tanh t\in(-1,1)
\]
parametrizes any other point on the geodesic line $(-1,1)\subset\C$ and an easy calculation shows that its $N$-orbit is just the circle through that point and tangent to $b_0$ (see Fig. \ref{fig:horo}). 
It is clear that by acting with the rotation group one gets all other horocycles, that is, all the circles in $\DD$ tangent to  the boundary. Thus, any other horocycle $\xi$ can be written in the form $ka\cdot\xi_0$ with $k\in K$ and $a\in A$. But this means
\[
\xi=(ka)N (ka)^{-1}(ka[o]),
\]
which exhibits $\xi$ as an orbit of a group conjugate to $N$, namely $(ka)N (ka)^{-1}$. This motivates the Definition~\ref{horo} below.

%Let $G=KAN$ be the fixed Iwasawa decomposition as in Assumption~1, and let $o$ denote the origin $\{K\}$ in $X$ and $\xi_0$ denote the orbit $N\cdot o$.
%%%%%%%%%%%%%%%%%%%%%%%%%%%%%%%%%%%%%%%%%%%%%%%%
\begin{definition}[\cite{gass}] \label{horo} A {\it horocycle}\index{horocycle} in $X$ is any orbit of any subgroup of $G$ conjugate to $N$, that is an orbit $N^g[x]$ where $x\in X$, $g\in G$ and $N^g=gNg^{-1}$.  We shall denote by $\Xi$ the set of all horocycles in $X$.% and we write  $ {\xi}_x=N[x]\in\Xi$.
\end{definition}
%%%%%%%%%%%%%%%%%%%%%%%%%%%%%%%%%%%%%%%%%%%%%%%%
 
By Theorem 1.1 in Chap.II in \cite{gass}, horocycles  are closed submanifolds of $X$,  the $G$-action on $X$ maps horocycles to horocycles and in fact the group $G$ acts transitively on $\Xi$ by 
\[(g,N^{h}[x])\mapsto g.(N^{h}[x]):=gN^{h}[x].\]
We fix $x\in X$ and we consider the horocycle $\xi=N[x]$. By Theorem 1.1 in Chap.II in \cite{gass}, the isotropy at $\xi$ is $M_xN$ and therefore
\begin{equation*}
%\label{2xi}
\Xi\simeq G/M_xN
\end{equation*}
under the diffeomorphism $gM_xN\mapsto gN[x]$.
%A few words on the isomorphism~\eqref{2xi} are in order. In what follows, we write $ {\xi}$ to mean the actual submanifold of $X$, so $ {\xi}\subset X$ whereas we denote by $\xi$ the element in the coset space $G/M_xN$ that labels it, so $\xi\in G/M_xN$.
%This is further clarified by defining explicitly  the diffeomorphism in~\eqref{2xi}, that is
%\begin{equation}\label{Horbit}
%\varphi_x\colon G/M_xN\to\Xi,
%\qquad
%gM_xN\mapsto gN[x]
%\end{equation}
%and, when no confusion arises on the reference point $x\in X$, we write for short
%\[
%\varphi_x(\xi)=\widehat{\xi}.
%\]
Furthermore, by Proposition 1.4 in Chap.II in \cite{gass}, $(K_x/M_x)\times A$ is diffeomorphic to $G/M_xN$ under the mapping 
\begin{equation}\label{xhoro}
(k_xM_x,a)\mapsto k_xaM_xN.
\end{equation}
Therefore, for each horocycle $ {\xi}\in\Xi$ there exist unique $k_xM_x\in K_x/M_x$ and $a\in A$ such that 
\begin{equation}
\label{hatnohat}
	{\xi}=k_xaN[x].
\end{equation}
Finally, since $K/M$ is diffeomorphic to $K_x/M_x$ under the mapping $\kappa_{o,x}(kM)=\kappa_x(k)M_x$, we define the diffeomorphism
\begin{equation}\label{Psix}
	\Psi_x\colon K/M\times A\longrightarrow {\Xi},\quad (kM,a)\mapsto \kappa_x(k)aN[x].
\end{equation}
%is a diffeomorphism of $(K_x/M_x)\times A$ onto $G/M_xN$. Therefore, since
%\[
%\varphi_x\circ\psi_x\colon(K_x/M_x)\times A\to G/M_xN\to\Xi
%\]
%is a diffeomorphism, for each horocycle $\widehat{\xi}\in\Xi$ there exist unique $k_xM_x\in K_x/M_x$ and $a\in A$ such that 
%\begin{equation}\label{hatnohat}
%\xi=k_xaM_xN,
%\qquad
%\widehat{\xi}=k_xaN[x].
%\end{equation}

%%%%%%%%%%%%%%%%%%%%%%%%%%%%%%%%%%%%%%%%%%%%%%%%
%Formula~\eqref{xhoro} says that horocycles can be parametrized by a boundary %point and an element of $A$ with respect to  any reference point  in $X$. 

Observe that the boundary point $kM\in K/M$ which identifies the horocycle $\xi=\kappa_x(k)aN[x]$ through \eqref{Psix} is independent of the choice of the reference point $x\in X$. Namely, for every $x, y\in X$
\[
\Psi_x(kM,a)=\Psi_y(kM,a')
\]
for some $a'\in A$. Indeed, if $\xi=k_xaN[x]$ and if we pick $y\in X$,  hence $k_yM_y\in K_y/M_y$ and $a'\in A$ such that $\xi=k_ya'N[y]$, then $k_yM_y=\kappa_y(k_x)M_y$ and this identifies the boundary point $\kappa_o(k_x)M$. Indeed, by the $K_yAN$ and $KAN$ Iwasawa decompositions of $k_x$, we have that%=\kappa_{y,o}(\kappa_y(k_x)M_y)\in K/M$. 
\[\kappa_y(k_x)\in k_xAN=\kappa_{o}(k_x)AN,\]
so that 
\[\kappa_{y,o}(\kappa_y(k_x)M_y)=\kappa_o( \kappa_y(k_x))M=\kappa_{o}(k_x)M.\]
%We summarize this in
%\begin{equation*}
%%\label{normals}
% \xi=k_xaN[x]=k_ya'N[y]
%\quad
%\implies\quad
%%k_yM_y=\kappa_y(k_x)M_y. 
%\kappa_{x,o}(k_xM_x)=\kappa_{y,o}(k_yM_y).
%\end{equation*}

We shall say that $\Psi_x(kM,a)$ represents the horocycle with \emph{normal}\index{normal!horocycle}  $kM$ and  \emph{composite distance} $\log a$ from $x$ (see below, Definition~\ref{CD}). We stress that the normal of a horocycle is independent of the choice of $x\in X$. The composite distance, however, is different for different reference points.

This parametrization generalizes the geometric picture in $\DD$, where a horocycle
$\xi=ka_tN[o]$ is identified by the boundary point $kM\in K/M$ to which it is tangent and 
the ``signed distance'' $t$ from the reference point, see Fig.~\ref{fig:horo}.%(for example $t$ if $\widehat{\xi}=Na_t[o]$).

%%%%%%%%%%%%%%%%%%%%%%%%%%%%%%%%%%%%%%%%%%%%%%%%
%For later use, and to stress the invariance of the normal, we introduce a mapping that keeps track of the reference point while using the same ``universal'' parameters $K/M\times A$. Thus, for $x\in X$ we define the diffeomorphism 
%\begin{equation}\label{Psix}
%\Psi_x:=\psi_x\circ(\kappa_x\otimes id_A)\colon K/M\times A\longrightarrow G/M_xN,
%\quad
%(kM,a)\longmapsto \kappa_x(k)aM_xN.
%\end{equation}
%The way to think of this is that $\Psi_x(kM,a)$ represents the horocycle with normal  $kM$ and  composite distance $\log a$ from $x$ (see below, Definition~\ref{CD}).

%%%%%%%%%%%%%%%%%%%%%%%%%%%%%%%%%%%%%%%%%%%%%%%%
\begin{proposition}\label{horoyb} Fix a reference point $x\in X$. The horocycle through $y\in X$ with normal $kM$ is $N^{\kappa_x(k)}[y]$.
\end{proposition}
\begin{proof} An equivalent statement is that, writing $k=\kappa_o(k_x)$ with $k_x\in K_x$,  the horocycle through $y$ with normal $\kappa_o(k_x)M$ is $k_xNk_x^{-1}[y]$ because $k_x=\kappa_x(k)$ by item (ii) in Lemma~\ref{kappas}.% and $k=\kappa_o(k_x)$ by item (i) of the same lemma.
\\
Since $k=\kappa_o(k_x)$, then $kM$ and $k_xM_x$ identify the same boundary point and a horocycle with normal $kM$ has the form $\xi=k_xaM_xN$ as in~\eqref{hatnohat}. If this represents a horocycle through $y$, then  there exists $g\in G$ such that 
\[
 {\xi}=gNg^{-1}[y]=\kappa_x(g)N\kappa_x(g)^{-1}[y].
\]
Now observe that there exist $\alpha\in A$ and $\nu\in N$ such that $\kappa_x(g)^{-1}[y]=\nu\alpha[x]$, then $ {\xi}=\kappa_x(g)\alpha N[x]$. Thus, since $ {\xi}=k_xaN[x]$, we have that 
\[ \kappa_x(g)\alpha N[x]=k_xaN[x],\]
which by \eqref{hatnohat} implies $\kappa_x(g)M_x=k_xM_x$.
%\varphi_x\circ\psi_x(k_xM_x,a)=\varphi_x(\xi)
%=\varphi_x(\kappa_x(g)\alpha M_xN)
%=\varphi_x\circ\psi_x(\kappa_x(g)M_x,\alpha)
%\]
Hence $\kappa_x(g)=k_xm_x$ for some $m_x\in M_x$. 
However, \eqref{MnormN} implies at once that $m_xNm_x^{-1}=N$, and hence
$N^{\kappa_x(g)}=N^{k_x}$. 
\end{proof}

%%%%%%%%%%%%%%%%%%%%%%%%%%%%%%%%%%%%%%%%%%%%%%%%
%We remark that from the above proof it turns out that $\alpha=a$ and hence, since also
%$\kappa_x(g)=k_xm_x$, and $m_x\in M_x\subset K_x$, the horocycle has the form
%\[
% {\xi}=k_xm_x aN[x]=k_xa(m_x Nm_x^{-1})m_x[x]=k_xaN[x]
%\]
%for a unique $a\in A$. Recall also that $k_x=\kappa_x(k)$, so that $ {\xi}=\kappa_x(k)aN[x]$.
%Motivated by these remarks, we formalize an important concept.
\begin{definition}\label{CD}
Fix a reference point $x\in X$ and choose $y\in X$ and $b\in K/M$, so that  by Proposition~\ref{horoyb} the horocycle  $ {\xi}= {\xi}(y,b)$ passing through $y$ with normal $b=kM$ is uniquely determined, and hence there exists a unique $a\in A$ such that
\begin{equation*}
%\label{canonicAl}
 {\xi}(y,kM)=\kappa_x(k)aN[x].
\end{equation*}
We denote by $A_x(y,b)\in\mathfrak{a}$ the {\it composite distance}\index{composite distance!horocycle} of the horocycle $ {\xi}(y,b)$ from $x\in X$, namely
\[
A_x(y,b)=\log a,
\]
\end{definition}
%%%%%%%%%%%%%%%%%%%%%%%%%%%%%%%%%%%%%%%%%%%%%%%%
The reader is warned not to confuse the composite distance $A_x(y,b)$,
which depends on $(y,b)\in X\times B$, with the Abelian component $A_x(g)$ of $g$ in the Iwaswawa decomposition $NAK_x$, which is a function on $G$ (see~\eqref{IWAx}). A relation between the two does exist, as pointed out in the next lemma, where we collect several properties of the composite distance which will play a crucial role in our work.
\begin{lemma}\label{propA}
Fix a reference point $x\in X$. Then:
\begin{enumerate}
		\item[(i)] for any $k_x\in K_x$ and $g\in G$ we have
		\begin{equation}\label{Abello}
		A_{x}(g[x]\,,\kappa_{o}(k_{x})M)=A_x(k_x^{-1}g),
		\end{equation}
where the right-hand side is defined by~\eqref{IWAx};
		\item[(ii)] for any $y\in X$, $kM\in K/M$ and $g\in G$ we have
		\begin{equation}\label{ginvA}		
		A_x(y,kM)=A_{g[x]}(g[y]\,,g \langle kM\rangle);
		\end{equation}
		\item[(iii)] for any $y,z\in X$ and $kM\in K/M$ we have
		\begin{equation}\label{cocycl}
		A_{x}(y,kM)=A_x(z,kM)+A_{z}(y,kM).
		\end{equation}
\end{enumerate}
\end{lemma}
\begin{proof} \begin{enumerate}
	 \item[(i)]
 Let $k_x\in K_x$ and $g\in G$. By Proposition~\ref{horoyb} and (ii) of Lemma \ref{kappas},  the horocycle passing through $g[x]$ with normal $\kappa_{o}(k_{x})M$ is $k_{x}Nk^{-1}_{x}g[x]$. By Definition \ref{CD}, we have that
		\[k_{x}Nk^{-1}_{x}g[x]=k_{x}\exp(A_{x}(g[x],\kappa_{o}(k_{x})M))N[x],\] 
		and so  $k_{x}^{-1}g\in N\exp(A_{x}(g[x],\kappa_{o}(k_{x})M))K_{x}$. This proves (i).

\item[(ii)] For simplicity, we first prove the statement in the case $x=o$. Let $y\in X$, $kM\in K/M$ and $g\in G$.
By Proposition~\ref{horoyb}, and the fact that $A$ normalizes $N$, the horocycle passing through $g[y]$ with normal $g \langle kM\rangle=\kappa_{o}(gk)M$  (see~\eqref{GonB}) is
\[
N^{\kappa_o(gk)}g[y]
=\kappa_o(gk)N\kappa_o(gk)^{-1}g[y]
=gkN(gk)^{-1}g[y].
\]
By the diffeomorphism given in~\eqref{hatnohat}, there exist $h\in K_{g[o]}$ and $a\in A$ such that
		\begin{equation}\label{eq:dimii}
		gkNk^{-1}[y]=haNg[o],
		\end{equation}
and thus, by definition
\begin{equation*}
%\label{LHS}
a=\exp(A_{g[o]}(g[y]\,,g\langle kM\rangle)).
\end{equation*}
		We need to show that $a=\exp(A_{o}(y,kM))$.
		Since  $K_{g[o]}=gKg^{-1}$, we have $h=gk_1g^{-1}$ for some $k_1\in K$ and we claim that
		\begin{equation}\label{claim0}
		k_1\kappa_{o}(g^{-1})M=kM.
		\end{equation}
		By \eqref{eq:dimii} we have that 
		\[k_1g^{-1}aNs_o(g[o])[o]=k_1g^{-1}aNg[o]=kNk^{-1}[y]=kNs_o(k^{-1}[y])[o].\]
		Since $s_o$ takes values in $AN$ and writing the $NAK$ decomposition of $g^{-1}$, there exist $a',a''\in A$ such that 
		\[k_1\kappa_{o}(g^{-1})a'N[o]=ka''N[o].\]
		Hence, by \eqref{hatnohat} we have that $k_1\kappa_{o}(g^{-1})M=kM$, that is the claim~\eqref{claim0}. Therefore, for some $m\in M$ the right-hand side of~\eqref{eq:dimii} is
\begin{align*}
haNg[o]&=gkm\kappa_o(g^{-1})^{-1}g^{-1}aNg[o]\\
&=gkmaN\left(\kappa_o(g^{-1})^{-1}g^{-1}\right)g[o]\\
&=gkmaN\kappa_o(g^{-1})^{-1}[o]\\
&=gkmaN[o]=gkaN[o]
\end{align*}
where in the second line we have used that $\kappa_o(g^{-1})^{-1}g^{-1}\in AN$ and then~\eqref{saN}. 
Summarizing, we have shown that
\[
gkNk^{-1}s_o(y)[o]=gkaN[o].
\]
By taking $e\in N$ on the left, there must be $n\in N$ such that $s_o(y)[o]=kan[o]$, so that
$(kan)^{-1}s_o(y)\in K$, whence $k^{-1}s_o(y)\in Kan$. This shows that 
\[a=\exp(A_o(k^{-1}s_o(y)))=\exp(A_o(y,kM)),\]
where the second equality follows by item (i). This concludes (ii) in the case $x=o$.
%, because by \eqref{LHS} and (i)  of this lemma
%\[
%A_{g[o]}(g[y]\,,g\langle kM\rangle )=A_o(k^{-1}s_o(y))=A_o(s_o(y)[o],kM)=A_o(y,kM)\;.
%\]
%		Hence $h=gk_1\kappa_o(g^{-1})g^{-1}$. So, by $g(k)\in gk\alpha N$ for some $\alpha\in A$, we can rewrite \eqref{eq:dimii} in the following way
%		\[gkNk^{-1}s_o(y)[o]=gk_1\kappa_o(g^{-1})g^{-1}aNg\kappa_o(g^{-1})[o]=gk_1aN[o],\]
%		where we also used $g\kappa_o(g^{-1})\in NA$. So $kNk^{-1}s_o(y)[o]=k_1aN[o]$ means that 
%		$a=\exp(A_o(k^{-1}s_o(y)))$ and this proves (ii) in the case $x=o$.\\
		The general case follows by the latter. Indeed, by applying it with $s_o(x)$ and $gs_o(x)$, respectively in the first and the second equality, we obtain
		\[A_x(y,kM)=A_o(s_o(x)^{-1}[y],s_o(x)^{-1}\langle kM\rangle )=A_{g[x]}(g[y],g\langle kM\rangle)\;.\]

		%	We recall that $s(g[x_0])=g\kappa(g)^{-1}\in AN$ and so $h=g\kappa(g)^{-1}k_1\kappa(g)g^{-1}$ for some $k_1\in K$. Furthermore, $g(k)=gk\exp(A_{x_0}(g^{-1}[x],kM))n$ for some $n\in N$. Hence
		%	$$gkNk^{-1}s(x)[x_0]=g\kappa(g)^{-1}k_1\kappa(g)g^{-1}aM_{g[x_0]}Ng[x_0]=g\kappa(g)^{-1}k_1aM_{g[x_0]}N[x_0].$$	
		%	So if $k_1=\kappa(g)k$,
		%	$Nk^{-1}s(x)[x_0]=aM_{g[x_0]}N[x_0]$
		%	and so $$a=\exp(A_{x_0}(k^{-1}s(x)))=\exp(A_{x_0}(x.kM)).$$

\item[(iii)] For simplicity we start by proving it for $x=o$, the general case follows. Let $y,z\in X$ and $kM\in K/M$. By the definition of $s_z$, we have that $s_z(o)^{-1}=s_o(z)$ and $K=s_z(o)K_zs_z(o)^{-1}$. Observe that, by the $K_zAN$ Iwasawa decomposition of $k$
\[s_z(o)k\in s_z(o)\kappa_{z}(k)AN=s_z(o)\kappa_{z}(k)s_z(o)^{-1}AN,\]
and then
\[\kappa_o(s_z(o)k)=s_z(o)\kappa_{z}(k)s_z(o)^{-1}.\]
	%	\[s_z(o)k\in s_z(o)\kappa_{z}(k)s_z(o)^{-1}\exp(H_o(s_z(o)k))N\;,\] 
%		since $s_z(o)\kappa_{z}(k)s_z(o)^{-1}\in K$.
		Furthermore, $s_y(o)k\in K\exp(H_o(s_y(o)k))N$, so that
		\begin{equation}\label{itemiii}
		s_z(o)kk^{-1}s_y(o)^{-1}\in s_z(o)\kappa_{z}(k)s_z(o)^{-1}N\exp(H_o(s_z(o)k)-H_o(s_y(o)k))K.
		\end{equation}
		Now, observe that by \eqref{eq:fondAH} and (i) it is possible to rewrite
		\begin{align*}
		H_o(s_z(o)k)-H_o(s_y(o)k)&=A_o(k^{-1}s_y(o)^{-1})- A_o(k^{-1}s_z(o)^{-1})\\
		&=A_o(s_y(o)^{-1}[o],kM)- A_o(s_z(o)^{-1}[o],kM)\\
		&=A_o(y,kM)- A_o(z,kM).
		\end{align*}
		Hence, \eqref{itemiii} becomes
		\[s_z(o)s_y(o)^{-1}\in s_z(o)\kappa_{z}(k)s_z(o)^{-1}N\exp(A_o(y,kM)- A_o(z,kM))K\;,\]
		and by conjugating by $s_{z}(o)^{-1}\in AN$
		\begin{align*}
		s_y(o)^{-1}s_z(o)\in& \kappa_{z}(k)s_z(o)^{-1}N\exp(A_o(y,kM)- A_o(z,kM))Ks_z(o)\\
		=&\kappa_{z}(k)N\exp(A_o(y,kM)- A_o(z,kM))s_z(o)^{-1}Ks_z(o)\\
		=&\kappa_{z}(k)N\exp(A_o(y,kM)- A_o(z,kM))K_z,
		\end{align*}
		where in the first equality we use \eqref{saN}.
		Finally, we observe that $s_y(o)^{-1}s_z(o)=s_o(y)s_z(o)=s_z(y)$ and then
		\[\kappa_{z}(k)^{-1}s_z(y)\in N\exp(A_o(y,kM)- A_o(z,kM))K_z.\]
		Therefore, by item (i) of Lemma \ref{kappas} and item (i) above
		\[A_o(y,kM)- A_o(z,kM)=A_z(\kappa_{z}(k)^{-1}s_z(y))=A_z(y,kM)\;.\]
		This proves the case $x=o$. The general case trivially follows:
		\begin{align*}
		A_x(z,kM)+A_{z}(y,kM)&=A_o(z,kM)-A_{o}(x,kM)+A_o(y,kM)-A_{o}(z,kM)\\
		&=A_x(y,kM).
		\end{align*}
This finishes the proof of the lemma.

	\end{enumerate}
\end{proof}
Let $x\in X$. By Definition~\ref{CD}, for every $(kM,a)\in K/M\times A$ and $z\in X$
\begin{equation}\label{charcxi}
	z\in\Psi_x(kM,a)\quad\Longleftrightarrow\quad A_{x}(z,kM)=\log a.
\end{equation}
%By  \cite[Theorem 1.1, Chap.II]{gass},  $G$ maps horocycles to horocycles. We denote by
%\[
%g.\widehat{\xi}:=\{g[y]\in X: y\in \widehat{\xi}\}
%\]
Then, by \eqref{charcxi} together with \eqref{ginvA} it follows that 
\begin{align*}
	z\in g.\Psi_x(kM,a)\quad&\Longleftrightarrow\quad g^{-1}[z]\in \Psi_x(kM,a)\\
	&\Longleftrightarrow\quad \log a=A_{x}(g^{-1}[z],kM)\\
	&\Longleftrightarrow\quad \log a=A_{g[x]}(z,g\langle kM\rangle)\\
	&\Longleftrightarrow\quad z\in \Psi_{g[x]}(g\langle kM\rangle,a).
\end{align*}
So that, 
\begin{equation}
\label{eq:actiongxiparametrized}
g.\Psi_x(kM,a)=\Psi_{g[x]}(g\langle kM\rangle,a).
\end{equation}
Furthermore, if $y\in X$, then by \eqref{charcxi}  and \eqref{cocycl} we have that 
\begin{align*}
	z\in\Psi_x(kM,a)\quad&\Longleftrightarrow\quad \log a=A_{x}(z,kM)\\
	&\Longleftrightarrow\quad \log a=A_{x}(y,kM)+A_{y}(z,kM)\\
	&\Longleftrightarrow\quad \log (a\exp(-A_{x}(y,kM)))=A_{y}(z,kM)\\
	&\Longleftrightarrow\quad z\in\Psi_y(kM,a\exp(A_{y}(x,kM))),
\end{align*}
where in the last equivalence we use the equality $A_{y}(x,kM)=-A_{x}(y,kM)$, which follows immediately by \eqref{cocycl}.
Hence, we have
\begin{equation}\label{xirefere}
	(\Psi_y^{-1}\circ\Psi_x)(kM,a)=(kM,a\exp(A_y(x,kM))).
\end{equation}

%Let $\varphi$ be a function defined on $\Xi$. Thanks to the new set of parametrizations, it is possible to write  $\varphi$ as a function on $B\times A$ in several different ways. For fixed $x\in X$, we denote by $\varphi_x$ the function defined on $B\times A$ by
%\[\varphi_x(b,a):=\varphi( {\xi}_x(b,a)).\]
%%By  \cite[Theorem 1.1, Chap.II]{gass},  $G$ maps horocycles to horocycles. We denote by
%%\[
%%g.\widehat{\xi}:=\{g[y]\in X: y\in \widehat{\xi} \}
%%\]
%%the horocycle obtained by the action of $g\in G$ on $\widehat{\xi}\in\Xi$. As just established, for $x\in X$ and $\widehat{\xi}\in\Xi$,  there exist unique $b\in B$ and $a\in A$ such that $\widehat{\xi}=\widehat{\xi}_x(b,a)$. A natural question is how to read the action  directly on $x$, $b$ and $a$. By \eqref{ginvA}, for every $y\in\widehat{\xi}$,
%%\[
%%\log a=A_{x}(y,b)=A_{g[x]}(g[y],g\langle b\rangle ),
%%\]
%%so that the action can be rewritten as 
%%\begin{equation}\label{actgxi}
%%g.\widehat{\xi}_x(b,a)=\widehat{\xi}_{g[x]}(g\langle b\rangle,a).
%%\end{equation}
%This is reflected  on the  functions  defined on $\Xi$, in the sense that
%\begin{equation}\label{actgxi}
%\varphi_x(g^{-1}\!.\, {\xi})=\varphi_{g^{-1}[x]}(g^{-1}\langle b\rangle,a).
%\end{equation}

%%%%%%%%%%%%%%%%%%%%%%%%%%%%%%%%%%%%%%%%

{\bf Positive definite symmetric matrices}.
We recall the example of the positive definite symmetric matrices in order to explicitly describe the horocycles.
Recall that the semisimple group associated to the symmetric space of the positive definite symmetric matrices is $G=\Sldr$. As we have already seen, the Iwasawa decomposition of $G$ is formed by $K=\Sod$, the subgroup $A$ of diagonal matrices with positive entries on the diagonal and the subgroup $N$ of the unit upper triangular matrix. Hence the principal horocycle is
\[ {\xi}_0=N[{\rm I}_d]=\{n\t n:n\in N \}\;.\]
Let $a=\diag(e^{a_1},\dots,e^{a_d})\in A$, then the horocycle obtained as the $N$-orbit of $aK\in\SPdr$ is the subset of $\SPdr$ of matrices of the form
\[
\Bigl\{ e^{a_i+a_j}\sum_{k=\max(i,j)}^d n_{i,k}n_{j,k} \Bigr\}_{i,j},
\]
for every choice of $d(d-1)/2$ values $n_{i,j}\in\R$ with $j>i$, where $n_{i,i}=1$. The subgroup $\overline{N}=\Theta(N)$ coincides with the lower unit triangular matrices. The $\overline{N}$-orbit of a positive definite diagonal matrix $aK$ is the set of all the symmetric positive definite matrices having $a^2$ as diagonal matrix in the usual $LD\t L$ decomposition. Furthermore, for every $a\in A$, we have $\overline{N}[a]=(N[a^{-1}])^{-1}$. It follows that the horocycle $N[a]$ is the subset of $\SPdr$ of matrices having $a^2$ as diagonal matrix in the $UD\t U$ decomposition.
\begin{center}
	\begin{figure}[h]
		\begin{center}
			\begin{tikzpicture}[scale=1.2]
			%\shade[shading=axis,shading angle=135,opacity=0.3] (0,0) rectangle (3.5,4);
			%\shade[shading=axis,shading angle=215,opacity=0.3] (-3.5,0) rectangle (0,4);
			%\shade[top color=white, bottom color=gray, opacity=0.25] (-3.5,0)--(3.5,0)--(3.5,4)--(-3.5,4)--(-3.5,0);
			\draw[->](-3.5,0)--(3.5,0);
			\draw[->](0,0)--(0,4);
			\filldraw[black] (0.5,0.9) circle (1pt) node [anchor=west]{$a[i]\!\!\!\!\!$};
			\filldraw[black] (0,1) circle (1pt) node [anchor=south east]{$i$};
			%\node[circle] at (0,0.5) [anchor=west]{$a^{-1}[i]$};
			\draw [dashed](1,0) arc(0:180:1);
			\draw (1,0.58) circle (0.58);
			\draw (-1,1.73) circle (1.73);
			\node at (2.1,1) {$N[a[i]]$};
			\node at (-2.75,3.3) {$\overline{N}[a[i]]$};
			\end{tikzpicture}
		\end{center}
		\caption{In the special case $d=2$, the $N$-orbit and the $\overline{N}$-orbit are tangent circles. More in general, for any $w\in W$ the intersection between the $N$-orbit and the $(wNw^{-1})$-orbit of a point coincides with the point itself, see  Proposition~1.7 in Chap.~II in \cite{gass}.}
	\end{figure}
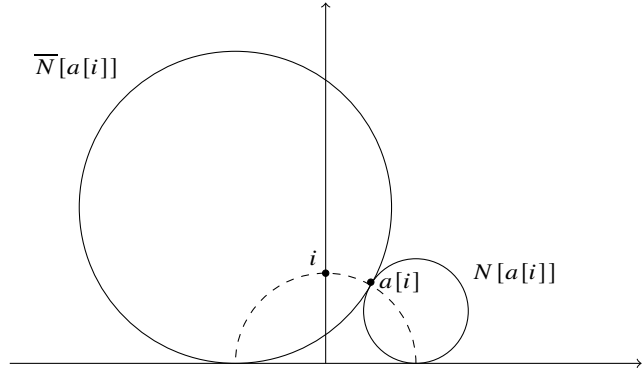
\end{center}
Let $p\in \SPdr$ and let $p=OD\t O$ be the spectral decomposition of $p$, with $O\in\Sod$ and $D$ diagonal matrix and let $k\in K$. Then we have
\[k[p]=kp\t k=kOD\t O\t k, \]
and since $kO\in\Sod$ then $k[p]$ has the same eigenvalues of $p$. In fact, the $K$-orbit of $p\in\SPdr$ is the subset of all the matrices in $\SPdr$ with the same eigenvalues of $p$ and if $a\in A$ the matrix $k[a]$ has the columns of $k$ as eigenvectors. Furthermore in each $K$-orbit there exists a diagonal matrix with  entries  ordered decreasingly on the diagonal, that is a matrix that lies on $A_+[{\rm I}_d]$.

Finally, by \eqref{hatnohat}, any horocycle $ {\xi}\in\Xi$ can be written as $ {\xi}=kaN[{\rm I}_d]$ for some $k\in K$ and $a\in A$. This is thus  the subset of $\SPdr$ of matrices having $a^2$ as diagonal matrix in the $UD\t U$ decomposition w.r.t. the $\R^d$-basis $\{ke_i\}_{i=1,\dots,d}$, where $\{e_i\}_{i=1,\dots,d}$ is the canonical basis of $\R^d$.

%%%%%%%%%%%%%%%%%%%%%%%%%%%%%%%%%%%%%%%%%%%	
\section{Analysis on Symmetric Spaces}\label{sec:analysis}
We collect in this section the analytic ingredients that come into play. Apart from the basic measures and function spaces, we introduce the Helgason-Fourier transform and the Radon transform and recall the results that we use throughout. The main references are \cite{gga}, \cite{gass}.

\subsection{Measures}
This section is devoted to the measures that will be involved in what follows. We first present the Haar measure and  then introduce the  measures on the spaces $X$, $B$ and $\Xi$.  These are necessary in order to define the  function spaces that we are interested in, among which the $L^2$-spaces that carry the regular representations. General references are \cite{folland16} for the first part, and \cite{gga} and \cite{gass} for the second.
\vskip0.2truecm
%\subsection{Integration on $X$}

\subsubsection{Haar measures and modular functions}\label{generalities}
We recall some basic definitions and results of Analysis on locally compact groups. We shall use them in the more specific context of Lie groups.  A standard reference is Chap.~2 in \cite{folland16}.

 A \emph{topological group}\index{topological!group} is a group $G$ endowed with  a topology relative to which the group operations
\[
(g,h)\mapsto gh,
\qquad
g\mapsto g^{-1}
\]
are continuous as maps $G\times G\to G$ and $G\to G$, respectively. $G$ is locally compact if every point has a compact neighborhood. We shall also assume our groups to be Hausdorff. In particular, all Lie groups are   locally compact topological groups.

A Borel measure $\mu$ on the topological space $X$,  that is, a meausure on the $\sigma$-algebra $\cB(X)$ of the Borel sets of $X$, is called a \emph{Radon measure}\index{Radon!measure} if:
\begin{enumerate}
\item[(i)] it is finite on compact sets;
\item[(ii)] it is outer regular on the Borel sets, that is for every Borel set $E$
\[
\mu(E)=\inf\{\mu(U):U\supset E,\;U\text{ open };\}
\]
\item[(iii)] it is inner regular on the open sets, that is for every open set $U$
\[
\mu(U)=\sup\{\mu(K):K\subset E,\;K\text{ compact }\}.
\]
\end{enumerate}

\begin{definition} A \emph{left Haar measure}\index{Haar!measure} on the topological group $G$ is a non zero Radon measure $\mu$ such that $\mu(xE)=\mu(E)$ for every Borel set $E\subset G$ and every $x\in G$. Similarly for \emph{right Haar measures}.
\label{haar}\end{definition}

Of course, the prototype of Haar measure is the Lebesgue measure on the additive group $\R^d$, which is invariant under left (and right) translations. 
Compactly supported continuous functions on a topological space $Y$ are denoted $C_c(Y)$.
An equivalent definition for the left Haar measure $\mu$ is to require that for every $f\in C_c(G)$ and $h\in G$,
\[
\int_Gf(hg)\de\mu(g)=\int_Gf(g)\de\mu(g)\;.
 \]
A fundamental result on Haar measures is the following theorem due to A.~Weil.

\begin{theorem}[Theorem 2.10, \cite{folland16}]\label{weil} Every locally compact  group  $G$ has a left Haar measure $\lambda$, which is essentially unique in the sense that if $\mu$ is any other left Haar measure, then there exists a positive constant $C$ such that $\mu=C\lambda$.
\label{existhaar}\end{theorem}
If we fix a left Haar measure $\mu$ on $G$, then for any $g\in G$ the measure $\mu_g$ defined by
\[
\mu_g(E)=\mu(Eg)
\]
is again a left Haar measure. Therefore there must exist a positive real number,
denoted $\Delta(g)$ such that
\[
\mu_g=\Delta(g)\mu.
\]
The function $\Delta:G\to\R_+$ is called the \emph{modular function}\index{modular function}. From now on, the choice of a left Haar measure 
$\mu$ is considered as implicitly made, and hence we write
\[
\de g:=\de\mu(g).
\]
\begin{proposition}[Proposition 2.24, \cite{folland16}]\label{modular} Let $G$ be a locally compact  group. The modular function $\Delta:G\to\R_+$ is a continuous homomorphism into the multiplicative group $\R_+$. Furthermore, for every $f\in L^1(G,\mu)$ we have
\[
\int_Gf(gh)\de g=\Delta(h)^{-1} \int_Gf(g)\de g.
\]
\end{proposition}
%\begin{theorem}[Theorem~2.10 \cite{fol2}]\label{weil}
%	Every locally compact group $G$ admits a left Haar measure. Furthermore, it is essentially unique, that is it is unique up to multiplication by a positive constant.
%\end{theorem}
%Let $G$ be a locally compact second countable group and $\mu$ a left Haar measure of its. For each $x\in G$ it is possible to define a measure $\mu_x$ as $\mu_x(E):=\mu(Ex)$ for every $E\in \mathcal{B}(G)$. It is immediate to see that $\mu_x$ is still a left Haar measure. Hence, by Theorem~\ref{weil} there exists $c_x>0$ such that $\mu_x(E)=c_x\mu(E)$ for every $E\in \mathcal{B}(G)$. This motivates an important definition: we denote $\Delta(x)=c_x$ and we called $\Delta:G\rightarrow(0,+\infty)$ the modular function of $G$. Such function can be interpreted as a way to control right multiplication on a group using a left invariant measure, as can be seen in the following proposition.
%\begin{proposition}[Proposition~2.24 \cite{fol2}]
%	The function $\Delta$ is a continuous homomorphism from $G$ to the group multiplicative group $(0,+\infty)$ and for every $f\in L^1(G,\mu)$ and $y\in G$,
%	\[\int_G f(xy)\de\mu(x)=\Delta(y^{-1})\int_G f(x)\de\mu(x)\;.\]
%\end{proposition}
A group for which every left Haar measure is also a right Haar measure, hence for which $\Delta\equiv1$, is called unimodular. Large classes of groups are \emph{unimodular}\index{unimodular!group}, such as the Abelian, compact, nilpotent, semisimple and reductive groups. Many solvable groups, however, are not. Prototypical examples of non unimodular groups are the Iwasawa $NA$ groups, such as the affine ``$ax+b$'' group.
A practical recipe for the computation of modular functions is given by the  following proposition.
\begin{proposition}[Proposition 2.30, \cite{folland16}]\label{modularLIE}
If $G$ is a connected Lie group and $\Ad$ denotes the
adjoint action of $G$ on its Lie algebra, then  $\Delta(g)=\det(\Ad(g^{-1}))$.
\end{proposition}

%A group $G$ is called unimodular if $\Delta\equiv 1$. Clearly, Abelian groups are unimodular. It is easy to see that each compact group $C$ is unimodular too, indeed since $\Delta$ is a continuous homomorphism $\Delta(C)\subseteq(0,+\infty)$ is a compact subgroup containing $1$ and so it must be $\{1\}$. By the same reason, the modular function of a group restricted to a compact subgroup is $\equiv 1$. They are unimodular also nilpotent, semisimple, and inductive groups, to name a few of examples.

The basic spaces $X$ and $\Xi$ in which we are interested are homogeneous spaces of the same group $G$. 
From the point of view of Analysis, the natural  question arises whether the homogeneous space $G/H$ admits a $G$-invariant Radon measure or not. The answer to this question is contained in Theorem~\ref{quotmeasure} below, which relates integration on $G$ to an iterated integral, first on $H$ and then on $G/H$. These formulae are achieved by means of the natural projection operator  $P:C_c(G)\rightarrow C_c(G/H)$, also known as Weil's mean opearator, defined by
\[
Pf(gH)=\int_H f(gh)\de h,
\]
which is well defined by the left invariance of $\de h$, the Haar measure on $H$. Furthermore, it is possible to see that $P$ is continuous and surjective. We are now in a position to state this classical result, also known as the Weil's decomposition theorem. Here $\Delta_G$ and $\Delta_H$ are the modular functions of $G$ and $H$, respectively.

\begin{theorem}[Theorem~2.51, \cite{folland16}]\label{quotmeasure}
Let $G$ be a locally compact group and $H$ a closed subgroup. There is a $G$-invariant Radon measure $\mu$ on $G/H$ if and only if $\Delta_G|_H=\Delta_H$. In this case, $\mu$ is unique up to a constant factor, and if the factor is suitably chosen then
	\[\int_Gf(g)\de g=\int_{G/H}Pf(gH)\de\mu(gH)=\int_{G/H}\int_H f(gh)\de h\de\mu(gH)\;, \]
	for every $f\in C_c(G)$.
\end{theorem}
Hence, there always exists a $G$-invariant Radon measure on $G/H$ whenever $H$ is compact, since $\Delta_G|_H=\Delta_H\equiv 1$. Indeed, the image of $H$ under both modular functions is a compact subgroup of the multiplicative group of positive reals, namely $\{1\}$. 

Although  many homogeneous spaces do not admit invariant measures (for example $\R$ as a homogeneous space of the ``$ax+b$'' group), all of them admit strongly quasi-invariant measures. If $\mu$ is a measure on $X=G/H$ and we write $\mu^g(E)=\mu(gE)$ for $E\in\cB(X)$, we say that $\mu$ is a \emph{quasi-invariant measure}\index{quasi-invariant!measure} if all the $\mu^g$ are equivalent, that is, mutually absolutely continuous. We say that $\mu$ is \emph{strongly quasi-invariant}\index{strongly quasi-invariant!measure} if there exists a continuous function $\lambda:G\times G/H\to(0,+\infty)$ such that 
\[
\de\mu^g(x)=\lambda(g,x)\de\mu(x),
\qquad
x\in X,g\in G.
\]
In other words, the requirement is that the Radon-Nikodym derivative $(\de\mu^g/\de\mu)(x)$ is jointly continuous in $g$ and $x$. As mentioned, all homogeneous spaces admit strongly quasi-invariant measures (see Proposition~2.56 and Theorem~2.58 in \cite{folland16}).

\subsubsection{Measures on semisimple Lie groups of the noncompact type}
Let $G$ a semisimple Lie group. By Theorem~\ref{weil}, there exists a (left) Haar measure on $G$, unique up to multiplication by a positive constant. We recall that by Theorem~\ref{iwa} there exist  subgroups $K$, $A$, and $N$ of $G$ such that $G=KAN=NAK$. Since each subgroup carries a Haar measure, the natural question arises whether it is possible to write the Haar measure of $G$ using the Haar measures of the three subgroups
involved, which are all, individually, unimodular.

Since $K$ is compact, we normalize its Haar measure in such a way that the total measure is 1. The Haar measure on $A$ is obtained by starting from the (positive) measure that any Riemannian manifold inherits from its metric, see e.g.  Chap.~ I in~\cite{gga}. The invariant metric is obtained by taking the restriction to $\ga\times\ga$ of the Killing form, which is positive definite on $\gp\times\gp\supset\ga\times\ga$, whereby $\ga$ is identified  with the tangent space to $A$ at the identity. The standard normalization is to multiply the Riemannian measure  by $(2\pi)^{-\ell/2}$, where $\ell=\dim A$. As for $N$, we normalize its Haar measure $\de n$ so that
\[
\int_{\overline{N}} e^{-2\rho(H(\overline{n}))} \de \overline{n}=1,
\]
where $\overline{N}=\Theta(N)$ and $\de \overline{n}$ is the pushforward of $\de n$ under $\Theta$.
The convergence of the above integral is no trivial matter, and is discussed in detail in~\cite{gga}.

%%%%%%%%%%%%%%%%%%%%%%%%%%%%%%%%%%%%%%%%%%%%%%%%
\begin{proposition}[Proposition 5.1, Chap. I, \cite{gga}]\label{haars} Let $\de k$, $\de a$ and $\de n$ be left-invariant Haar measures on $K$, $A$ and $N$, respectively. Then the left Haar measure $\de g$ on $G$ can be normalized  so that 
	\begin{align*}
		\int_Gf(g)\de g&=\int_{K\times A\times N}f(kan)e^{2\rho\log a}\de k\de a\de n\\
		&=\int_{N\times A\times K}f(nak)e^{-2\rho(\log a)}\de n\de a\de k\\
		&=\int_{A\times N\times K}f(ank)\de a\de n\de k
	\end{align*}
	for every $f\in C_c(G)$.
\end{proposition}

The case of the group $AN$ deserves a separate comment. We recall by Sect.~\ref{sec:prelim} that $AN$ is in fact a semidirect product since $A$ acts on $N$  by conjugation. Furthermore, for any $H\in\ga$ and any root vector $X_\alpha\in\gg_\alpha$  it holds
\[
\Ad(\exp H)(X_\alpha)={\E}^{\ad H}(X_\alpha)=\sum_0^{\infty}\frac{(\ad H)^k}{k!}X_\alpha
={\E}^{\alpha(H)}X_\alpha.
\]
It follows that upon choosing a basis of $m_\alpha$ root vectors for each positive root $\alpha$ it is 
\[
\det\Ad(\exp H)|_{\gn}=\prod_{\alpha>0}{\E}^{m_\alpha\alpha(H)}
\]
or,  using~\eqref{rho},
\[
\det\Ad a|_{\gn}={\E}^{2\rho(\log{a})}.
\]
Proposition~\ref{modularLIE} now entails that the modular function of the $AN$ Iwasawa group is
\begin{equation}\label{modAN}
\Delta(na)={\E}^{-2\rho(\log{a})}.
\end{equation}
Indeed, in the computation of $\det\Ad(na)$ on $\gn+\ga$, all is relevant is the action of $\Ad a$ on $\gn$ because the action of $\Ad a$ is unimodular on $\ga$ since  $A$ is Abelian,  the action of
$\Ad n$ is unimodular on $\gn$ because  $N$ is nilpotent and that of $\Ad n$ on $\ga$  is again unimodular
because its projection on $\ga$ is the identity (see also Cor.~5.2 in Chap.~I in \cite{gga}).
%%%%%%%%%%%%%%%%%%%%%%%%%%%%%%%%%%%%%%%%%%%%%%%%
\subsubsection{Measures on $X$} 

In order to do Analysis on the symmetric space $X$ it is important to introduce some basic functions spaces and differential operators. The reader is referred to  Chap.~II in~\cite{gga}.
 %pagina 236

A quick way to introduce differential operators on $X$ is to say that $D$ is such an operator if it is a linear mapping of $C_c^{\infty}(X)$ that decreases supports. Such operators have local nature, in the sense that  it is possible to find for any  coordinate patch $(\cU,\phi)$ in $X$ and any open set $\cW$ with compact closure in $\cU$ a finite number of smooth functions $a_\alpha$ on $\cW$ such that
\[
Df=\sum_\alpha a_\alpha(D^\alpha(f\circ\phi^{-1}))\circ\phi
\]
for any $f\in C^{\infty}(\cW)$, where 
\[
D^\alpha=\frac{\partial^{|\alpha|}}{\partial x_1^{\alpha_1}\partial x_2^{\alpha_2}\dots\partial x_d^{\alpha_d}}
\]
is the standard partial derivative operator in $\R^d$ associated with the multi-index $\alpha\in\Z_+^d$. Because of this local nature, it is then possible to extend any differential operator $D$ to $C^{\infty}(X)$.

 On any differentiable manifold, hence on a symmetric space $X$, two are the most relevant spaces to consider if  distribution theory is among the desirable targets. These are  the space of smooth  complex valued functions $\mathcal{E}(X)$ on $X$ and the space $\mathcal{D}(X)$ of  smooth  complex valued functions
 with compact support on $X$. When this notation, due to Schwartz, is adopted, it is meant that these vector spaces are endowed with  suitable topologies, see Chap.~II in~\cite{gga} for the details. We stress that in our analysis the topologies on $\mathcal{E}(X)$ and $\mathcal{D}(X)$ do not enter into play. 
 
%When these topologies are not implicitely assumed, we shall simply write $C^{\infty}(X)$ and $C_c^{\infty}(X)$. 

Now, our purpose is to determine an explicit $G$-invariant measure on the symmetric space $X=G/K$, whose existence is guaranteed by the fact that $K$ is compact (see the comment after Theorem~\ref{quotmeasure}).  Recall that, by Proposition~\ref{haars},  if $g=nak$, then the Haar measure of $G$ can be normalized so that 
\begin{equation*}
	\de g=e^{-2\rho(\log a)}\de n\de a\de k,
\end{equation*}
where $\de k$, $\de a$, and $\de n$ are the Haar measures on $K$, $A$ and $N$ that have been fixed in the previous paragraph. 

We endow $X$  with the $G$-invariant  measure $\de x$ obtained as the pushforward of  $\de g$ under the canonical projection $G\to G/K$. Thus, for any smooth compactly supported function $f\in\mathcal{D}(X)$ 
\[
\int_X f(x)\de x=\int_G f(g[o])\de g=\int_{NA} f(na[o])e^{-2\rho(\log a)}\de n\de a.
\]
We henceforth denote by $L^2(X)$ the Lebsegue space of square integrable (equivalence classes of)  functions with respect to this measure.
The \emph{quasi-regular representation}\index{quasi-regular!representation} $\pi$ of $G$ on $L^2(X)$ is then defined in the usual way, namely
\[
\pi(g)f(x):=f(g^{-1}[x]),\qquad f\in L^2(X),\,g\in G.
\]
It is a unitary non-irreducible representation. Actually, it is possible to construct a family of Hilbert spaces in which $L^2(X)$ can be decomposed as a direct integral, whereby  the restriction of $\pi$ to each of them is irreducible. These are the spherical principal series representations, discussed in  Chap.~VI in \cite{gass}. It is also well known that $\pi$ is not square integrable.

\subsubsection{Measures on the Boundary} 
We shall now define positive measures on the boundary $B$ using its various possible parametrizations.
Since $K$ and $M$ are compact subgroups of $G$, there exists a probability $K$-invariant measure $\mu^{o}$ on $B=K/M$, see the comment below Theorem~\ref{quotmeasure}. The choice of this measure is such that  Weil's decomposition  holds, assuming that we normalize the Haar measure of $M$ in such a way that the total measure is $1$.  For every other choice of the reference point $x\in X$ the analogous objects $K_x$, $M_x$ and $\mu^x$ can be introduced. The relation between $\mu^o$ and $\mu^x$ can be determined explicitly.
We consider the diffeomorphism $T_x\colon K\to K_x$ defined by $k\mapsto s_o(x)ks_o(x)^{-1}$. Its restriction to $M$ is a diffeomorphism between $M$ and $M_x$. Hence, $T_x$ induces the diffeomorphism $\tilde{T}_x\colon K/M\to K_x/M_x$ defined by
\[
\tilde{T}_x(kM)=T_x(k)M_x=s_o(x)ks_o(x)^{-1}M_x=s_o(x)kMs_o(x)^{-1}\;.\] 
Let $(\tilde{T}_x)_*(\mu^o)$ be the pushfoward of the measure  $\mu^o$ under $\tilde{T}_x$. Clearly, $(\tilde{T}_x)_*(\mu^o)$ is a $K_x$-invariant probability measure on $K_x/M_x$ and therefore $\mu^x=(\tilde{T}_x)_*(\mu^o)$. As we saw  in~\eqref{identification}, $K_x/M_x$ is diffeomorphic to the boundary $K/M$ through the map induced by $\kappa_o$.  Therefore, we can consider the following $K_x$-invariant probability measure on the boundary $B=K/M$
\[\nu^x:=(\kappa_o)_*(\mu^x).\]
It is worth observing that $\nu^o=\mu^o$ and the following relation follows
\[\nu^x=(\kappa_o\circ\tilde{T}_x)_*(\nu^o)\;.\]

\begin{lemma}\label{leminv}
	The measure $\nu^o$ is $G$-quasi-invariant. Let $F\in C(K/M)$ and $g\in G$,% then
	\begin{equation}\label{nuinv}
	\int_{K/M} F(g^{-1}\langle kM\rangle )\de\nu^o(kM)=\int_{K/M} F(kM)e^{-2\rho(H_o(gk))}\de\nu^o(kM).
	\end{equation}
\end{lemma}
\begin{proof}
	By Lemma 5.19 in Chap.I in \cite{gass}, for every $H\in C(K)$ and $g\in G$,
	\begin{equation}\label{gquasinu}
	\int_{K} H(\kappa_{o}(g^{-1}k))\de k=\int_{K} H(k)e^{-2\rho(H_o(gk))}\de k.
	\end{equation} 
	A function $F\in C(K/M)$ will now be regarded as an $M$-right invariant continuous function on $K$. By  our choice of $\nu^o$, Theorem~\ref{quotmeasure} holds and hence
	\begin{align*}
		\int_{K} F(k)\de k&=\int_{K/M}\int_M F(kM\,m)\de m\de\nu^o(kM)\\
		&=\int_{K/M} F(kM)\int_M\de m\de\nu^o(kM)\\
		&=\int_{K/M} F(kM)\de\nu^o(kM),
	\end{align*}
	where we have used the normalization of the Haar measure of $M$.
	The function $k\mapsto F(g^{-1}\langle k\rangle)=F(\kappa_{o}(g^{-1}k))$ is $M$-invariant by $\kappa_{o}(g^{-1}km)=\kappa_{o}(g^{-1}k)m$. Since $m\in M$ commutes with $A$ and $N$, 
	\[
	gkm\in \kappa_o(gk)m\exp(H_o(gk))N\
	\]
	and so $k\mapsto H_o(gk)$ is $M$-invariant. It follows that  $k\mapsto F(k)e^{-2\rho(H_o(gk))}$ is also $M$-invariant. The assertion 
	follows by  applying~\eqref{gquasinu} to $F$ in place of $H$ and then rewriting the integrals over $K$ of the $M$-invariant functions as integrals over $K/M$ w.r.t. $\nu^o$  as before. 
\end{proof}

Now we investigate the relation between the different boundary measures introduced above. If $F\in C(K/M)$ and $x\in X$, then
\begin{align*}
	\int_{K/M} F(kM)\de\nu^{x}(kM)&=\int_{K/M} F(\kappa_{o}(\tilde{T}_x(kM)))\de\nu^{o}(kM)\\
	&=\int_{K/M} F(\kappa_{o}(s_o(x)k)M)\de\nu^{o}(kM)\\
	&=\int_{K/M} F(kM)e^{-2\rho(H_{o}(s_o(x)^{-1}k))}\de\nu^{o}(kM)\\
	&=\int_{K/M} F(kM)e^{2\rho(A_{o}(x,kM))}\de\nu^{o}(kM)
\end{align*}
by Lemma \ref{leminv} and then applying item (i) of Lemma~\ref{propA} together with~\eqref{eq:fondAH} , since
\[
-H_{o}(s_o(x)^{-1}k)=A_o(k^{-1}s_o(x))=A_o(s_o(x)[o],kM)=A_{o}(x,kM).
\]
By expressing the integral of a function on $K/M$ with respect to either $\nu^{x}$ or 
$\nu^{y}$ as above and then using~\eqref{cocycl} in the form
\[
A_{o}(x,kM)=A_o(y,kM)+A_{y}(x,kM),
\]
the Radon-Nikodym derivative between the measures $\nu^{x}$ and  $\nu^{y}$ is then
\begin{equation}\label{eq:radonik}
\frac{\de\nu^{x}}{\de\nu^{y}}(kM)=e^{2\rho(A_{y}(x,kM))}.
\end{equation}

Let $x\in X$, $g\in G$ and $F\in C(K/M)$. Using first~\eqref{eq:radonik} with $y=o$ and then~\eqref{nuinv}
\begin{align*}
	\int_{K/M}F(g^{-1}\langle kM\rangle )\de\nu^x(kM)&=\int_{K/M}F(g^{-1}\langle kM\rangle)e^{2\rho(A_{o}(x,kM))}\de \nu^o(kM)\\
	&=\int_{K/M}F(kM)e^{2\rho(A_{o}(x,g\langle kM\rangle))}e^{-2\rho(H_{o}(gk))}\de \nu^o(kM).
\end{align*}
 Now observe that, by \eqref{Abello} and \eqref{ginvA},
\begin{align*}
	A_{o}(x,g\langle kM\rangle)-H_{o}(gk)&=A_{g^{-1}[o]}(g^{-1}[x],kM)+A_{o}(k^{-1}g^{-1})\\
	&=A_{g^{-1}[o]}(g^{-1}[x],kM)+A_{o}(g^{-1}[o],kM)\\
	&=A_{o}(g^{-1}[x],kM),
\end{align*}
the latter equality being just~\eqref{cocycl} from Lemma~\ref{propA}.  
Hence, we obtain a sort of dual relation between the $G$-action on the boundary and that on the reference points of the boundary measures, namely
\begin{equation}\label{ginv}
\int_{K/M}F(g^{-1}\langle kM\rangle)\de\nu^x(kM)=\int_{K/M}F(kM)\de\nu^{g^{-1}[x]} (kM).
\end{equation}

\subsubsection{Measures on $\Xi$}
Finally, in order to develop the theory in which we are interested, we need to introduce a $G$-invariant measure on $\Xi$. We denote by $\sigma$ the measure on $A$ with density $e^{2\rho(\log a)}$ with respect to the Haar measure $\de a$. For every $x\in X$, we can endow $\Xi$ with the measure $\de\xi$ obtained as the pushforward  of the measure $\nu^x\otimes\sigma$ on $K/M\times A$ by means of the map $\Psi_x$, i.e.
\[
\de\xi={\Psi_x}_*(\nu^x\otimes\sigma).
\]
It turns out that $\de\xi$ is independent of the choice of $x\in X$. We denote by $L^1(\Xi)$ and $L^2(\Xi)$ the spaces of absolutely integrable functions and square-integrable functions with respect to the measure $\de\xi$, respectively. By definition, for every $F\in L^1(\Xi)$
\begin{align*}
	\int_{\Xi} F(\xi)\de\xi&=\int_{K/M\times A}(F\circ\Psi_x)(kM,a)\de(\nu^x\otimes\sigma)(kM,a)
	\\&=\int_{K/M\times A}(F\circ\Psi_x)(kM,a)e^{2\rho(\log a)}\de\nu^x(kM)\de a\;.
\end{align*}
It is easy to verify that $\de\xi$ is $G$-invariant. 
We point out that Helgason introduced this measure w.r.t. $o\in X$, see Lemma~3.1 in Chap.~II in \cite{gass}. Since in our treatment it is important to change the reference point the expression above suits our needs.

The group $G$  acts on $L^2(\Xi)$ by the quasi-regular representation $\hat{\pi}\colon G\rightarrow \mathcal{U}(\Le^2(\Xi))$ defined by
\[
\hat{\pi}(g)F(\xi):=F(g^{-1}.\xi),\qquad F\in\Le^2(\Xi),\:g\in G.
\]
Equivalently, given $x\in X$, by \eqref{eq:actiongxiparametrized}
\begin{equation}
\label{eq:actiongxiparametrizedfunction}
(\hat{\pi}(g)F)\circ \Psi_x(kM,a)=F\circ\Psi_{g^{-1}[x]}(g^{-1}\langle kM\rangle,a),
\end{equation}
for every $(kM,a)\in K/M\times A$ and $g\in G$.

We need to introduce some more notation.
We denote by $\Delta^{-\frac{1}{2}}$ the map on $K/M\times A$ defined by 
 \[
 \Delta^{-\frac{1}{2}}(kM,a)={\E}^{\,\rho(\log{a})}.
\] The reason for such notation resides in the fact that this function has the same expression of the inverse of the square root of the modular function of the $AN$ Iwasawa group, see~\eqref{modAN}. \\
Finally, for every $x\in X$, we introduce the space $L_x^2(K/M\times A)$  of square-integrable functions on $K/M\times A$ w.r.t. the measure $\nu^x\otimes{\rm d} a$. For every $F\in L^2(\Xi)$, we denote by $\Psi^*_x F$ the $(L^2(\Xi), L_x^2(K/M\times A))$-pull-back of $F$ by $\Psi_x$, that is, we introduce  the unitary operator $\Psi^*_x\colon L^2(\Xi)\to L_x^2(K/M\times A)$ given by
\[
\Psi^*_x F(kM,a)=(\Delta^{-\frac{1}{2}}\cdot(F\circ \Psi_x))(kM,a)
\]
for almost every $(kM,a)\in K/M\times A$. In order to see that $\Psi^*_x$ is  unitary, observe that for every $F\in L^2(\Xi)$ we have that
	\begin{align*}
	&\int_{K/M\times A}|\Psi^*_xF(kM,a)|^2\de\nu^x(kM)\de a\\
	&=\int_{K/M\times A}|(\Delta^{-\frac{1}{2}}\cdot(F\circ\Psi_x))(kM,a)|^2\de\nu^x(kM)\de a
	\\&=\int_{K/M\times A}|(F\circ\Psi_x)(kM,a)|^2e^{2\rho(\log a)}\de\nu^x(kM)\de a\\
	&=\int_{\Xi} |F(\xi)|^2\de\xi=\|F\|^2_{L^2(\Xi)},
	\end{align*}
so that $\Psi^*_x$ is an isometry  from $L^2(\Xi)$ into $L_x^2(K/M\times A)$. Surjectivity is also clear.

\subsection{The Helgason-Fourier Transform}%%%%%%%%%%%%%%%%%%%%%%%%%%%%%

The Helgason-Fourier transform was defined by Helgason in analogy with the Fourier transform on Euclidean spaces in polar coordinates. We briefly recall its definition and its main features. 
%We denote by $\mathcal{D}(X)$ the space of smooth and compactly supported functions on $X$ endowed with the standard Schwartz topology (see Sect.~~\ref{appendix} for details).  
\begin{definition}[\S 1, Chap.~III, \cite{gass}]
	The \textit{Helgason-Fourier transform}\index{Helgason-Fourier!transform} of $f\in \textcolor{black}{\mathcal{D}(X)}$ is the function $\mathcal{H}f:K/M\times\mathfrak{a}^*\longrightarrow\mathbb{C}$ defined by
	\[	\mathcal{H}f(kM,\lambda)=\int_ Xf(x)e^{(-i\lambda+\rho)(A_o(x,kM))}\de x.
%	\qquad(kM,\lambda)\in K/M\times\mathfrak{a}^*.
	\]
\end{definition}

As the Euclidean Fourier transform, the Helgason-Fourier transform extends to a unitary operator on $\Le^2(X)$. The Plancherel measure involves the \emph{Harish-Chandra $\mathbf{c}$ function}\index{Harish-Chandra $\mathbf{c}$ function}, a cornerstone in the analysis on symmetric spaces \cite{chandraI58}, \cite{chandraII58}. It is a meromorphic function 
 $\mathbf{c}\colon\mathfrak{a}_c^*\to\C$ defined on the complexified dual space $\mathfrak{a}_c^*$ for which various formulae are available (see e.g.~\cite{jewel}). It may thus be restricted to the real space $\mathfrak{a}^*$. 
As an example, in the case of the unit disk, if $\Re(i\lambda)>0$, then
\[
\mathbf{c}(\lambda)=\pi^{-1/2}\frac{\Gamma(\frac{1}{2}i\lambda)}{\Gamma(\frac{1}{2}(i\lambda+1))},
\]
so that
\[
|\mathbf{c}(\lambda)|^{-2}=\frac{\pi\lambda}{2}\tanh\left(\frac{\pi\lambda}{2}\right).
\]

We denote by  
$L_{o,\textbf{c}}^2(K/M\times\mathfrak{a}^*)$ the space of the functions on $K/M\times\mathfrak{a}^*$ that are square-integrable w.r.t. the measure $w^{-1}\left|\textbf{c}(\lambda)\right|^{-2}\de \nu^o\de \lambda$, where   $w$ stands for the cardinality of the Weyl group $W$. 
\begin{proposition}
 		For every $f_1, f_2\in\cD(X)$
\begin{equation}\label{eq:plancherelformula}
\int_x f_1(x)\overline{f_2(x)}{\rm d}x=\int_{\mathfrak{a}^*\times K/M}\cH f_1(kM,\lambda)\overline{\cH f_2(kM,\lambda)}\de\nu^o(kM)\frac{\de\lambda}{w|\mathbf{c}(\lambda)|^2}.
\end{equation}
\end{proposition}
The rest of the paragraph is devoted to state the Plancherel theorem for the Helgason-Fourier transform. 
\vskip0.2truecm
{\bf Property $\sharp$.} We say that a function $F\in L_{o,\textbf{c}}^2(K/M\times\mathfrak{a}^*)$ satisfies Property $\sharp$ if for every $x\in X$ the function 
\begin{equation}\label{invhf}
\mathfrak{a}^*\ni\lambda\longmapsto\int_{K/M}e^{(\rho+i\lambda)(A_o(x,kM))}F(kM,\lambda)\de\nu^o(kM)
\end{equation}
is $W$-invariant almost everywhere (see the comments after~\eqref{WonSigma} for the $W$-action on $\ga^*$). 
\vskip0.2truecm
We denote by $L_{o,\textbf{c}}^2(K/M\times\mathfrak{a}^*)^\sharp$ the space of functions $F$ in $L_{o,\textbf{c}}^2(K/M\times\mathfrak{a}^*)$ satisfying Property $\sharp$.
We observe that the integral in \eqref{invhf} is absolutely convergent for almost every $\lambda\in\mathfrak{a}^*$. By Fubini theorem, for every $F\in L_{o,\textbf{c}}^2(K/M\times\mathfrak{a}^*)$ we have that
\[
\|F\|_{L_{o,\textbf{c}}^2(K/M\times\mathfrak{a}^*)}^2=\int_{\mathfrak{a}^*}\int_{ K/M}|F(kM,\lambda)|^2\de\nu^o(kM)\frac{\de\lambda}{w|\mathbf{c}(\lambda)|^2}<+\infty.
\]
Thus, the function $F(\cdot,\lambda)$ is in $L^2(K/M,\nu^o)\subseteq L^1(K/M,\nu^o)$ for almost every $\lambda\in\mathfrak{a}^*$ and, since $\rho(A_o(x,\cdot))$ is bounded on $K/M$, the integrability properties of $F(\cdot,\lambda)$ continue to hold for the function $e^{(\rho+i\lambda)(A_o(x,\cdot))}F(\cdot,\lambda)$. 

Every function $F\in L_{o,\textbf{c}}^2(K/M\times\mathfrak{a}^*)^\sharp$ is uniquely determined by its restriction on $K/M\times\ga_+^*$. Here $\ga_+^*$ denotes the \emph{positive Weyl chamber}\index{positive chamber!Weyl}
\[
\mathfrak{a}^*_+=\{\lambda\in\mathfrak{a}^* : A_{\lambda}\in\mathfrak{a}^+\},
\]
where $A_{\lambda}$ represents $\lambda$ via the Killing form, in the sense that $\lambda(H)=B(A_{\lambda},H)$. If we suppose that $F,G\in L_{o,\textbf{c}}^2(K/M\times\mathfrak{a}^*)^\sharp$ are such that $F_1|_{K/M\times\ga_+^*}=F_2|_{K/M\times\ga_+^*}$, then
\begin{align*}
	&\int_{K/M}e^{(\rho+is\lambda)(A_o(x,kM))}(F_1-F_2)(kM,s\lambda)\de\nu^o(kM)\\
	&=\int_{K/M}e^{(\rho+i\lambda)(A_o(x,kM))}(F_1-F_2)(kM,\lambda)\de\nu^o(kM)=0
	\end{align*}
for a. e. $\lambda\in\ga_+^*$ and for every $s\in W$. Therefore, by Lemma~5.3 in Chap.~II in \cite{gass}, we can conclude that $F_1-F_2=0$ in $L_{o,\emph{\textbf{c}}}^2\left(K/M\times\mathfrak{a}^*\right)$. 

By the Paley-Wiener theorem for the Helgason Fourier transform (Theorem~5.1 in Chap.~III in \cite{gass}), $\cH f\in L_{o,\textbf{c}}^2(K/M\times\mathfrak{a}^*)^\sharp$ for every $f\in\cD(X)$, so that $\mathcal{H}f$ is uniquely determined by its restriction on $K/M\times\ga_+^*$. We denote by $L_{o,\emph{\textbf{c}}}^2\left(K/M\times\mathfrak{a}_+^*\right)$ the space of the functions on $K/M\times\mathfrak{a}_+^*$ that are square-integrable w.r.t. the measure $\left|\textbf{c}(\lambda)\right|^{-2}\de \nu^o\de \lambda$ and the Plancherel theorem for the Helgason-Fourier transform reads:

\begin{theorem}[Theorem 1.5, Chap.~III, \cite{gass}]\label{extH}
	The restricted Helgason-Fourier transform  $f\mapsto \mathcal{H}f|_{K/M\times \ga^*_+}$ extends to a unitary operator $\mathscr{H}$ from $L^2(X)$ onto $L_{o,\emph{\textbf{c}}}^2\left(K/M\times\mathfrak{a}_+^*\right)$.
\end{theorem}
%It is worth observing that for every $f\in \cD(X)$
%\[
%\mathscr{H} f=\mathcal{H} f|_{K/M\times \ga^*_+},
%\]
By the Plancherel formula \eqref{eq:plancherelformula}, $\mathcal{H}$ is an isometry from $\mathcal{D}(X)$ into $L_{o,\emph{\textbf{c}}}^2\left(K/M\times\mathfrak{a}^*\right)$. 
% \[
% \|\cH f\|_{L_{o,\textbf{c}}^2(K/M\times\mathfrak{a}^*)}= \|\mathscr{H}f\|_{L_{o,\textbf{c}}^2(K/M\times\mathfrak{a}_+^*)}.
% \]
Furthermore, we show that, by Theorem~\ref{extH}, $\mathcal{H}(\cD(X))$ embeds densely in $L_{o,\textbf{c}}^2(K/M\times\mathfrak{a}^*)^\sharp$. Let $F\in L_{o,\emph{\textbf{c}}}^2\left(K/M\times\mathfrak{a}^*\right)^\sharp$ be such that $\langle F,\mathcal{H}f\rangle_{L_{o,\textbf{c}}^2(K/M\times\ga^*)}=0$ for every $f\in\mathcal{D}(X)$. Then, by Fubini theorem we have that
\begin{align}\label{rangeh}\nonumber
	0&=\frac{1}{w}\int_{\ga^*}\int_{K/M}F(kM,\lambda)\overline{\int_Xf(x)e^{(-i\lambda+\rho)(A_o(x,kM))}\de x}\de\nu^o(kM)\frac{\de\lambda}{|\mathbf{c}(\lambda)|^2}\\ \nonumber
	&=\frac{1}{w}\int_{\ga^*}\int_X\int_{K/M}F(kM,\lambda)e^{(i\lambda+\rho)(A_o(x,kM))}\de\nu^o(kM)\overline{f(x)}\de x\frac{\de\lambda}{|\mathbf{c}(\lambda)|^2}\\
		\nonumber &=\int_{\ga_+^*}\int_X\int_{K/M}F(kM,\lambda)e^{(i\lambda+\rho)(A_o(x,kM))}\de\nu^o(kM)\overline{f(x)}\de x\frac{\de\lambda}{|\mathbf{c}(\lambda)|^2}\\
	&=\int_{\ga_+^*}\int_{K/M}F(kM,\lambda)\overline{\mathcal{H}f(kM,\lambda)}\de\nu^o(kM)\frac{\de\lambda}{|\mathbf{c}(\lambda)|^2},
%	&=\langle F_{\bigr|K/M\times \ga^*_+},\mathcal{H}f_{\bigr|K/M\times \ga^*_+}\rangle_{L^2(K/M\times\ga_+^*)}
\end{align}
where we use that $F$ satisfies Property $\sharp$ and $|\mathbf{c}|^2$ is $W$-invariant. Hence, \eqref{rangeh} yields 
\[ 
\langle F|_{K/M\times \ga^*_+},\mathcal{H}f|_{K/M\times \ga^*_+}\rangle_{L_{o,\textbf{c}}^2(K/M\times\ga_+^*)}=\langle F|_{K/M\times \ga^*_+},\mathscr{H}f\rangle_{L_{o,\textbf{c}}^2(K/M\times\ga_+^*)}=0,
\]
for every $f\in \mathcal{D}(X)$, and Theorem~\ref{extH} implies that $F\equiv 0$ a.e. on $K/M\times \ga^*_+$. Therefore, $F=0$ in $L_{o,\emph{\textbf{c}}}^2\left(K/M\times\mathfrak{a}^*\right)$ and $\mathcal{H}(\cD(X))$ embeds densely in $L_{o,\textbf{c}}^2(K/M\times\mathfrak{a}^*)^\sharp$.

We have the following equivalent version of Theorem~\ref{extH}, that better suits our needs. 

\begin{theorem}\label{extH2}
	The Helgason-Fourier transform  $\mathcal{H}$ extends to a unitary operator $\mathscr{H}$ from $L^2(X)$ onto $L_{o,\emph{\textbf{c}}}^2\left(K/M\times\mathfrak{a}^*\right)^\sharp$. 
	\end{theorem}	
In what follows, we always consider $\mathscr{H}$ taking values in $L_{o,\emph{\textbf{c}}}^2\left(K/M\times\mathfrak{a}^*\right)^\sharp$.	
\color{black}

\subsection{The Horocyclic Radon Transform}\label{sec:radon}	%%%%%%%%%%%%%%%%%%%%%%%%%%%%%%%%%%%%%%%%%%%%%%	

In what follows we introduce the horocyclic Radon transform, we study its range, and we investigate its intertwining properties with the quasi-regular representations $\pi$ and $\hat{\pi}$ of $G$. 

Because horocycles admit (several) explicit parametrizations, we define the horocyclic Radon transform appealing directly to the basic parametrization, as clarified in the definition that follows.

\begin{definition}
	The \emph{horocyclic Radon transform}\index{horocyclic Radon!transform} $\mathcal{R}f$ of a function $f\in \mathcal{D}(X)$ is the map $\mathcal{R}f:\Xi\to\mathbb{C}$ defined by 
	\[
	(\mathcal{R}f\circ\Psi_o)(kM,a)=\int_N f(kan[o])\de n,
	\]
	for every $(kM,a)\in K/M\times A$. 
\end{definition}
If we change parametrization, and use equality \eqref{xirefere}, for any $x\in X$ we obtain the equivalent definition
\begin{align}\label{eq:radonwrtx}
\nonumber(\mathcal{R}f\circ\Psi_x)(kM,a)&=(\mathcal{R}f\circ\Psi_o)(kM,a\exp(A_o(x,kM)))\\&=\int_N f(ka\exp(A_o(x,kM))n[o])\de n.
\end{align}

%	$(\omega,n)\mapsto h_{\omega,n}^v$ and we shall write $\mathcal{R}_vf=\mathcal{R}f\circ\Psi_v$. 

%{\color{blue} Dovrebbe essere $\Delta^{-1/2}$!} 

\begin{definition}
	Let $f\in \mathcal{D}(X)$. We denote by $\mathcal{A}f$ the map $\mathcal{A}f:K/M\times A\to\mathbb{C}$ defined by 
	\[
	\mathcal{A}f(kM,a):=
	\Psi_o^*(\mathcal{R}f)(kM,a)=
	(\textcolor{black}{\Delta^{-\frac{1}{2}}}\cdot(\mathcal{R}f\circ\Psi_o))(kM,a).
	\]
\end{definition}
It is worth observing that if the function $f$ is $K$-bi-invariant, then $\mathcal{A}f$ coincides with the \emph{Abel transform}\index{Abel!transform} of $f$ introduced by Helgason in Chap.~III in \cite{gass}.

We need to introduce the Fourier transform on the Abelian group $A$. 
\begin{definition}[\S 4.2, Chap.~4, \cite{folland16}]
Let $s\in L^1(A)$. The \emph{Fourier transform}\index{Fourier!transform} $\mathcal{F}s$ of $s$ is defined on $\mathfrak{a}^*$ by	
\[	\mathcal{F} s(\lambda)=\int_As(a)e^{-i\lambda(\log a)}\de a\;.\]
\end{definition}
We now state a fundamental theorem in the $L^2$ theory of the Fourier transform. 

\begin{theorem}[Theorem 4.26, Chap.~4, \cite{folland16}]\label{fouriertransformabeliangroup}
The Fourier transform $\mathcal{F} :L^1\cap L^2(A)\to C(\mathfrak{a}^*)$ extends uniquely to a unitary operator from $L^2(A)$ onto $L^2(\mathfrak{a}^*)$. In particular, 
\begin{equation*}
\|\mathcal{F} s\|_{L^2(\mathfrak{a}^*)}=\|s\|_{L^2(A)}.
\end{equation*}
\end{theorem}

We denote by $R$ the \emph{regular representation}\index{regular!representation} of $A$ on $L^2(A)$, which is defined for every $s\in L^2(A)$ and for every $\alpha\in A$ by 
\[
R_\alpha s(a)=s(\alpha^{-1}a),\quad a\in A.
\]
Furthermore, we denote by $M$ the representation of $A$ on $L^2(\mathfrak{a}^*)$ defined for every $r\in L^2(\mathfrak{a}^*)$ and for every $\alpha\in A$ by 
\[
M_\alpha r(\lambda)=e^{-i\lambda(\log{\alpha})}r(\lambda),\quad \lambda\in\mathfrak{a}^*.
\]
\begin{proposition}[\S 7.2, Chap. 5, \cite{hol95}]
\label{prop:fouriertransformintertwines}
The Fourier transform $\mathcal{F} :L^2(A)\to L^2(\mathfrak{a}^*)$ intertwines the regular representation $R$ with the representation $M$, i.e.
\[
\mathcal{F}  R_\alpha=M_\alpha \mathcal{F} ,
\]
for every $\alpha\in A$.
\end{proposition}

We are now ready to recall the result which relates the Helgason-Fourier transform with the horocyclic Radon transform. We refer to Proposition~\ref{fst} as the Fourier Slice Theorem for the horocyclic Radon transform in analogy with the polar Radon transform, see \cite{helgason99} as a classical reference. For the reader's convenience, we include the proof.
\begin{proposition}[\S 5, Chap.~III, \cite{gass}]\label{fst}
	For every $f\in \mathcal{D}(X)$ and $kM\in K/M$, the function $a\mapsto\mathcal{A}f(kM,a)$ is in $L^1(A)$ and 
	\begin{equation}\label{eq:fst}
	(I\otimes\mathcal{F})\mathcal{A}f(kM,\lambda)= \mathcal{H}f(kM,\lambda),
	\end{equation}
	for almost every $\lambda\in\mathfrak{a}^*$.
\end{proposition}
\begin{proof}
	If $f\in \mathcal{D}(X)$
	%	, we denote by $\tilde{f}$ the $K$-right-invariant function on $G$ defined by $\tilde{f}(g)=f\circ\pi_K(g)=f(g[o])$, where $\pi_K:G\rightarrow X$ denotes the quotient projection. Since $K$ is a compact subgroup of $G$ and $supp(f)\subset\subset X$, we have $supp(f)\subset\subset G$ [citare Folland]. Let 
	and $kM\in K/M$, then by Proposition~\ref{haars} and \eqref{Abello}
	\begin{align*}
	\int_A |\mathcal{A}f(kM,a)|\de a=&\int_A e^{\rho(\log a)}|\mathcal{R}f\circ\Psi_o(kM,a)|\de a\\
	\leq&\int_A\int_N e^{\rho(\log a)}|f(kan[o])|\de n\de a\\
	=&\int_A\int_N\int_K e^{\rho(\log a)}|f(kank_1[o])|\de k_1\de n\de a\\
	=&\int_G e^{\rho(A_o(g))}|f(kg[o])|\de g\\
	=&\int_G e^{\rho(A_o(k^{-1}g))}|f(g[o])|\de g\\
	=&\int_{\text{supp}(f)} e^{\textcolor{black}{\rho(A_o(x,kM))}}|f(x)|\de x<+\infty.
	\end{align*}
	Thus, $\mathcal{A}f(kM,\cdot)$ is in $L^1(A)$ and by similar steps it is easy to prove that 
	\begin{equation*}
	(I\otimes\mathcal{F})\mathcal{A}f(kM,\lambda)=\mathcal{H}f(kM,\lambda),
	\end{equation*}
	for almost every $\lambda\in\mathfrak{a}^*$.
	%		\begin{align*}
	%		(I\otimes\mathcal{F})\mathcal{A}_vf(\omega,t)&=\sum_{n\in \mathbb{Z}}\mathcal{A}_vf(\omega,n)q^{itn}=\sum_{n\in \mathbb{Z}}q^\frac{n}{2}\mathcal{R}_vf(\omega,n)q^{itn}\\
	%		&=\sum_{n\in \mathbb{Z}}q^\frac{n}{2}\!\sum_{x\in h_{\omega,n}^v}\!f(x)q^{itn}
	%		=\sum_{x\in X}f(x)q^{(\frac{1}{2}+it)\kappa_\omega(v,x)}=\mathcal{H}_vf(\omega,t),
	%		\end{align*}
	
\end{proof}

%For our purposes, it is necessary to recall the Paley-Wiener theorem for functions on a symmetric space \cite{gass}.
%A $C^\infty$ function $F$ on $K/M\times\mathfrak{a}_\mathbb{C}^*$ holomorphic in $\lambda$ is called a holomorphic function of uniform exponential type if there exists a constant $R\geq 0$ such that for every $N\in\mathbb{N} $
%\begin{equation}\label{eq:esponentialcond}
%\vertiii{F}_{N,R}=\sup_{kM\in K/M,\:\lambda\in\mathfrak{a}_\mathbb{C}^*}e^{-R|\rm{Im}\lambda|}(1+|\lambda|)^N|F(kM,\lambda)|<+\infty.
%\end{equation}
%We denote by $\mathcal{C}_o(K/M\times\mathfrak{a}_\mathbb{C}^*)^\sharp$ the set of all the holomorphic functions of uniform exponential type for which~\eqref{invhf} is Weyl invariant.
%\begin{theorem}[\cite{gass}, Chap.III, Theorem 5.1]
%	The Helgason-Fourier transform of a function in $\mathcal{D}(X)$ has an analytic extension in $\mathcal{C}_o(K/M\times\mathfrak{a}_\mathbb{C}^*)^\sharp$. Moreover if $F\in\mathcal{C}_o(K/M\times\mathfrak{a}_\mathbb{C}^*)^\sharp$, then $F\bigr|_{K/M\times\mathfrak{a}^*}$ is the Helgason-Fourier transform of a function in $\mathcal{D}(X)$. Finally, let $F\in\mathcal{C}_o(K/M\times\mathfrak{a}_\mathbb{C}^*)^\sharp$ and $f\in\mathcal{D}(X)$ be such that $F\bigr|_{K/M\times\mathfrak{a}^*}=\mathcal{H}f$. Then $F$ satisfies \eqref{eq:esponentialcond} if and only if ${\rm supp}(f)\subseteq B(o,R)$. 
%\end{theorem}
Let $f\in\mathcal{D}(X)$. By the Paley-Wiener theorem for the Helgason Fourier transform (Theorem~5.1 in Chap.~III in \cite{gass}), $\cH f$ is rapidly decreasing in the variable $\lambda\in\mathfrak{a}^*$ uniformly over $K/M$, that is for every $n\in\mathbb{N} $
\begin{equation*}
%\label{eq:esponentialcond}
\vertiii{\mathcal{H}f}_{n}:=\sup_{kM\in K/M,\:\lambda\in\mathfrak{a}^*}(1+|\lambda|)^n|\mathcal{H}f(kM,\lambda)|<+\infty.
\end{equation*}
By Theorem~\ref{fouriertransformabeliangroup} and Proposition~\ref{fst}, we have that 
\begin{align*}
\int_{\Xi}|\mathcal{R}f(\xi)|^2\de\xi&=\int_{K/M\times A}|\Psi_o^*(\mathcal{R}f)(kM,a)|^2\de\nu^o(kM)\de a\\
&=\int_{K/M\times\mathfrak{a}^*}|(I\otimes\mathcal{F})(\Psi_o^*(\mathcal{R}f))(kM,\lambda)|^2\de\nu^o(kM)\de \lambda\\
&=\int_{K/M\times\mathfrak{a}^*}|\mathcal{H}f(kM,\lambda)|^2\de\nu^o(kM)\de\lambda\\
&=\int_{K/M\times\mathfrak{a}^*}\frac{(1+|\lambda|)^{2n}|\mathcal{H}f(kM,\lambda)|^2}{(1+|\lambda|)^{2n}}\de\nu^o(kM)\de\lambda\\
&\leq\vertiii{\mathcal{H} f}_{n}^2\int_{\mathfrak{a}^*}\frac{1}{(1+|\lambda|)^{2n}}\de\lambda<+\infty,
\end{align*}
for every $n>\dim A/2$. Therefore, $\mathcal{R}f\in L^2(\Xi)$ for every $f\in \mathcal{D}(X)$.
\color{black}

The horocyclic Radon transform intertwines the regular representations $\pi$ and $\hat{\pi}$ of $G$.
%	 This result is a direct
%	consequence of the fact that $X$ and $\Xi$ carry $G$-invariant measures ${\rm d} x$ and ${\rm d}\lambda$.
\begin{proposition}\label{interradon}
	For every $g\in G$ and $f\in \mathcal{D}(X)$
	\[\mathcal{R}(\pi(g)f)=\hat{\pi}(g)(\mathcal{R}f)\;.\]
\end{proposition}
\begin{proof}
	Let $g\in G$ and $f\in \mathcal{D}(X)$. It is sufficient to show that $\mathcal{R}(\pi(g)f)\circ\Psi_{o}=\hat{\pi}(g)(\mathcal{R}f)\circ\Psi_{o}$ on $K/M\times A$. Let $(kM,a)\in K/M\times A$. Then
	\begin{align*}
	\mathcal{R}(\pi(g)f)\circ\Psi_o(kM,a)&=\int_N \pi(g)f(kan[o])\de n\\
	&=\int_N f(g^{-1}kan[o])\de n\\
	&=\int_N f(\kappa_o(g^{-1}k)\exp(H_o(g^{-1}k))an[o])\de n,
	\end{align*}
	where we used the decomposition $g^{-1}k\in \kappa_o(g^{-1}k)\exp(H_o(g^{-1}k))N$  and the fact that $A$ normalizes $N$. Now, by \eqref{eq:fondAH}, \eqref{Abello} and \eqref{cocycl}, we have 
	\[H_o(g^{-1}k)=-A_o(k^{-1}g)=-A_o(g[o],kM)=A_{g[o]}(o,kM)\;.\]
	Finally, by $g^{-1}(kM)=\kappa_o(g^{-1}k)M$ and \eqref{eq:radonwrtx} we have that
	\begin{align*}
	\mathcal{R}(\pi(g)f)\circ\Psi_o(kM,a)&=\int_N f(\kappa_o(g^{-1}k)\exp(A_{g[o]}(o,kM))an[o])\de n\\
	&=\int_N f(\kappa_o(g^{-1}k)\exp(A_{o}(g^{-1}[o],g^{-1}\langle kM\rangle))an[o])\de n\\
	&=\mathcal{R}f\circ\Psi_{g^{-1}[o]}(g^{-1}\langle kM\rangle,a)\\
	&=(\hat{\pi}(g)\mathcal{R}f)\circ\Psi_{o}(kM,a),
	\end{align*}
	where we used the action of $G$ on $\Xi$ given in \eqref{eq:actiongxiparametrizedfunction}.
	%	\begin{equation*}\label{intercc}
	%		\mathcal{R}(\pi(g)f)(\xi)=\sum_{x\in\xi}f(g^{-1}[x])=\sum_{y\in g^{-1}.\xi}f(y)=\hat{\pi}(g)(\mathcal{R}f)(\xi),
	%	\end{equation*}
	%	for every $\xi\in\Xi$.
\end{proof}

We now introduce a closed subspace of $L^2(\Xi)$ which will play a crucial role because it is the range of the unitarization 
of the horocyclic Radon transform. By definition, for every $x\in X$ and every $F\in L^2(\Xi)$
\begin{align*}\label{eq:firstl2zfunction}
\|F\|^2_{L^2(\Xi)}&=\int_{K/M}\int_A |\Psi_x^*F(kM,a)|^2\de a{\rm d}\nu^x(kM)<+\infty.
%\nonumber&=\sum_{n\in\mathbb{Z}}\int_{\Omega} |(\Delta^{\frac{1}{2}}\cdot(F\circ\Psi_v))(\omega,n)|^2{\rm d}\nu^v(\omega)<+\infty.
\end{align*}
So that, the function $\Psi_x^*F(kM,\cdot)$ is in $L^2(A)$ for almost every $kM\in K/M$. Then, by Plancherel formula and Fubini theorem 
\begin{align*}
\|F\|^2_{L^2(\Xi)}&=\int_{K/M\times A}|\Psi_x^*F(kM,a)|^2{\rm d}\nu^x(kM)\de a\\
&=\int_{K/M\times \mathfrak{a}^*}|(I\otimes\mathcal{F})\Psi_x^*F(kM,\lambda)|^2{\rm d}\nu^x(kM)\de \lambda\\
&=\int_{\mathfrak{a}^*}\int_{K/M} |(I\otimes\mathcal{F})\Psi_x^*F(kM,\lambda)|^2{\rm d}\nu^x(kM)\de \lambda<+\infty.
\end{align*}
So that, for almost every $\lambda\in\mathfrak{a}^*$ the function $(I\otimes\mathcal{F})\Psi_x^*F(\cdot,\lambda)$ is in $L^2(K/M,\nu^x)\subseteq L^1(K/M,\nu^x)
$
and
\begin{align*}
&|\int_{K/M}(I\otimes\mathcal{F})\Psi_x^*F(kM,\lambda)\de\nu^x(kM)|\\
&\leq\int_{K/M}|(I\otimes\mathcal{F})\Psi_x^*F(kM,\lambda)|\de\nu^x(kM)<+\infty.
\end{align*}
\vskip0.2truecm
{\bf Property $\flat$.} We say that a function $F\in L^2(\Xi)$ satisfies Property $\flat$ if for every $x\in X$ the function 
\begin{equation*}
%\label{eq:equivalentfrequencycondition}
\ga^*\ni\lambda\longmapsto\int_{K/M}(I\otimes\mathcal{F})\Psi_x^*F(kM,\lambda)\de\nu^x(kM)
\end{equation*}
is $W$-invariant almost everywhere. 
\vskip0.2truecm

We denote by $L^2_{\flat}(\Xi)$ the space of functions $F\in L^2(\Xi)$ satisfying Property $\flat$. 
Notice that by the considerations above, the integral appearing in Property $\flat$ is finite for almost every $\lambda\in\mathfrak{a}^*$.
Our main results in Sect.~\ref{sec:unit} are based on the characterization of $L_\flat^2(\Xi)$ given in Proposition~\ref{prop:fundamentaloperator} below. We denote 
by $L_o^2(K/M\times\mathfrak{a}^*)$ the space of square-integrable functions on $K/M\times\mathfrak{a}^*$ w.r.t. the measure $\nu^o\otimes{\rm d} \lambda$.
%We denote by $\mathcal{A}$ the space of functions in $L^2(\Xi)$ such that 
%\[
%	 n\mapsto (\Delta^{\frac{1}{2}}\cdot(F\circ\Psi_v))(\omega,n) 
%	\]
%is in $L^1(\mathbb{Z})$ for almost every $\omega\in\Omega$. We observe that $\mathcal{A}$ is dense in $L^2(\Xi)$ since it contains the space of continuous compactly supported functions on $\Xi$.

\begin{proposition}\label{prop:fundamentaloperator}
	The operator $\Phi_o$ defined on $F\in L^2(\Xi)$ by 
	\begin{equation*}
	\Phi_oF(kM,\lambda)=(I\otimes\mathcal{F})\Psi_o^*F(kM,\lambda)
	%=(I\otimes\mathcal{F})(\Delta^\frac{1}{2}\cdot(F\circ\Psi_{x}))(kM,\lambda)
	,\qquad \text{a.e.}\, (kM,\lambda)\in K/M\times\mathfrak{a}^*,
	\end{equation*}
	is an isometry from $L^2(\Xi)$ into $L_o^2\left(K/M\times\mathfrak{a}^*\right)$. Furthermore, a function $F$ belongs to $L_\flat^2(\Xi)$
	if and only if $\Phi_o F$ satisfies Property $\sharp$.
	%	\begin{equation}\label{eq:symmetrytheta}
	%		\int_{\Omega}p_v(x,\omega)^{\frac{1}{2}-it}\Phi_v F(\omega,t)\de\nu^v(\omega)=\int_{\Omega}p_v(x,\omega)^{\frac{1}{2}+it}\Phi_v F(\omega,-t)\de\nu^v(\omega),
	%	\end{equation}
	%	for every $x\in X$ and almost every $t\in\mathbb{T}$. 
\end{proposition}
%Let $F\in\mathcal{A}$. 
%By Proposition~\ref{prop:fundamentaloperator}, $F\in L_\flat^2(\Xi)$ implies that $\Phi_oF$ satisfies~\eqref{invhf} for every $x\in X$. Conversely, if 
%we want to prove that a function $F\in L^2(\Xi)$ satisfies \eqref{eq:equivalentfrequencycondition} it is enough to verify that \eqref{invhf} holds true for $o$, and hence for every $x\in X$. This last remark 
%will prove very useful in our proofs.
\begin{proof} By Parseval identity, for every $F\in L^2(\Xi)$ we have that 
	\begin{align*}
	&\int_{ K/M\times\mathfrak{a}^*}\left|\Phi_oF(kM,\lambda)\right|^2\de \nu^{o}(kM)\de \lambda\\&=\int_{K/M}\int_{\mathfrak{a}^*}\left|(I\otimes\mathcal{F})\Psi_o^*F(kM,\lambda)\right|^2\de \lambda\de \nu^{o}(kM)\\
	&=\int_{ K/M\times A}\left|\Psi_o^*F(kM,a)\right|^2\de \nu^o(kM)\de a=\|F\|_{ L^2(\Xi)}^2,
	\end{align*}
	so that $\Phi_o$ is an isometry from $L^2(\Xi)$ into $L_o^2\left( K/M\times\mathfrak{a}^*\right)$. Now, let $F\in L^2(\Xi)$. 
	By equation \eqref{xirefere} and by the definition of the regular representation $R$ of $A$, for almost every $kM\in K/M$ and $\lambda\in\mathfrak{a}^*$ we have that
	\begin{align*}
	\Phi_oF(kM,\lambda)&=(I\otimes\mathcal{F})\Psi_o^*F(kM,\lambda)=(I\otimes\mathcal{F})(\Delta^{-\frac{1}{2}}\cdot(F\circ\Psi_o))(kM,\lambda)\\
	&=e^{\rho(A_o(x,kM))}(I\otimes\mathcal{F})(I\otimes R_{\exp(A_x(o,kM))^{-1}})(\Delta^{-\frac{1}{2}}\cdot(F\circ\Psi_x))(kM,\lambda).
	\end{align*}
	Therefore, by Proposition~\ref{prop:fouriertransformintertwines} we obtain
		\begin{align}\label{eq:foundamentalrelation}
	\nonumber\Phi_oF(kM,\lambda)&=e^{\rho(A_o(x,kM))}(I\otimes M_{\exp(A_x(o,kM))^{-1}})(I\otimes\mathcal{F})(\Delta^{-\frac{1}{2}}\cdot(F\circ\Psi_x))(kM,\lambda)\\
	\nonumber&=e^{(\rho-i\lambda)(A_o(x,kM))}(I\otimes\mathcal{F})(\Delta^{-\frac{1}{2}}\cdot(F\circ\Psi_x))(kM,\lambda)\\
	&=e^{(\rho-i\lambda)(A_o(x,kM))}(I\otimes\mathcal{F})\Psi_x^*F(kM,\lambda).
	\end{align}
	Now, for every $x\in X$ and for almost every $\lambda\in\mathfrak{a}^*$, \eqref{eq:foundamentalrelation} yields
	\begin{align}\label{eq:rangeequalities}
	\nonumber&\int_{ K/M}e^{(\rho+i\lambda)(A_o(x,kM))}\Phi_oF(kM,\lambda)\de\nu^{o}(kM)\\
	\nonumber&=\int_{ K/M}e^{(\rho+i\lambda)(A_o(x,kM))}e^{(\rho-i\lambda)(A_o(x,kM))}(I\otimes\mathcal{F})\Psi_x^*F(kM,\lambda)\de\nu^{o}(kM)\\
	\nonumber&=\int_{ K/M}(I\otimes\mathcal{F})\Psi_x^*F(kM,\lambda)e^{2\rho(A_o(x,kM))}\de\nu^{o}(kM)\\
	&=\int_{ K/M}(I\otimes\mathcal{F})\Psi_x^*F(kM,\lambda)\de\nu^{x}(kM).
	\end{align}
	Equality~\eqref{eq:rangeequalities} allows us to conclude that $F$ satisfies Property $\flat$ if and only if $\Phi_oF$ satisfies~Property $\sharp$ and this concludes our proof.
	
\end{proof}	
\begin{corollary}\label{cor:radonbemolle}
	For every $f\in \mathcal{D}(X)$, 
	\begin{equation*}
	%\label{eq:phivradonhilbert}
	\Phi_o(\mathcal{R}f)=\mathcal{H} f
	\end{equation*}
	in $L_o^2( K/M\times\mathfrak{a}^*)$ and $\mathcal{R}f\in L^2_\flat(\Xi)$.
\end{corollary}
\begin{proof}
	The proof follows immediately by Proposition~\ref{fst} and the fact that the Helgason-Fourier transform satisfies Property $\sharp$.
\end{proof}

Some comments are in order. Proposition~\ref{prop:fundamentaloperator} with Corollary~\ref{cor:radonbemolle} shows the link between the range of the Radon transform with the range of the Helgason-Fourier transform, which will play a crucial role in our main result.  
The range $\mathcal{R}(\mathcal{D}(X))$ has already been completely characterized in Chap.~IV in \cite{gass}. As it will be made clear in the next section, Property $\flat$ better suits our needs.

\section{Unitarization and Intertwining}\label{sec:unit}
In order to obtain the unitarization for the horocyclic Radon transform that we are after, we need some technicalities. Figure~\ref{fig:diagram} below might help the reader to keep track of all the spaces and operators involved in our construction.
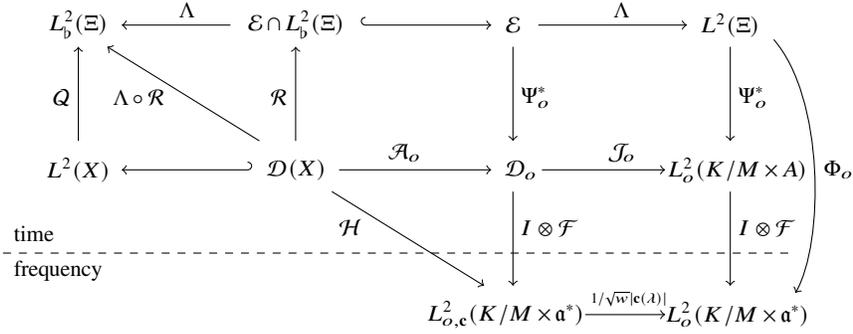
\begin{figure}[h]
	\begin{center}
		\begin{tikzpicture}[scale=0.963]
		\node at (0,0) {$L^2(\Xi)$};
		\node at (-3,0) {$\mathcal{E}$};
		\draw[<-] (-0.6,0)--(-2.65,0);
		\draw[<-] (-0.9,-2)--(-2.6,-2);
		\draw[<-] (-0.9,-4)--(-2,-4);
		\draw[->] (-3,-0.3)--(-3,-1.6);
		\node at (-2.9,-2) {$\mathcal{D}_o$};
		\draw[->] (0,-2.3)--(0,-3.6);
		\draw[->] (0,-0.3)--(0,-1.6);
		\draw[->] (-3,-2.3)--(-3,-3.6);
		\node at (-3.1,-4) {$L_{o,\mathbf{c}}^2(K/M\times\mathfrak{a}^*)$};
		\node at (-3,-1) [anchor=west]{$\Psi_o^*$};
		\node at (-3,-2.8) [anchor=west]{$I\otimes\mathcal{F}$};
		\node at (0.1,-2) {$L_{o}^2(K/M\times A)$};
		\node at (1.5,-2) {$\Phi_o$};
		\node at (0.1,-4) {$L_{o}^2(K/M\times\mathfrak{a}^*)$};
		\node at (-1.5,0) [anchor=south]{$\Lambda$};
		\node at (-1.5,-2) [anchor=south]{$\mathcal{J}_o$};
		\node at (-1.4,-4) [anchor=south]{\tiny{$1/\sqrt{w}{|\mathbf{c}(\lambda)|}$}};
		\node at (0,-2.8) [anchor=west]{$I\otimes\mathcal{F}$};
		\node at (0,-1) [anchor=west]{$\Psi_o^*$};
		\node at (-6,-1) [anchor=east]{$\mathcal{R}$};
		\node at (-9,-1) [anchor=east]{$\mathcal{Q}$};
		\node at (-9,-2) {$L^2(X)$};
		\draw[<-] (-9,-0.3)--(-9,-1.6);
		\node at (-6,-2) {$\mathcal{D}(X)$};
		\draw[<-] (-6,-0.3)--(-6,-1.6);
		\draw[left hook->] (-6.6,-2)--(-8.4,-2);
		\draw[<-] (-8.6,-0.3)--(-6.5,-1.6);
		\node at (-7.7,-1) [anchor=east]{$\Lambda\circ\mathcal{R}$};
		\node at (-9,0) {$L_\flat^2(\Xi)$};
		\node at (-6,0) {$\mathcal{E}\cap L_\flat^2(\Xi)$};
		\draw[->] (-6.9,0)--(-8.4,0);
		\node at (-7.5,0) [anchor=south]{$\Lambda$};
		\draw[right hook->] (-5.1,0)--(-3.3,0);
		\draw[<-] (-3.4,-3.6)--(-5.5,-2.3);
		\node at (-5,-2.8) [anchor=east]{$\mathcal{H}$};
		\draw[->] (-5.4,-2)--(-3.3,-2);
		\node at (-4.5,-2) [anchor=south]{$\mathcal{A}_o$};
		\draw[->] (0.6,-0.2)..controls(1.3,-1)and(1.3,-3)..(0.9,-3.7);
		\draw[dashed] (0.8,-3.15)--(-10.1,-3.15);
		\node at (-10,-3.10) [anchor=south west] {time};
		\node at (-10,-3.15) [anchor=north west] {frequency};
%		\node at (-12,0) {$L_\flat^2(\Xi)$};
%		\node at (-12,-2) {$L^2(X)$};
%		\draw[->] (-9.6,0)--(-11.4,0);
%		\draw[->] (-9.6,-2)--(-11.4,-2);
%		\draw[->] (-12,-1.6)--(-12,-0.3);
%		\node at (-10.5,0) [anchor=south]{$\pi(g)$};
%		\node at (-10.5,-2) [anchor=south]{$\hat{\pi}(g)$};
%		\node at (-12,-1) [anchor=east]{$\mathcal{Q}$};
		\end{tikzpicture}
	\end{center}
	\caption{Spaces and operators that come into play in our construction.}\label{fig:diagram}
\end{figure}

We put
\[\mathcal{D}_o=\{\varphi\in L_o^2(K/M\times A):(I\otimes\mathcal{F})\varphi\in L_{o,\emph{\textbf{c}}}^2(K/M\times\mathfrak{a}^*)\}\]
and we define the operator $\mathcal{J}_o\colon\mathcal{D}_o\subseteq L_o^2(K/M\times A)\rightarrow L_o^2(K/M\times A)$ as the Fourier multiplier
\[(I\otimes\mathcal{F})(\mathcal{J}_o\varphi)(kM,\lambda)=\frac{1}{\sqrt{w}\left|\textbf{c}(\lambda)\right|}(I\otimes\mathcal{F})\varphi(kM,\lambda),\quad \text{a.e.}\, (kM,\lambda)\in K/M\times\mathfrak{a}^*. \]
We define the set of functions 
\[
\mathcal{E}=\{F\in L^2(\Xi):\Phi_oF\in L_{o,\emph{\textbf{c}}}^2(K/M\times\mathfrak{a}^*) \}
\] 
and we consider the operator $\Lambda\colon\mathcal{E}\subseteq L^2(\Xi)\rightarrow L^2(\Xi)$ given by
\[
\Lambda F={\Psi_o^*}^{-1}\mathcal{J}_o\Psi_o^*F.
\]

As a direct consequence of the definition of $\Lambda$ and $\mathcal{J}_o$, for every $F\in\mathcal{E}$ and for almost every $( kM,\lambda)\in  K/M\times\mathfrak{a}^*$ we have (see the rightmost block in  Fig.~\ref{fig:diagram})
\begin{align}
\label{eq:philambda}
\nonumber\Phi_o(\Lambda F)( kM,\lambda)&=(I\otimes\mathcal{F})(\mathcal{J}_o\Psi_{o}^*F)( kM,\lambda)\\
\nonumber&=\frac{1}{\sqrt{w}\left|\textbf{c}(\lambda)\right|}(I\otimes\mathcal{F})(\Psi_{o}^*F)( kM,\lambda)\\
&=\frac{1}{\sqrt{w}\left|\textbf{c}(\lambda)\right|}\Phi_oF( kM,\lambda).
\end{align}
\par
The operator $\Lambda$ intertwines the regular representation $\hat{\pi}$ as shown by the next proposition.

\begin{proposition}\label{interhat}
	The subspace $\mathcal{E}$ is $\hat{\pi}$-invariant and for all $F\in\mathcal{E}$ and $g\in G$ 
	\begin{equation}\label{eq:intertwininglambda}
	\hat{\pi} (g)\Lambda F=\Lambda\hat{\pi}(g)F.
	\end{equation}
\end{proposition}
\begin{proof}
	We consider $F\in\mathcal{E}$, $g\in G$ and we prove that $\hat{\pi}(g)F\in\mathcal{E}$. %If we denote by $\psi_g$ the bijection of $ K/M\times\mathbb{Z}$ into itself given by $\psi_g( kM,n)=(g^{-1}\cdot kM,n)$, then
	By \eqref{eq:actiongxiparametrizedfunction}
	\[
	\hat{\pi} (g)F\circ\Psi_o( kM,a)=F\circ\Psi_{g^{-1}[o]}(g^{-1}\langle kM\rangle,a)
	\]
	for almost every $( kM,a)\in K/M\times A$. Therefore, we have 
	\[
	\Psi_{o}^*(\hat\pi(g)F)( kM, a)=\Psi_{g^{-1}[o]}^*F(g^{-1}\langle kM\rangle,a)
	\]
	and consequently by equation~\eqref{eq:foundamentalrelation}
	\begin{align}\label{eq:expressionpihat}
	\nonumber\Phi_o(\hat\pi(g)F)( kM, \lambda)&=(I\otimes\cF_A)(\Psi_{g^{-1}[o]}^*F)(g^{-1}\langle kM\rangle,\lambda)\\
	&=e^{(\rho-i\lambda)(A_{g^{-1}[o]}(o,g^{-1}\langle kM\rangle))}\Phi_o(F)( g^{-1}\langle kM\rangle, \lambda)
	\end{align}
	for almost every $( kM,\lambda)\in K/M\times \mathfrak{a}^*$.
	By  equations~\eqref{eq:expressionpihat}, \eqref{ginv} and \eqref{eq:radonik}
	\begin{align*}
	&\int_{ K/M\times\mathfrak{a}^*}|\Phi_o(\hat{\pi}(g)F)( kM,\lambda)|^2\frac{\de\nu^{o}( kM)\de \lambda}{w|\textbf{c}(\lambda)|^2}\\
	%&=\int_{ K/M\times\mathbb{T}}|(I\otimes\mathcal{F})(\qdot\cdot(F\circ\Psi_{g^{-1}[v]}\circ\psi_g))( kM,t)|^2\frac{c_q\de\nu^{v}( kM)\de t}{|\textbf{c}(\frac{1}{2}+it)|^2}\\
	&=\int_{\mathfrak{a}^*}\int_{ K/M}|\Phi_o(F)( g^{-1}\langle kM\rangle,\lambda)|^2e^{2\rho(A_{g^{-1}[o]}(o,g^{-1}\langle kM\rangle))}\frac{\de\nu^{o}( kM)\de \lambda}{w|\textbf{c}(\lambda)|^2}\\
	%&=\int_{\mathfrak{a}^*}\int_{ K/M}|\Phi_{g^{-1}[o]}F( kM,\lambda)|^2p_v(g^{-1}[v], kM)\frac{\de\nu^{o}( kM)\de \lambda}{w|\textbf{c}(\lambda)|^2}\\
	&=\int_{ K/M\times\mathfrak{a}^*}|\Phi_{o}F( kM,\lambda)|^2e^{2\rho(A_{g^{-1}[o]}(o,kM))}\frac{\de\nu^{g^{-1}[o]}(kM)\de \lambda}{w|\textbf{c}(\lambda)|^2}\\
		&=\int_{ K/M\times\mathfrak{a}^*}|\Phi_{o}F( kM,\lambda)|^2\frac{\de\nu^{o}( kM)\de \lambda}{w|\textbf{c}(\lambda)|^2}<+\infty
	\end{align*}
	and we conclude that $\hat{\pi}(g)F\in\mathcal{E}$. We next prove the intertwining property~\eqref{eq:intertwininglambda}. We have already observed that, by Proposition~\ref{prop:fundamentaloperator}, it is enough to prove that 
	\begin{align*}
	\Phi_o(\hat{\pi} (g)\Lambda F)=\Phi_o(\Lambda\hat{\pi}(g)F)
	\end{align*}
	for every $g\in G$ and $F\in\mathcal{E}$. By  equations~\eqref{eq:expressionpihat} and \eqref{eq:philambda},
	for almost every $(kM,\lambda)\in K/M\times\mathfrak{a}^*$, we have the chain of equalities 
	\begin{align*}
	\Phi_o(\hat{\pi} (g)\Lambda F)(kM,\lambda)
	&=e^{(\rho-i\lambda)(A_{g^{-1}[o]}(o,g^{-1}\langle kM\rangle))}\Phi_o(\Lambda F)( g^{-1}\langle kM\rangle, \lambda)\\	%=(I\otimes\mathcal{F})(\Psi_{v}^*(\hat{\pi} (g)\Lambda F))(kM,t)\\
	%&=(I\otimes\mathcal{F})(\Delta^{1/2}\cdot(\Lambda F\circ\Psi_{g^{-1}[v]}\circ\psi_g))(kM,t)\\
	%&=(I\otimes\mathcal{F})(\Psi_{g^{-1}[v]}^*(\Lambda F))(g^{-1}\cdotkM,t)\\
	%&=(I\otimes\mathcal{F})(\mathcal{J}_{g^{-1}[v]}\Psi_{g^{-1}[v]}^*F)(g^{-1}\cdotkM,t)\\
	%&=p_{g^{-1}[v]}(v,g^{-1}\cdotkM)^{\frac{1}{2}+it}(I\otimes\mathcal{F})(\mathcal{J}_v(\qdot\cdot(F\circ\Psi_v)))(g^{-1}\cdotkM,t)\\
	%&=\frac{\sqrt{c_q}}{\left|\textbf{c}(\frac{1}{2}+it)\right|}p_{g^{-1}[v]}(v,g^{-1}\cdotkM)^{\frac{1}{2}+it}(I\otimes\mathcal{F})(\Delta^{1/2}\cdot(F\circ\Psi_v))(g^{-1}\cdotkM,t)\\
	&=\frac{1}{\sqrt{w}\left|\textbf{c}(\lambda)\right|}e^{(\rho-i\lambda)(A_{g^{-1}[o]}(o,g^{-1}\langle kM\rangle))}\Phi_o(F)( g^{-1}\langle kM\rangle, \lambda)\\
	&=\frac{1}{\sqrt{w}\left|\textbf{c}(\lambda)\right|}\Phi_{o}(\hat\pi(g)F)(kM,\lambda)
	%&=\frac{\sqrt{c_q}}{\left|\textbf{c}(\frac{1}{2}+it)\right|}(I\otimes\mathcal{F})(\Delta^{1/2}\cdot(F\circ\Psi_{g^{-1}[v]}\circ\psi_g))(\omega,t)\\
	%&=\frac{\sqrt{c_q}}{\left|\textbf{c}(\frac{1}{2}+it)\right|}(I\otimes\mathcal{F})(\Psi_{v}^*(\hat{\pi} (g)F))(\omega,t)\\
	%&=(I\otimes\mathcal{F})(\mathcal{J}_v\Psi_{v}^*(\hat{\pi} (g)F))(\omega,t)
	=\Phi_o(\Lambda\hat{\pi}(g)F)(kM,\lambda),
	\end{align*}
	which proves the intertwining relation. 
\end{proof}
The next result follows directly by Proposition~\ref{prop:fundamentaloperator} and equation~\eqref{eq:philambda}.
\begin{corollary}\label{cor:lambda}
	For every $F\in \mathcal{E}$, $\Lambda F\in L^2_{\flat}(\Xi)$ if and only if $F\in L^2_{\flat}(\Xi)$.
\end{corollary}
\begin{proof}
	By Proposition~\ref{prop:fundamentaloperator}, $\Lambda F\in L^2_{\flat}(\Xi)$ if and only if $\Phi_o(\Lambda F)$ satisfies~Property $\sharp$. 
	%We compute
	%\begin{align}%\label{eq:philambda}
	%\nonumber\Phi_v(\Lambda F)(\omega,t)&=(I\otimes\mathcal{F})(\mathcal{J}_v\Psi_{v}^*F)(\omega,t)\\
	%\nonumber&=\frac{\sqrt{c_q}}{\left|\textbf{c}(\frac{1}{2}+it)\right|}(I\otimes\mathcal{F})(\Psi_{v}^*F)(\omega,t)\\
	%&=\frac{\sqrt{c_q}}{\left|\textbf{c}(\frac{1}{2}+it)\right|}\Phi_vF(\omega,t).
	%\end{align}
	By \eqref{eq:philambda} and since $\lambda\mapsto \left|\textbf{c}(\lambda)\right|$ is $W$-invariant, $\Phi_o(\Lambda F)$ satisfies~Property $\sharp$ if and only if $\Phi_o(F)$ satisfies~Property $\sharp$, which is equivalent 
	to $F\in L^2_\flat(\Xi)$. This concludes the proof. 
\end{proof}

\noindent
%We denote by $L_\flat^2(\Xi)\subseteq L^2(\Xi)$ the closed subspace of functions satisfying (ii) in Definition~\ref{radoncond}.
We are now in a position to prove our main result.
\begin{theorem}\label{thm:unitarizationtheorem}
	The composite operator $\Lambda\mathcal{R}$ extends to a unitary operator
	\[\mathcal{Q}\colon L^2(X)\longrightarrow  L_\flat^2(\Xi)\]
	 which intertwines the representations $\pi$ and $\hat{\pi}$, i.e.
	\begin{equation}\label{intertw}
	\hat{\pi}(g)\mathcal{Q}=\mathcal{Q}\pi(g),\hspace{8mm}g\in G.
	\end{equation}
\end{theorem}
Theorem~\ref{thm:unitarizationtheorem} implies that $\pi$ and the restriction $\hat{\pi}|_{L_\flat^2(\Xi)}$ of $\hat{\pi}$ to $L_\flat^2(\Xi)$ are unitarily equivalent representations. Moreover, $\hat{\pi}|_{L_\flat^2(\Xi)}$ (and then $\hat{\pi}$) is not irreducible, too. 
\begin{proof}
	We first show that $\Lambda\mathcal{R}$ extends to a unitary operator $\mathcal{Q}$ from $L^2(X)$ onto $L^2(\Xi)$. It might be useful to keep in mind see the leftmost block in  Fig.~\ref{fig:diagram}. Let $f\in \mathcal{D}(X)$, by the Fourier Slice Theorem \eqref{eq:fst}, the Plancherel formula and the definition of $\mathcal{J}_o$ and $\Lambda$, we have that
	\begin{align*}
	\|f\|_{ L^2(X)}^2&=\|\mathcal{H} f\|_{L_{o,\emph{\textbf{c}}}^2(K/M\times\mathfrak{a}^*)^\sharp}^2\\
	&=\|(I\otimes\mathcal{F})(\Psi_{o}^*(\mathcal{R}f))\|_{L_{o,\emph{\textbf{c}}}^2(K/M\times\mathfrak{a}^*)^\sharp}^2\\
	&=\int_{K/M\times\mathfrak{a}^*}|(I\otimes\mathcal{F})(\mathcal{J}_o\Psi_{o}^*(\mathcal{R}f))(kM,\lambda)|^2\de\nu^{o}(kM)\de \lambda\\
	&=\int_{K/M\times\mathfrak{a}^*}|(I\otimes\mathcal{F})(\Psi_{o}^*(\Lambda\mathcal{R}f))(kM,\lambda)|^2\de\nu^{o}(kM)\de \lambda\\
	%&=\int_{K/M\times A}|(\Delta^{-\frac{1}{2}}\cdot\mathcal{J}_o(\mathcal{A}_{o}f))(kM,n)|^2q^n\de\nu^{o}(kM)\de n\\
	&=\int_{K/M\times A}|\Psi_{o}^*(\Lambda\mathcal{R}f)(kM,a)|^2\de\nu^{o}(kM)\de a\\
	&=\|\Lambda\mathcal{R}f\|_{L^2(\Xi)}^2.
	\end{align*}
	%		&=\int_{K/M\times\mathfrak{a}^*}\left|\mathscr{H}_of(kM,t)\right|^2\deL\\
	%		&=\int_{K/M\times\mathfrak{a}^*}\left|(I\otimes\mathcal{F})((kM,n)\mapsto q^\frac{n}{2}\mathcal{R}f(h_{kM,n}^o))(kM,t)\right|^2\deL\\
	%		&=\int_{K/M\times\mathfrak{a}^*}\left|(I\otimes\mathcal{F})(\mathcal{J}((kM,n)\mapsto q^\frac{n}{2}\mathcal{R}f(h_{kM,n}^o)))(kM,t)\right|^2\de t\de\nu^o(kM)\\
	Hence, $\Lambda\mathcal{R}$ is an isometric operator from $\mathcal{D}(X)$ into $L^2(\Xi)$. Since $\mathcal{D}(X)$ is dense in $\Le^2(X)$, $\Lambda\mathcal{R}$ extends to a unique isometry from $\Le^2(X)$ onto the closure of $\mathrm{Ran}(\Lambda\mathcal{R})$ in $L^2(\Xi)$. We must show that $\Lambda\mathcal{R}$ has dense image in $L_\flat^2(\Xi)$. The inclusion $\textrm{Ran}(\Lambda\mathcal{R})\subseteq L_\flat^2(\Xi)$ follows immediately from Corollary~\ref{cor:radonbemolle} and Corollary~\ref{cor:lambda}.
	%		Since $\Sq$ is dense in $L_\flat^2(\Xi)$, it is sufficient to show that $\Lambda\mathcal{R}$ has dense image in $\Sq$ or, equivalently, that $(\mathrm{Ran}(\Lambda\mathcal{R}))^\perp=\{0\}$.
	Let $F\in L_\flat^2(\Xi)$ be such that $\langle F,\Lambda\mathcal{R}f\rangle_{L^2(\Xi)}=0$ for every $\mathcal{D}(X)$. By the Plancherel formula and the Fourier Slice Theorem \eqref{eq:fst} we have that
	\begin{align*}
	0&=\langle F,\Lambda\mathcal{R}f\rangle_{L^2(\Xi)}\\
	&=\int_{K/M\times A}(F\circ\Psi_{o})(kM, a)\overline{(\Lambda\mathcal{R}f\circ\Psi_{o})(kM, a)}e^{2\rho(\log a)}\de\nu^{o}(kM)\de a\\
	&=\int_{K/M\times A}(\Psi_{o}^*F)(kM, a)\overline{(\mathcal{J}_o\Psi_{o}^*(\mathcal{R}f))(kM, a)}\de\nu^{o}(kM)\de a\\
	&=\int_{K/M\times\mathfrak{a}^*}\Phi_o(F)(kM, \lambda)\overline{(I\otimes\mathcal{F})(\mathcal{J}_o\Psi_{o}^*(\mathcal{R}f))(kM, \lambda)}\de\nu^{o}(kM)\de \lambda\\
	&=\int_{K/M\times\mathfrak{a}^*}\Phi_o(F)(kM, \lambda)\overline{(I\otimes\mathcal{F})(\Psi_{o}^*(\mathcal{R}f))(kM, \lambda)}\frac{\de\nu^{o}(kM)\de \lambda}{\sqrt{w}|\textbf{c}(\lambda)|}\\
	&=\int_{K/M\times\mathfrak{a}^*}\sqrt{w}|\textbf{c}(\lambda)|\Phi_o(F)(kM,\lambda)\overline{\mathcal{H}_{o}f(kM, \lambda)}\frac{\de\nu^{o}(kM)\de \lambda}{w|\textbf{c}(\lambda)|^2}.
	\end{align*}
	For simplicity, we denote by $\Theta F$ the function on $K/M\times\mathfrak{a}^*$ defined as 
	\[
	\Theta F(kM,\lambda)=\sqrt{w}|\textbf{c}(\lambda)|\Phi_o(F)(kM,\lambda),\qquad \text{a.e.}\, (kM,\lambda)\in K/M\times\mathfrak{a}^*.
	\]
	Hence we have proved that $\langle\Theta F, \mathcal{H} f\rangle=0$
	for every $f\in \mathcal{D}(X)$. The next two facts follow immediately by Proposition~\ref{prop:fundamentaloperator}. Since $\Phi_o$ is an isometry from $L^2(\Xi)$ into $L_o^2\left(K/M\times\mathfrak{a}^*\right)$, the function $\Theta F$ belongs to $L_{o,\emph{\textbf{c}}}^2(K/M\times\mathfrak{a}^*)$. Further, since $F\in L_\flat^2(\Xi)$ and since $\lambda\mapsto \left|\textbf{c}(\lambda)\right|$ is $W$-invariant, then $\Theta F	\in \Le_{o,\emph{\textbf{c}}}^2(K/M\times\mathfrak{a}^*)^\sharp$.
	By Theorem \ref{extH}, $\mathcal{H}(\mathcal{D}(X))$ is dense in $\Le_{o,\emph{\textbf{c}}}^2(K/M\times\mathfrak{a}^*)^\sharp$. Hence, $\Theta F=0$ in $\Le_{o,\emph{\textbf{c}}}^2(K/M\times\mathfrak{a}^*)^\sharp$ 
	and then $\Phi_o(F)=0$ in $L_o^2\left(K/M\times\mathfrak{a}^*\right)$. Since $\Phi_o$ is an isometry from $L^2(\Xi)$ into $L_o^2\left(K/M\times\mathfrak{a}^*\right)$, then $F=0$ in $L^2(\Xi)$. 
	Therefore, $\overline{\mathrm{Ran}(\Lambda\mathcal{R})}=L_\flat^2(\Xi)$ and $\Lambda\mathcal{R}$ extends uniquely to a surjective isometry
	$$\mathcal{Q}\colon \Le^2(X)\longrightarrow L_\flat^2(\Xi).$$
	Observe that $\mathcal{Q}f=\Lambda\mathcal{R}f$ for every $f\in \mathcal{D}(X)$.
	The intertwining property \eqref{intertw} follows immediately from Proposition~\ref{interradon} and Proposition~\ref{interhat}.
\end{proof}

%\printindex

%%%%%%%%%%%%%%%%%%%%%%%%%%%%%%%%%%%%%%%%%%%%%%%%
%%%%%%%%%%%%%%%%%%%%%%%%%%%%%%%%%%%%%%%%%%%%%%%%
%%%%%%%%%%%%%%%%%%%%%%%%%%%%%%%%%%%%%%%%%%%%%%%%
%\bibliographystyle{plain}
%		\bibliography{biblio}

\end{document}